\documentclass[10pt]{article}
\usepackage[paperwidth=8.5in,paperheight=11.00in,top=1.25in, bottom=1.25in, left=1.00in, right=1.00in]{geometry}
\usepackage{microtype}

\usepackage{amsmath}
\usepackage{amssymb}
\usepackage{mathtools}
\usepackage{bm}
\usepackage{commath}
\mathtoolsset{showonlyrefs=true}
\numberwithin{equation}{section}
\allowdisplaybreaks

\usepackage{amsthm}
\theoremstyle{plain}
\newtheorem{theorem}{Theorem}
\numberwithin{theorem}{section}
\newtheorem{lemma}[theorem]{Lemma}
\newtheorem{proposition}[theorem]{Proposition}

\theoremstyle{definition}
\newtheorem{definition}[theorem]{Definition}

\newtheorem{remark}[theorem]{Remark}
\newtheorem{assumption}[theorem]{Assumption}

\usepackage{algorithm}
\usepackage{algpseudocode}
\usepackage{tikz}
\usetikzlibrary{shapes.geometric, arrows, automata, positioning}
\usepackage{adjustbox}
\usepackage{subcaption}

\usepackage[authoryear]{natbib}
\usepackage[dvipsnames]{xcolor}
\usepackage[plainpages=false, pdfpagelabels]{hyperref}
\hypersetup{colorlinks=true, citecolor=RoyalBlue, linkcolor=RubineRed, urlcolor=Turquoise}

\renewcommand\d{\partial}

\newcommand\bs[1]{\boldsymbol{#1}}
\newcommand\dd{\mathrm{d}}
\newcommand\Pb{\mathbb{P}}
\newcommand\Eb{\mathbb{E}}
\newcommand\wtilde{\widetilde}
\newcommand\what{\widehat}

\newcommand\N{\mathbb{N}}
\newcommand\R{\mathbb{R}}
\newcommand\cS{\mathcal{S}}
\newcommand\cX{\mathcal{X}}
\newcommand\cF{\mathcal{F}}

\newcommand\cV{\mathcal{V}}
\newcommand\bL{\mathbf{L}}
\newcommand\bK{\mathbf{K}}

\newcommand\bI{\mathbf{I}}
\newcommand\bE{\mathbf{E}}
\newcommand\bA{\mathbf{A}}

\renewcommand{\(}{\left(}
\renewcommand{\)}{\right)}
\renewcommand{\[}{\left[}
\renewcommand{\]}{\right]}

\begin{document}

\title{Bootstrap Percolation in Random Graphs of Unbounded Rank}

\author{
Nils Detering\thanks{
    Heinrich Heine University D\"usseldorf.
    (E-mail: \href{mailto:nils.detering@hhu.de}{nils.detering@hhu.de}).
}
\and
Jimin Lin\thanks{
    University of California, Santa Barbara.
    (E-mail: \href{mailto:jiminlin@pstat.ucsb.edu}{jiminlin@pstat.ucsb.edu}).
}
}
\date{}
\maketitle

\begin{abstract}
Bootstrap percolation in (random) graphs is a contagion dynamic among a set of vertices with certain threshold levels. The process is started by a set of initially infected vertices, and an initially uninfected vertex with threshold $k$ gets infected as soon as the number of its infected neighbors reaches $k$. This process has been studied extensively in \textit{rank one} models. These models can generate random graphs with heavy-tailed degree sequences but they are not capable of generating networks with a flexible stochastic block structure. In this paper, we treat a class of random graphs of unbounded rank that can generate flexible stochastic block structures. Our main result determines the limit in probability of the final fraction of infected vertices from the fixed point of a non-linear operator defined on a suitable function space. We propose a neural network based algorithm to calculate this fixed point efficiently. We further derive criteria based on the Fr\'echet derivative of the operator that allow one to determine whether small infections spread through the entire graph or rather stay local.
\end{abstract}

\section{Introduction}
\textit{Bootstrap percolation} models the spread of some activation or infection among a set of vertices. It has been studied on structures such as trees, (random) graphs, and lattices, among others. Common to all of these studies is the specification of a local rule according to which the infection spreads locally between neighbors, and one is then interested in understanding the process on a global level. The most classical setting is possibly that of $k$-threshold percolation, where the process starts with a set of initially infected vertices and subsequently a vertex gets infected as soon as $k$ of its neighbors are infected. Bootstrap percolation has been used to study the demagnetization in magnetic materials (\cite{chalupa_bootstrap_1979}), impulses in brain neuronal networks (\cite{tlusty_remarks_2009}, \cite{goltsev_stochastic_2010}), storage failure in computer networks (\cite{kirkpatrick_percolation_2002}), contagion in financial networks (\cite{Cont2016,Detering2016}), and the spreading of disease (\cite{HURD202175}).

Bootstrap percolation in random graphs was first investigated in \cite{balogh_bootstrap_2007} and \cite{fontes_bootstrap_2008} for the configuration model where a phase transition of the $k$-threshold percolation in a $d$-regular graph was derived. More general results for graphs with arbitrary degree distributions were obtained in \cite{baxter_bootstrap_2010} and \cite{amini_bootstrap_2010}. Bootstrap percolation in the Erd\H{o}s R\'enyi $G_{n,p}$ model was explored comprehensively in \cite{janson_bootstrap_2012}. The authors of \cite{turova_bootstrap_2015} found a narrower critical window for phase transition when a one-dimensional lattice was added to the $G_{n,p}$ graph. A variant of bootstrap percolation in the $G_{n,p}$ model that considered the synchronous and asynchronous percolation processes with inhibitory and excitatory vertices was studied recently in \cite{einarsson_bootstrap_2019}. For bootstrap percolation in the inhomogeneous Chung-Lu random graph model with a power-law degree distribution, a threshold function for the number of initially infected vertices has been derived in \cite{amini_bootstrap_2014}. The threshold function ensures that a positive fraction of vertices becomes infected. The authors of \cite{fountoulakis2018phase} determine weight sequences for which such a critical phenomenon occurs. The size of the final set of infected vertices triggered by a fixed proportion of initially infected vertices is determined in \cite{articleAmini}. In a directed Chung-Lu model with heterogeneous thresholds and initial infections, the paper \cite{meyer-brandis_bootstrap_2019} studies the size of the final set of infected vertices and, in some cases, lower bounds that do not depend on the magnitude of the initial infection. For random geometric graphs, in \cite{bradonjic_bootstrap_2014}, a critical point of the first order phase transition of bootstrap percolation is determined and in  \cite{falgas2016bootstrap} situations are studied where local infection spreads globally. The work \cite{janson_modified_2019} presents a modified non-monotone bootstrap percolation which allows infected vertices to recover. The influence of the underlying geometry of inhomogeneous geometric graphs on the percolation speed was studied in the recent work \cite{koch_bootstrap_2016}.

A majority of the random graph models mentioned above, including the Erd\H{o}s R\'enyi model, the Chung-Lu model, and certain geometric models, are \textit{rank one} random graphs. They can generate heterogeneous degrees while still being very tractable in analytical terms. They are a subset of a more general model introduced in \cite{bollobas_phase_2007}. In this general model, every vertex has a type $x\in \cS$ from some type space $\cS$, and the probability of connection between two vertices of types $x$ and $y$ is given by $\kappa (x,y) /n$, where $\kappa : \cS\times \cS\rightarrow \mathbb{R}_+$ is called the kernel function. The rank one models are those specifications where $\kappa$ is of the special form $\kappa (x,y) = \phi (x) \phi (y)$ for some function $\phi : \cS\rightarrow \mathbb{R}_+$, and variants thereof. The name originates from the fact that, in this case, for a graph with $n$ vertices and types $x_i, i\in [n]$, the matrix of connection probabilities has rank one and is given by $\Phi \cdot \Phi^T$ where $\Phi=(\phi (x_1), \dots ,\phi(x_n))^T$. Here $A^T$ denotes the transpose of the vector (matrix) of $A$. For the random graph models of rank one, it turns out that the final set of infected vertices can actually be determined by solving an one-dimensional fixed point equation (\cite{articleAmini,meyer-brandis_bootstrap_2019}). In \cite{Detering2018} a model with finite rank $K$ has been proposed that maintains the multiplicative structure of the connection probabilities in the rank one models. In this case, the resulting fixed point equations are multidimensional. In \cite{torrisi2022bootstrap} a model with rank $2$ has been considered in a setting where the number of edges is of order larger than $n$.

Despite the capacity to capture inhomogeneity, a shortcoming of the finite rank models is their inability to build a flexible block model structure. However, it is well known that in most real networks one observes a block structure where often the geographical location determines the membership in a certain block. One example is the worldwide interbank lending network. The intra-country connection between core banks and local banks can be large in some countries, but the inter-country connection between the core banks and foreign local banks is usually much smaller. Another example is social networks where the probability of connections between individuals strongly depends on the location of residence. To alleviate this neglect of the observed structural phenomena, one straightforward approach is to adopt a more flexible kernel function for generating the connection probability, which, in the matrix analogy, corresponds to using a matrix with unbounded rank for the connection probabilities.

In this paper, we therefore study a bootstrap percolation process in inhomogeneous random graphs with connection probabilities described by a general kernel function $\kappa$. It turns out that in this general case the analytic treatment becomes inherently infinite dimensional in contrast to all special cases treated earlier. This leads to new technical challenges that require a new set of approximation results. We derive a space $\mathcal{F}_b$ of real-valued functions defined on $\cS$ and a non-linear operator $\Psi_{\kappa}: \mathcal{F}_b \rightarrow \mathcal{F}_b$ that allow us to determine the limit in probability of the result of the bootstrap percolation process for large $n$. More precisely, if $\hat{f} \in \mathcal{F}_b$ is the pointwise least fixed point of the operator $\Psi_{\kappa}$ (i.e. $\Psi_{\kappa}(\hat{f})=\hat{f}$) and the function $\hat{f}$ is continuous, then, for large $n$, the final number of infected vertices is close to $n \int_{\cS} \hat{f} \dd\mu $ with high probability, where $\mu$ is a measure on $\cS$ that describes the type distribution of the vertices. In most cases, when $\kappa$ satisfies some Lipschitz condition, we can actually show that a least fixed point exists and that it is continuous. In the next step, we then find that the Fr\'echet derivative of $\Psi_{\kappa}$ at the origin (the function that is constantly equal to zero) determines whether small infections spread to a large part of the graph or stay local. We provide an algorithm that effectively determines the fixed point $\hat{f}$ and we provide an extensive numerical case study.

The content of this paper is organized as follows. In Section \ref{sec:con}, we introduce the random graph and outline our main results. In Section \ref{sec:main_proof} we explain our proof strategy and provide the thorough proofs of our main results, Theorem~\ref{thm:fp} and Theorem~\ref{thm:resilience}. In Section~\ref{case:study} we present our algorithm to determine the least fixed point of the operator $\Psi_{\kappa}$. We also perform an extensive case study that shows that our asymptotic results hold already for fairly small networks. In Appendix~\ref{sec:dis}, we study systems with only finitely many types. These systems serve as the tool in our proof of the general setting. In Appendix~\ref{sec:aux}, we collect proofs of several auxiliary results.

\section{Model and results} \label{sec:con}

{\em Random graph model:}
For each $n \in \mathbb{N}$, we consider a vertex set $[n] = \{1, 2, \dots, n\}$. Each vertex $i \in [n]$ is assigned a deterministic parameter $s_i (n)\in \cS$. This parameter is called the vertex type and it takes values in a compact metric space $\cS$. The type is a vertex characteristic that will determine its local connectivity properties. Examples of vertex types in a social network could be gender, nationality, or location of residence, to name a few. Let $\bm{s} (n)=(s_1 (n), \dots, s_n (n))$ be the vector of the types of all vertices. We shall often omit $n$ and just write $\bm{s}$ and $s_i$ for $i\in [n]$. We consider directed random graphs without self-loops. The connection probability $p_{ij} = p_{ij}(n)$ between two vertices $i, j \in [n]$ depends now on their types $s_i$ and $s_j$ and is specified by some non-negative and Borel measurable kernel function: $\kappa:\cS^2 \rightarrow \mathbb{R}$: a directed edge between vertex $i$ and $j$ ($i\neq j$) is present with probability 
$$p_{ij} = \min\{1, \kappa(s_i, s_j)/n\}.$$
Further, let the event that an edge is present be independent of the presence of all other edges. 

{\em Bootstrap percolation:}
In addition to the parameter $s_i$ we assign to each vertex $i\in [n]$ a second deterministic parameter $r_i(n) \in \mathbb{N}_0$. 
The parameter $r_i(n)$ represents a threshold level of vertex $i$, and determines in the percolation process how many neighbors of the vertex $i$ need to be infected for the vertex $i$ to become infected itself. Denote by $\bm{r} (n)=(r_1(n),\dots,r_n (n))$ the vector of thresholds. 

More precisely, let $\mathbb{D}_0$ be the set of initially infected vertices, i.e. $\mathbb{D}_0:=\{ i\in [n] \mid r_i(n)=0 \}$. The set $\mathbb{D}_0$ starts the process. Let $N(i):=\{ j \in [n] \mid E_{ji}=1\}$ denote the set of neighbors with a directed edge pointing to vertex $i$, where $E_{ji}$ is a binary indicator of an edge from $j$ to $i$. Then, in the first generation, those vertices get infected whose number of edges from vertices in $\mathbb{D}_0$ reaches or exceeds their threshold, i.e. $\mathbb{D}_1 = \{i \in [n]: |\mathbb{D}_0 \cap N(i)| \ge r_i(n)\}$.
As the infection continues to spread, the infected vertices in the $m$-th generation, where $m \in \mathbb{N}$, are given by
\begin{align}
    \mathbb{D}_m = \{i \in [n]: |\mathbb{D}_{m-1} \cap N(i)| \ge r_i(n)\}.
\end{align}
Since there are $n$ vertices in total, the spread certainly comes to an end after at most $n$ steps and $\mathbb{D}_{M} = \mathbb{D}_{M+1}=\dots $ for some $M\leq n$. We denote by $\mathbb{D}_{n}$ the final set of infected vertices. Note that given the realization of the random graph, the bootstrap percolation process is deterministic. In Figure~\ref{fig:percolation} we exemplify the bootstrap percolation process for a graph with $n=10$ vertices. In this example graph, $2$ vertices are initially infected and all remaining vertices have a threshold equal to $2$. The graph is a random sample arising from a kernel that we specify in our case study in Section~\ref{case:study}.

\begin{figure}[t]
\begin{center}
\includegraphics[clip, width=0.8 \linewidth]{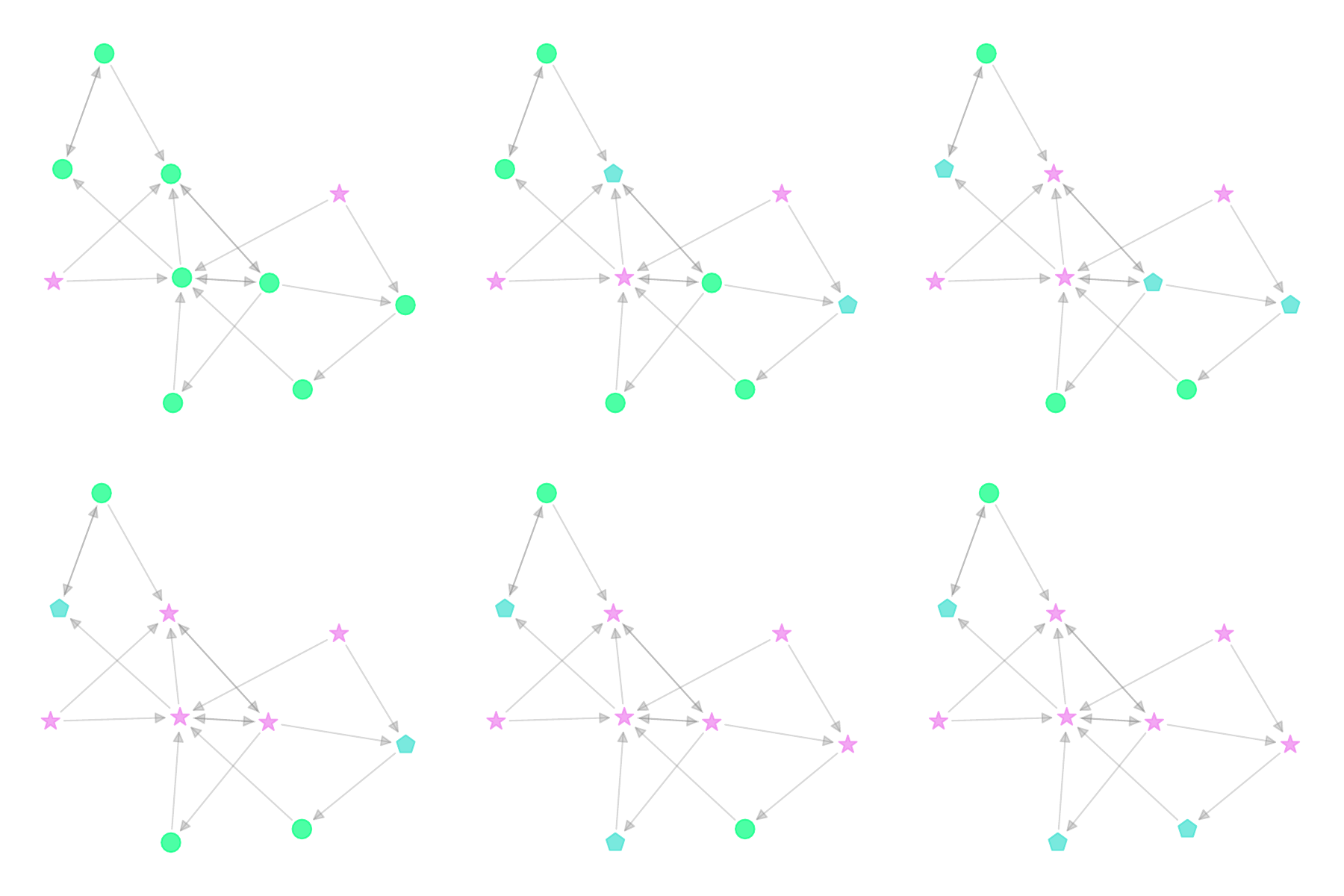}
\caption{
An example of a bootstrap percolation process on a directed graph with $n=10$ vertices, each having threshold $2$. The process evolves over six rounds before terminating, as no further infections occur. The rounds are arranged from left to right and from top to bottom. Green circular vertices have not yet received any incoming edge from an infected vertex; blue pentagonal vertices have received exactly one incoming edge from an infected vertex; and violet star vertices are infected, either initially or during the process.
}
\label{fig:percolation}
\end{center}
\end{figure}
In this paper we investigate the fraction $n^{-1}\abs{\mathbb{D}_{n}}$ of infected vertices as $n\rightarrow \infty$, in the random graph described above. This requires that we have a sequence of random graphs parameterized by $n$. Recall that for fixed $n$, all relevant information about the probabilistic structure of the skeleton of the graph is encoded in $\kappa$ and $\bm{s} (n)$ while the additional information required for the bootstrap percolation process is encoded in $\bm{r} (n)$. A sequence of random graphs, together with their percolation parameters, can then be derived by specifying $\kappa$ and a sequence of vectors $(\bm{s} (n),\bm{r} (n))_{n\in \mathbb{N}}$. We shall pose a regularity condition that ensures that the proportion of vertices of each type and threshold converges as $n\rightarrow \infty$.

\begin{assumption}\label{ass:regularity}For each $n$ and for given $k\in \mathbb{N}_0$ and Borel set $A \subset \cS$, let $U_k^{(n)}(A) = \{i \in [n]: r_i(n)=k, s_i(n) \in A\}$ be the set of vertices with threshold equal to $k$ and type in $A$. Let further $\nu^{(n)}(k,A) := |U_k^{(n)}(A)| / n$ be their proportion. We shall assume that there exists a measure $\nu$ on $\mathbb{N}_0 \times \cS$ such that for each fixed $k$, the measure $\nu(k,\cdot)$ is a Borel measure on $\cS$ and that for every $k\in \mathbb{N}_0$ and Borel set $A$ it holds that:
\begin{align} \label{eq:nu_init}
\nu^{(n)}(k,A) \xrightarrow[n\rightarrow \infty]{} \nu (k,A).
\end{align} 
We denote by $\mu$ the marginal measure of $\nu$ on $\cS$, and by $\eta_k$ the Radon-Nikodym derivative of $\nu (k, \cdot)$ against $\mu$ on $\cS$, that is, for every Borel set $A$,
\begin{align} \label{eq:margin}
\mu(A)=\sum_{k=0}^{\infty} \nu(k,A), 
&& \nu(k, A) = \int_{A} \eta_k(s) \dd \mu(s).
\end{align}
\end{assumption}
One way to ensure that Assumption~\ref{ass:regularity} is satisfied is by starting with a measure $\nu$ on $\mathbb{N}_0 \times \cS$ and assigning types and thresholds to vertices i.i.d. from this distribution. Then, the strong law of large numbers ensures that condition \eqref{eq:margin} holds. We provide examples in Section \ref{case:study}.
Let the vector $\cV:=(\kappa,\nu, \bm{s} (n),\bm{r} (n))_{n\in \mathbb{N}}$ collect all information about the vertex sequence and denote the random graph with $n$ vertices derived from this sequence by $G(n,\cV)$. The random graph model used here is an adapted version of the model proposed in \cite{bollobas_phase_2007}, enriched by the threshold values $\bm{r}(n)$, which are crucial parameters for the process we study. While the actual \textit{final set of infected vertices} $\mathbb{D}_n$ depends on the realization of the random graph $G(n,\cV)$, we will see that under Assumption~\ref{ass:regularity} for large $n$, a law of large numbers holds, and we can determine the limit in probability of $n^{-1}|\mathbb{D}_n|$ for $n\rightarrow \infty$ purely in terms of $\kappa$, $\nu$ and $\cS$.

\subsection{Results}
In this section we present the main results of this paper about the final number of infected vertices at the end of the bootstrap percolation process. We pose the following assumption throughout this paper. 

\begin{assumption} \label{ass:kernel}
We assume that the kernel $\kappa$ is continuous on the space $\cS\times \cS$, equipped with the product topology.
\end{assumption}

We provide specifications for $\kappa$ in Section \ref{case:study}. In light of Assumption~\ref{ass:kernel}, note that by the Tychonoff theorem, it follows from compactness of $\cS$ that $\cS \times \cS$ is compact with respect to the product topology, and then in particular that $\kappa$ is bounded on $\cS \times \cS$. We let in the following $M_{\kappa}$ be the bound of $\kappa$. Our main result will show that under Assumption~\ref{ass:regularity} and \ref{ass:kernel} the final proportion of infected vertices $n^{-1}|\mathbb{D}_n| $ converges in probability to a quantity that can explicitly be determined by solving a fixed point equation in a function space.
\begin{remark}
It would be interesting to allow for discontinuous $\kappa$ or a state space $\cS$ that is not compact. This more general model setup has the advantage that $\kappa$ is not necessarily bounded and that one may generate random graphs in which vertices of some type have very large degrees. However, without Assumption~\ref{ass:kernel} it appears that even the existence of a least fixed point of the related operators cannot be guaranteed, except for some special cases. 
From a practical perspective, when fitting the model to data, a bounded $\kappa$ may not be that restrictive, as one may choose a sufficiently large cutoff based on the largest observed degree. However, some effects of the process might only show when the tails of the distribution are fully included. In Section~\ref{case:study} we provide a numerical case study with two different unbounded kernels which suggests that the results extend to many unbounded kernels. Analyzing the precise growth conditions on $\kappa$ under which results hold for unbounded kernels seems very interesting and likely involves more advanced functional analytic techniques. We leave this for future research. Since allowing for bounded but possibly discontinuous $\kappa$ does not significantly enrich the theoretical treatment of the percolation process that we consider, we restrict to a setup that keeps the presentation simpler. 
\end{remark}
We now introduce the operators that will be important for the analysis of the bootstrap percolation process.

{\em Related operators:}
Let $\cF_b$ be the set of bounded, non-negative, and Borel measurable functions on $\cS$,
\begin{align} \label{eq:Fb_def}
\cF_b:=\{f:\cS \rightarrow \mathbb{R}_+, \text{bounded, Borel}\}.
\end{align} 
We equip $\cF_b$ with the supremum norm $\norm{\cdot}_{\infty}$, defined by $\norm{f}_{\infty}:=\sup_{s\in \cS} \abs{f(s)}$, which turns $\cF_b$ into a metric space with metric given by $d(f_1,f_2)=\norm{f_1(s)-f_2(s)}_{\infty}$. We stress that $\cF_b$ is actually not a vector space but only a convex cone because it contains only the non-negative functions. Although $\cF_b$ is not a vector space we will still use the norm symbol $\norm{\cdot}_{\infty}$ which causes no problem because $\cF_b\subset \{ f:\cS \rightarrow \mathbb{R}, \text{bounded}\}$, which is a Banach space. We further define 
\begin{align} \label{eq:F1_def}
\cF_1:=\{f:\cS \rightarrow [0,1], \text{Borel}\}\subset \cF_b,
\end{align}
the subset of those functions that are bounded by $1$, and the two constant functions $\cF_b \ni \mathbf{0}: \cS \rightarrow \{0\}$ and $\cF_b \ni \mathbf{1}: \cS \rightarrow \{1\}$. 

For a Borel measurable kernel function $\kappa:\cS^2 \rightarrow \mathbb{R}$, we define the operators $\Lambda_{\kappa}:\cF_b\rightarrow \cF_b , P_{\kappa}^k:\cF_b\rightarrow \cF_1$ and $\Psi_{\kappa}:\cF_b\rightarrow \cF_1$ by
\begin{equation}\label{eq:operator}
    \begin{aligned} 
        \Lambda_{\kappa} [f](\cdot) 
            &:= \int_{s \in \cS} \kappa(s, \cdot) f(s) \dd \mu(s), \\
        P_{\kappa}^k [f](\cdot) 
            &:= \frac{\(\Lambda_{\kappa} [f](\cdot)\)^k}{k !} e^{-\Lambda_{\kappa} [f](\cdot)}, \\
        \Psi_{\kappa} [f](\cdot) 
            &:= \sum_{k=0}^\infty \eta_k(\cdot)\(1- \sum_{k'=0}^{k-1}P_{\kappa}^{k'} [f](\cdot)\),
    \end{aligned}
\end{equation}
where $\eta_k$ is the Radon-Nikodym derivative defined in \eqref{eq:margin} and the convention $0^0 = 1$ is used. We use bold centered dot notation with bracket $[\boldsymbol{\cdot}]$ for arguments that are function-valued $f \in \cF_b$, and centered dot notation with parentheses $(\cdot)$ for arguments of type $s \in \cS$. In the following two lemmas, we state some important properties of the operators defined above. For $f, g \in \cF_b$, we say that $f \le g$ if and only if $f(s) \le g(s)$ for all $s \in \cS$.

\begin{lemma}[Monotonicity]\label{prop:psi_mono} 
For $f,g \in \cF_b$ with $f \le g$ it holds that $\Psi_{\kappa} f \le \Psi_{\kappa} g$.
\end{lemma}

To state the next lemma we define subsets of $\cF_b$ and $\cF_1$ defined in Equation~\eqref{eq:Fb_def} and \eqref{eq:F1_def} that consist of continuous functions: For $\star \in \{b,1\}$, define
\begin{equation} \label{eq:Fc_def}
\begin{aligned}
\cF_\star^c
    &:=\{f\in \cF_\star: f \text{ is continuous}\} \subset \cF_\star, \\
\cF_\star^{Lip}
    &:= \{f \in \cF_\star: f \text{ is Lipschitz with constant } L_{\kappa} (1+M_{\kappa})M_{\kappa} e^{M_{\kappa}}\} \subset \cF_\star^c.
\end{aligned}
\end{equation}

\begin{lemma}[Continuous image of $\Psi_{\kappa}$]\label{prop:psi_lipschitz}
For all $f \in\cF_b$, it holds $\Lambda_{\kappa} [f] \in \cF_b^{c}$, $P_{\kappa}^k [f]\in \cF_1^{c}$ and $\Psi_{\kappa} [f]\in \cF_1^{c}$. Moreover, if in addition $\kappa$ is Lipschitz continuous with constant $L_{\kappa}$, then for all $f\in\cF_b$, it holds $\Lambda_{\kappa} [f] \in \cF_b^{Lip}$, $P_{\kappa}^k [f]\in \cF_1^{Lip}$ and $\Psi_{\kappa} [f]\in \cF_1^{Lip}$.
\end{lemma}

We will see that if a least fixed point of the operator $\Psi_{\kappa}$ exists, then it allows us to determine the limit in probability of the final proportion of infected vertices. By fixed point, we mean a function $f\in\cF_b$ such that $f = \Psi_{\kappa} f$. With a Lipschitz condition on $\kappa$, we can prove the existence of a least fixed point of $\Psi_{\kappa}$:
\begin{lemma}[Existence of the least fixed point]\label{prop:fp_exist}
Let $\kappa$ be Lipschitz continuous with constant $L_{\kappa}$ and $\cF_1^{Lip}$ as defined in Equation~\eqref{eq:Fc_def}. Let further $\mathcal{H} \subset \cF_1^{Lip}$ be the set defined by 
\begin{align} \label{eq:fp}
\mathcal{H}:= \{f \in \cF_1^{Lip} : f = \Psi_{\kappa} f\}.
\end{align}
Then $\mathcal{H}\neq \emptyset$ and $\mathcal{H}$ includes all the fixed points. Moreover, there exists a function $\eta_0 \leq \hat{f}\in \mathcal{H}$ such that $\hat{f}\leq h$ for every $h \in \mathcal{H}$. We call $\hat{f}$ the \textit{least fixed point}. Clearly, it then also holds that $\int_{\cS} \hat{f} \dd \mu =\min_{f \in \mathcal{H}} \big\{ \int_{\cS} f \dd \mu \big\}$.
\end{lemma}

\begin{remark}\label{rank1:remark}
Let us consider the special case of rank one models, a particular popular example being the Norros-Reittu model \cite{Norros_Reittu_2006}. Assume that the kernel $\kappa:\cS^2 \rightarrow \mathbb{R}$ has multiplicative form, i.e. $\kappa (s,s')= \phi (s)\phi (s')$ with $\phi: \cS\rightarrow \mathbb{R}$ continuous. From the definition of $\Lambda_{\kappa} [f], P_{\kappa}^k [f]$ and $\Psi_{\kappa} [f]$, it then follows that the fixed point equation $f = \Psi_{\kappa} [f]$ can be written as $$f(s') 
= g\Bigl(s',\phi(s') \int_{s\in \mathcal S} \phi(s)\, f(s)\,\dd\mu(s)\Bigr)$$ for monotonically increasing $g : \mathbb{R}_+ \rightarrow [0,1]$ given by $$g(s',z)=\sum_{k=0}^\infty \eta_k (s') \(1- \sum_{k'=0}^{k-1}\frac{z^k}{k !} e^{-z}\).$$ Any fixed point must then be of the form $f_C(s) = g\bigl(s,\phi(s)\, C\bigr)$ with the constant $C := \int_{s \in \mathcal S} \phi(s)\, f(s)\,\dd\mu(s)$. But for $f_C$ to satisfy the fixed point condition, it must hold that $C = \int_{s\in \mathcal S} \phi(s)\, f_C(s)\,\dd\mu(s)=\int_{s\in \mathcal S} \phi(s)\, g\bigl(s,\phi(s)\, C\bigr)\,\dd\mu(s)$. We therefore obtain an one-dimensional map $$F(C) := \int_{s \in \mathcal S} \phi(s)\, g\bigl(s,\phi(s)\, C\bigr)\,\dd\mu(s)$$ and fixed point condition $C=F(C)$. Once we find a fixed point $\hat{C}$ for $F$, the fixed point $\hat{f}$ can be obtained via $\hat{f} (s)= g\bigl(s,\phi(s)\, \hat{C}\bigr)$, and conversely any fixed point arises in that way. 

This consideration includes the simplest example of the homogeneous and sparse Erd\H{o}s--R\'enyi graph with $c>0$ fixed and $\kappa (s,s')=c=\phi (s)\phi (s')$ with $\phi (s)=\sqrt{c}$. In this case, $\eta_k$ can be considered constant as the connection probability does not depend on the types, and $$g(s',z)=g(z)=\sum_{k=0}^\infty \eta_k \(1- \sum_{k'=0}^{k-1}\frac{z^k}{k !} e^{-z}\).$$ Any fixed point $\hat{f}$ is then of the form $\hat{f} (s)= g\bigl(\sqrt{c}\, \hat{C}\bigr)$ and thus is a constant function. This recovers the well known result about bootstrap percolation in this model, see \cite{janson_bootstrap_2012}, where an one-dimensional fixed point equation is derived.
\end{remark}
The proof of Lemma~\ref{prop:fp_exist} relies on an argument that requires us to find a subset $\mathcal{H}_1 \subset \cF_b$ of measurable functions that contains all fixed points of $\Psi_{\kappa}$ (if any) and which is such that for any subset $\mathcal{H}_2 \subset  \mathcal{H}_1$ the pointwise supremum (respectively infimum) of the functions in $\mathcal{H}_2$ is a function in $\mathcal{H}_1$. Because the supremum of measurable functions is not measurable in general, we need to choose a continuity property that is stable under taking the supremum. Without the assumption that $\kappa$ is Lipschitz continuous, it is therefore not possible to show the existence of a fixed point, except for some special cases where $\kappa$ is either of a particular structural form, or the thresholds of the vertices are all either $0$ or $1$. 

Next, we shall specify the Fr\'echet derivative on a convex cone. Just as the Jacobian represents all partial and directional derivatives of a map from $\mathbb{R}^n$ to $\mathbb{R}^m$, the Fr\'echet derivative does the same for functions between infinite-dimensional spaces. It is the linear operator that provides the best first-order approximation of a function near a point and encodes all directional derivatives at once.
Recall that $\cF_b$ is a convex cone such that $a f_1 + b f_2 \in \cF_b$ for $f_1, f_2 \in \cF_b$ and any non-negative scalar $a,b\geq 0$. Though the Fr\'echet derivative is usually defined on a vector space, it can be defined in the same spirit on $\cF_b$. Let $C(\cF_b,\cF_b)$ be the space of continuous operators from $\cF_b$ to itself and 
$$L^+ (\cF_b):=\{ L\in C(\cF_b,\cF_b) | L(a f_1 + b f_2)=aL(f_1) + b L(f_2)  \text{ for all } a,b\in \R_+ , f_1,f_2\in  \cF_b \}.$$
By definition $L^+ (\cF_b)$ is closed under linear combinations with non-negative scalar. For an operator $\Phi: \cF_b \rightarrow \cF_b$, we call $D\Phi: \cF_b \rightarrow L^+(\cF_b)$ the Fr\'echet derivative of $\Phi$ in $L^+(\cF_b)$ if for all $f,h\in \cF_b$ it holds 
\begin{equation}
\lim_{\norm{h}_{\infty}\rightarrow 0} \frac{\norm{\Phi(f+h) -\Phi(f)  -D \Phi f [h] }_{\infty}}{\norm{h}_{\infty}}=0.
\end{equation}
\begin{lemma}[Fr\'echet derivative] \label{prop:gd}
The operator $\Psi_{\kappa} : \cF_b \rightarrow \cF_b$ is Fr\'echet differentiable at every point $f\in \cF_b$ and the derivative $D \Psi_{\kappa}: \cF_b \rightarrow L^+(\cF_b)$, is given by 
\begin{equation}\label{eq:derivative}
D \Psi_{\kappa}f [\boldsymbol{\cdot}](\cdot) = \Lambda_{\kappa} [\boldsymbol{\cdot}]  (\cdot) V[f](\cdot)
\end{equation}
where $V[f](\cdot) := \sum_{k=1}^{\infty} \eta_k(\cdot)  P_{\kappa}^{k-1}[f] (\cdot)$.
\end{lemma}

\subsubsection{Final fraction of infected vertices.} \label{sec:final_infec}

We are now ready to state our first main result about the final number of infected vertices.
We start with graphs with $\nu (0, \cS) > 0$, which ensures that as $n\rightarrow \infty$ there exists a positive fraction of initially infected vertices. Later in Section \ref{sec:res}, we will turn to graphs without initial infections, and then study the impact of imposing a small fraction of infections. This will lead to the idea of resilience of the random graph. Let $\xrightarrow{p}$ denote convergence in probability. In the following we denote by $\mathbb{D}(G(n, \cV))$ the final set of infected vertices in the random graph $G(n, \cV)$.

\begin{theorem} [Final fraction of infected vertices]\label{thm:fp}
Let $\cV$ be a vertex sequence with resulting random graph $G(n,\cV)$ and assume that a least fixed point $\hat{f}$ of $\Psi_{\kappa}$ exists. Then the following holds:
\begin{enumerate}
\item For every $\varepsilon > 0$ it holds that 
\begin{equation}
 \lim_{n\rightarrow \infty} \Pb \left(n^{-1}|\mathbb{D}(G(n, \cV))| \geq \int_{\cS} \hat{f} \dd \mu -\varepsilon \right)=1.  
\end{equation}
\item If, in addition, there exists some non-negative and continuous function $h \in \cF_b$ and $\epsilon > 0$ such that
\begin{align} \label{eq:gd_cond_N}
D \Psi_{\kappa} \hat{f} [h] - h 
    < - \epsilon \mathbf{1},
\end{align}
then
\begin{align} \label{eq:fp_fz_N}
n^{-1}|\mathbb{D}(G(n, \cV))| \xrightarrow[n\rightarrow \infty]{p} \int_{\cS} \hat{f} \dd \mu.
\end{align}
\end{enumerate}
\end{theorem}

In particular for the case of $\kappa$ Lipschitz, we know by Lemma~\ref{prop:fp_exist} that a least fixed point exists. Except for some pathological cases, then indeed also the derivative condition \eqref{eq:gd_cond_N} holds. In fact, in most cases one may simply choose $h=\mathbf{1}$. This is because the first fixed point of the operator $\Psi_{\kappa}$ is a first joint zero for the function $f\mapsto \Psi_{\kappa}[f]-f$. Because $\Psi_{\kappa}$ is monotonically increasing by Lemma~\ref{prop:psi_mono}, it follows that $D \Psi_{\kappa} \hat{f} [\mathbf{1}]\geq \mathbf{0}$ and because $\hat{f}$ is the least fixed-point that $D \Psi_{\kappa} \hat{f} [\mathbf{1}] - \mathbf{1}\leq \mathbf{0}$. If the densities $\eta_k$ are continuous, then $D \Psi_{\kappa} \hat{f} [\mathbf{1}]$ is continuous and attains a maximum on $\mathcal{S}$. The only way that no $\epsilon>0$ exists such that $D \Psi_{\kappa} \hat{f} [\mathbf{1}] - \mathbf{1}\leq -\epsilon \mathbf{1}$ and thus possibly condition \eqref{eq:gd_cond_N} fails, is then that $D \Psi_{\kappa} \hat{f} [\mathbf{1}](s)=1$ for some $s\in \mathcal{S}$, which is a knife-edge tangency case vanishing with minor modifications of the system specifications. Indeed, in case of rank 1 models, it was shown in Remark~\ref{rank1:remark} that one obtains an one-dimensional fixed-point equation. The situation where \eqref{eq:gd_cond_N} fails would correspond to the first zero of $f\mapsto \Psi_{\kappa}[f]-f$ being a saddle point. Therefore, in most cases, Theorem~\ref{thm:fp} allows us to determine the final fraction of infected vertices at the end of the bootstrap percolation process resulting from an initial infection $\nu (0, \cS)>0$.

\subsubsection{Resilience.} \label{sec:res}

In the previous section, we studied the final set of infected vertices for a random graph with a positive fraction of initially infected vertices, i.e. $\nu (0,\cS) > 0$. However, for a graph without any initial infection we now ask the following important question: If we impose infection on a certain proportion of vertices to start the percolation, to what extent does the final fraction of infected vertices depend on the proportion of imposed infection? Does it exceed a certain level $\Delta$, no matter how small the infection is that we impose? This leads us to the question of resilience of the graph. To specify, we use the tilde notation to refer to the initially uninfected graph, starting from $\wtilde{\nu}(0, \cS) = 0$. We transform the graph from $\wtilde{\nu}(0, \cS) = 0$ to $\nu(0, \cS) > 0$ by setting the thresholds of some vertices to zero to mark them as infected while keeping the thresholds of all other vertices unchanged. If there exists a lower bound $\Delta$ such that any proportion of imposed infection leads to at least $\Delta n$ infected vertices at the end of the process, we call the graph \textit{non-resilient}. On the contrary, we call it \textit{resilient} if a reduction of the fraction of imposed infected vertices towards $0$ ensures that also the final fraction of infected vertices approaches $0$.

To formalize the setup, let $\wtilde{\cV}:=(\kappa,\wtilde{\nu},(\bm{s} (n),\wtilde{\bm{r}} (n))_{n\in \mathbb{N}}$ be a vertex sequence that fulfills Assumption~\ref{ass:regularity} and that is such that $\wtilde{\nu} (0,\cS)=0$. For each $n\in \mathbb{N}$, we now define a new threshold sequence $\bm{r} (n)= (r_1(n),\dots,r_n (n))$ which is such that for each $n\in \mathbb{N},1 \leq i\leq n$ either $r_i(n)=\wtilde{r}_i(n)$  or $r_i(n)=0$. Note that this change does not affect the probabilistic structure of the random graph, which remains totally unchanged, but only marks some vertices as infected. For $k\in \mathbb{N}$ and Borel set $A\subset \cS$, define
$$\nu^{(n)}(k,A) := n^{-1} \sum_{i \in [n]} I_{\{k\}}(r_i)I_{A}(s_i),$$
where $I_{B}$ is the indicator function of set $B$.
We further assume that for the new sequence there exists again a measure $\nu$ on $\mathbb{N}_0 \times \cS$ such that for each fixed $k$, the measure $\nu(k,\cdot)$ is a Borel measure on $\cS$ and that for every $k\in \mathbb{N}_0$ and Borel set $A$ it holds that:
\begin{align}
\nu^{(n)}(k,A) \xrightarrow[n\rightarrow \infty]{} \nu (k,A).
\end{align} 
Moreover, we assume that $\nu (0,\cS)>0$. This means that as a result of the threshold changes, a fraction of vertices is infected, and these infected vertices start the percolation process. Let as before $\cV:=(\kappa,\nu, \bm{s} (n),\bm{r} (n))_{n\in \mathbb{N}}$.

By construction of $\cV$, where we impose infection on $\wtilde{\cV}$ by changing the thresholds of some vertices to zero, we have the following relations for all Borel sets $A \subset \cS$:
\begin{align} \label{eq:resl_nu0k}
\wtilde{\nu}(k, A) - \nu (k, A) \geq 0\text{ for } k \geq 1, &&
    \nu(0, A)  =\sum_{k=1}^\infty \( \wtilde{\nu}(k, A) - \nu (k, A) \).
\end{align}
In the same fashion as in \eqref{eq:margin}, for subset $A \subset \cS$, we denote the marginal distribution and the Radon-Nikodym derivative by
\begin{align}
\wtilde{\mu}(A) := \sum_{k=0}^{\infty} \wtilde{\nu}(k, A),
&& \wtilde{\nu}(k, A) = \int_{A} \wtilde{\eta}_k(s) \dd \wtilde{\mu}(s).
\end{align}
We now define resilience and non-resilience as follows:
\begin{definition}
Let $\wtilde{\cV}$ be the initial sequence of uninfected vertices  and $\cV$ the sequence of vertices with imposed infections. Denote by $G(n, \wtilde{\cV})$ and $G(n, \cV)$ the corresponding random graphs that have the same distribution and differ in thresholds. 

We call the random graph $G(n, \wtilde{\cV})$
\begin{enumerate}
\item \textit{non-resilient}, if there exists $\Delta > 0 $ only depending on $\wtilde{\nu}$ such that $$\lim_{n\rightarrow \infty} \Pb \left(|\mathbb{D}(G(n, \cV))|/n \ge \Delta\right)=1.$$
\item  \textit{resilient}, if for all $\alpha >0$, there exists $\delta >0$ such that $$\lim_{n\rightarrow \infty} \Pb \left(|\mathbb{D}(G(n, \cV))|/n < \alpha \right)=1$$ for $0< \nu(0,\cS) < \delta$.
\end{enumerate}
\end{definition}
In other words, for a \textit{non-resilient} graph, no matter how small the initial infection is, the final fraction of infected vertices will always exceed a certain level, while, for a \textit{resilient} graph, the final fraction of infected vertices is small as long as the initial infection is small. In the following, we provide a condition that allows us to determine whether a random graph is resilient or non-resilient. 
Introduce the two operators $\wtilde{\Psi}: \cF_b \rightarrow \cF_b$ and $\Psi: \cF_b \rightarrow \cF_b$ corresponding to the uninfected $G(n,\wtilde{\cV})$ and the infected graph $G(n,\cV)$:
\begin{align}
\wtilde{\Psi} [f](\cdot) := \sum_{k=1}^\infty \wtilde{\eta}_k(\cdot)\(1- \sum_{k'=0}^{k-1}P_{\kappa}^{k'} [f](\cdot)\), &&
\Psi [f](\cdot) := \sum_{k=0}^\infty \eta_k(\cdot)\(1- \sum_{k'=0}^{k-1}P_{\kappa}^{k'} [f](\cdot)\),
\end{align}
where $P_{\kappa}^{k}$ is as defined in equation \eqref{eq:operator}. Note that the summation of $\wtilde{\Psi}$ starts from $k=1$ as $\wtilde{\nu}(0, \cS) =0$. 
For the operator $\wtilde{\Psi}$ we can calculate the Fr\'echet derivative $D \wtilde{\Psi}f$ at $f$ as in equation \eqref{eq:derivative}. The following theorem provides the resilience condition. Recall that $D \wtilde{\Psi} \mathbf{0}[h]$ is the Fr\'echet derivative of the operator $\wtilde{\Psi}$ at the point $\mathbf{0}$, applied to $h$.
\begin{theorem}[Resilience] \label{thm:resilience}
If there exists a continuous and non-negative function $h \in \cF_b$ such that
\begin{enumerate}
\item $D \wtilde{\Psi} \mathbf{0}[h] > h$. Then, the graph is non-resilient.
\item $D \wtilde{\Psi} \mathbf{0}[h] < h$. Then, the graph is resilient.
\end{enumerate}
Moreover, the graph cannot be both non-resilient and resilient, i.e. there is no pair of functions $h, g \in \cF_b$ such that $D \wtilde{\Psi}\mathbf{0}[h] > h$ and $D \wtilde{\Psi}\mathbf{0}[g] < g$ hold at the same time. 
\end{theorem}
\begin{remark}
It is worth noting that our notion of resilience is in some cases related to the existence of a giant component. In \cite[Thm. 3.1]{bollobas_phase_2007} it is shown that a giant component exists if the $\mathcal{L}^2$ norm of the operator $\Lambda_{\kappa}$ is larger than $1$. In our model, if all the vertices have a threshold equal to one, then our condition 1. implies that the $\mathcal{L}^2$ norm of $\Lambda_{\kappa}$ is larger than $1$ and the network must have a giant component.
\end{remark}

\section{Main proofs} \label{sec:main_proof}
In this section, for our two main results, Theorem~\ref{thm:fp} and Theorem~\ref{thm:resilience}, we provide the proof strategy, auxiliary results, and then the formal proofs. Proofs of auxiliary results are collected in Appendix~\ref{sec:aux}.

\begin{figure}[h]
\begin{center}
\begin{adjustbox}{width=0.8\linewidth}
\begin{tikzpicture}[shorten >=1pt,node distance=5cm,on grid,auto]
\tikzstyle{rrec} = [
draw, rectangle, rounded corners, align=center,
minimum width=1.5cm, 
minimum height=1cm,
]
\tikzstyle{arrow} = [->,>=stealth]
\tikzstyle{double_arrow} = [<->,>=stealth]
\node (vdlk_p) [rrec] {$G(n, \cV^{d+}_{L,K})$};
\node (vlk_p) [rrec, right = 5cm of vdlk_p] {$G(n, \cV^+_{L,K})$};
\node (vk) [rrec, below right = 1.5cm and 4cm of vlk_p] {$G(n, \cV_{K})$};
\node (vlk_m) [rrec, below left = 1.5cm and 4cm of vk] {$G(n, \cV^-_{L,K})$};
\node (vdlk_m) [rrec, left of = vlk_m] {$G(n, \cV^{d-}_{L,K})$};
\node (v) [rrec, right = 3cm of vk] {$G(n, \cV)$};
\node (eq) [align=left, below left = 1.5cm and 2cm of vdlk_p] {
    \footnotesize Solve $L$-dim fixed point eq.\\
    \footnotesize to determine $|\mathbb{D}(G(n, \cV_{L, K}^{d\pm}))|$
};
\draw [double_arrow] (vdlk_p) -- (vlk_p)
    node[above, pos=0.5] {
        \footnotesize Embedding
    }
    node[below, align=left, pos=0.5]{
        \footnotesize$|\mathbb{D}(G(n, \cV_{L, K}^{d+}))| = $ \\
        \footnotesize$|\mathbb{D}(G(n, \cV_{L, K}^+))|$
    };
\draw [arrow] (vlk_p) -| (vk)
    node[above, pos=0.25] {
        \footnotesize Coupling, \footnotesize $L \to \infty$
    } 
    node[below, align=left, pos=0.25]{
        \footnotesize $|\mathbb{D}(G(n, \cV_{L, K}^+))| \ge$ \\
        \footnotesize $|\mathbb{D}(G(n, \cV_{L, K}))|$
    };
\draw [double_arrow] (vdlk_m) -- (vlk_m)
    node[below, pos=0.5] {
      \footnotesize Embedding
    }
    node[above, align=left, pos=0.5]{
        \footnotesize $|\mathbb{D}(G(n, \cV_{L, K}^{d-}))|=$ \\
        \footnotesize $|\mathbb{D}(G(n, \cV_{L, K}^-))|$
    };
\draw [arrow] (vlk_m) -| (vk)
    node[below, pos=0.25] {
        \footnotesize Coupling, \footnotesize $L \to \infty$
    } 
    node[above, align=left, pos=0.25]{
        \footnotesize $|\mathbb{D}(G(n, \cV_{L, K}^-))| \le$ \\
        \footnotesize $|\mathbb{D}(G(n, \cV_{L, K}))|$
    } ;
\draw [arrow] (vk) -- (v) 
    node[above, pos=0.5] {
        \footnotesize $K \to \infty$
    };
\draw (vdlk_p) -| (eq);
\draw (vdlk_m) -| (eq);
\end{tikzpicture}
\end{adjustbox}
\caption{
Road map to prove Theorem~\ref{thm:fp}.
}
\label{fig:proof}
\end{center}
\end{figure}

\subsection{Final fraction of infected vertices}\label{proof:strat:final:frac}
The proof strategy for Theorem~\ref{thm:fp} is illustrated in Figure~\ref{fig:proof}. We first consider random graphs with a \textit{finite number of vertex types and a bounded threshold} (FTBT), denoted by $G(n, \cV^{d\pm}_{L,K})$, where $L$ denotes the number of types and $K$ denotes the maximal threshold. We then apply results from \cite{wormald_differential_1995} to show in Proposition~\ref{prop:fp_z} that the final number of infected vertices $|\mathbb{D}(G(n, \cV^{d\pm}_{L,K}))|$ in the FTBT random graph for large $n$ is with high probability close to a solution of an $L$-dimensional fixed point equation. Next, the FTBT graphs $G(n, \cV^{d\pm}_{L,K})$ are embedded into graphs with a \textit{compact type space $\cS$, a stepwise constant kernel, and bounded threshold} (SCBT), denoted by $G(n, \cV^\pm_{L,K})$ in such a way that they have the same final number of infected vertices ($|\mathbb{D}(G(n, \cV^{d\pm}_{L,K}))|=|\mathbb{D}(G(n, \cV^\pm_{L,K})|$).

The SCBT graphs serve as lower ($G(n, \cV^-_{L,K})$) and upper ($G(n, \cV^+_{L,K})$) bounds for a graph having the original \textit{compact type space $\cS$ and kernel with a bounded threshold $K$} (CSBT). These random graphs are denoted by $G(n, \cV_{K})$. We obtain the SCBT graphs for the upper and lower bounds from a sequence of nested partitions of the space $\cS$, indexed by $L \in \N$. Based on this sequence of partitions we can then define stepwise constant coupling kernels $\kappa_L^+$ and $\kappa_L^-$ for each $L \in \mathbb{N}$. The kernels $\kappa_L^+$ and $\kappa_L^-$ dominate the original $\kappa$ from above and below, and the least fixed points of $\Psi_{\kappa_L^-}$ and $\Psi_{\kappa_L^+}$ converge to the least fixed point of $\Psi_{\kappa}$. This allows us to couple the CSBT graph $G(n, \cV_{K})$ with two sequences of random graphs $G(n, \cV^{\pm}_{L,K})$, one which, for each random realization, has more edges than $G(n, \cV_{K})$, and another one with fewer edges. This leads to the bounds for the final fraction of infected vertices $|\mathbb{D}(G(n, \cV^{-}_{L,K}))| \le |\mathbb{D}(G(n, \cV_K))| \le |\mathbb{D}(G(n, \cV^{+}_{L,K}))|$.

In a next step we show that $$\lim_{L \to \infty}\abs{|\mathbb{D}(G(n, \cV^{-}_{L,K}))|-|\mathbb{D}(G(n, \cV^{+}_{L,K}))|}=0$$ and determine $|\mathbb{D}(G(n, \cV^{-}_{L,K}))|$ for large $n$ in Proposition~\ref{prop:fp_K}. This finally leads to the proof of Theorem~\ref{thm:fp} for $|\mathbb{D}(G(n, \cV))|$, the final infection in the random graph with the \textit{compact type space $\cS$} and unbounded threshold (CSUT) $G(n, \cV)$.
One particular challenge in these approximations is that the fixed point equations that determine $|\mathbb{D}(G(n, \cV^{\pm}_{L,K}))|$ are increasing in dimension with $L$ increasing. 

In the first step of the proof, in which we determine the final fraction of infected vertices in the finite type setting (FTBT) in Proposition~\ref{prop:fp_z}, we can follow ideas in \cite{articleAmini,Cont2016,Detering2018,Detering2016,BichuchDetering2021} to apply the concentration results in \cite{wormald_differential_1995} to derive the finite-dimensional fixed point equation that allow us to determine the final outcome of the process. All the remaining ideas of the proof, including the approximations, embeddings and operator results involved are entirely new.

\begin{definition}[Graph with finite vertex type and bounded threshold (FTBT)] \label{def:dis_sys}
For fixed $L,K \in \N$ we consider a random graph with vertices having types in the set $[L]=\{1, \dots, L\}$ and thresholds $m \in [K] \cup \{0\}=\{0, \dots, K\}$. The proportion of vertices with certain threshold and type is described by quantities $(\mathbf{\nu}^{l,(n)}_k)_{l \in [L], \ k \in [K] \cup \{0\}}$, where
\begin{equation}
\mathbf{\nu}^{l,(n)}_k = n^{-1}\sum_{i \in [n]} I_{\{k\}}(r_i)I_{\{l\}}(s_i).
\end{equation}
We assume that there exist constants $(\mathbf{\nu}^{l}_k )_{l \in [L], \ k \in [K] \cup \{0\}}$, to which the proportion of vertices with a certain type and threshold converges:
\begin{equation}
\mathbf{\nu}^{l,(n)}_k \xrightarrow[n\rightarrow \infty]{} \mathbf{\nu}^{l}_k.
\end{equation}
For the connection probability, we define a discrete kernel $\kappa^d_L: [L]^2 \rightarrow \mathbb{R}$. Note that we use a superscript $d$ to distinguish the kernel defined on $[L]^2$ from the one defined on $\cS^2$. Then, for vertices $i,j \in [n]$ with types $s_i, s_j \in [L]$, their connection probability is given by $\min \{ 1, \kappa^d_L (s_i, s_j) /n\}$. We define the vertex sequence $\cV_{L,K}^d =(\kappa^d_L,\nu, \bm{s} (n),\bm{r} (n))$ for this graph and denote by $G(n,\cV_{L,K}^d)$ the resulting random graph.
\end{definition}
For this finite-type vertex sequence, the percolation process can be fully described by the solution of a system of ordinary differential equations of dimension $L$. The final fraction of infected vertices $n^{-1}|\mathbb{D}(G(n, \cV_L^d))|$ is then, for large $n$, determined from the least joint zero of this system. For $l\in [L]$ and $k\in \mathbb{N}_0$, with slight abuse of notation, we define the functions $\nu_k^l: [0,1]^{L} \rightarrow \mathbb{R}$ by 
\begin{equation}\label{eq:ode_sol_z_1}
\begin{aligned}
\nu_0^l(\mathbf{z})
    &= - z^l + \nu_{0}^l(\mathbf{0}) + \sum_{k'=1}^K \nu_{k'}^l(\mathbf{0})\( 1 - \sum_{k''=0}^{k'-1} p(k'',\lambda^l(\mathbf{z})) \) \\
\nu_k^l(\mathbf{z})
    &= \sum_{k'=k}^K \nu_{k'}^l(\mathbf{0}) p(k' -k,\lambda^l(\mathbf{z})), \ \forall 1 \le k \le K-1, \\
\nu_K^l(\mathbf{z})
    &= \nu_{K}^l(\mathbf{0}) p(0,\lambda^l(\mathbf{z})),
\end{aligned}
\end{equation}
where  $\nu_{k}^l(\mathbf{0})$ is the initial proportion of vertices with type $l$ and threshold $k$, $p(k, \lambda) := \tfrac{\lambda^k}{k!}e^{-\lambda}$ is a Poisson probability function, $\mathbf{z}:=(z^l)_{l \in [L]} \in [0,1]^{L}$ is an $L$-dimensional vector, and $\lambda^l(\mathbf{z}) := \sum_{l' \in [L]} \kappa^d(l',l) z^{l'}$. Note that the quantities $\mathbf{z}$, $\lambda$, and $p(k, \lambda)$ are the discrete counterparts of the operators $f$, $\Lambda$ and $P$ defined in \eqref{eq:operator}. In a sequential formulation introduced in Appendix~\ref{sec:dis}, $\nu_{k}^l$ tracks the proportion of type $l$ vertices that need $k$ more infected neighbors to become infected. The argument vector contains for each type $l\in [L]$, the fraction of vertices of this type which are infected and whose impact on the system has already been explored. 

Define the \textit{least joint zero} as
\begin{align}\label{finite:fixed:point}
\what{\mathbf{z}}
    := \operatorname*{min} \left\{ \mathbf{z}\in [0,1]^{L} :\( \nu_0^l(\mathbf{z}) \)_{l\in [L]}= \mathbf{0} \right\},
\end{align}
where the minimum is with respect to the partial order obtained from the component-wise comparison $\le$. The existence of $\what{\mathbf{z}}$ follows from the Knaster-Tarski fixed point theorem. The following proposition determines the final fraction of infected vertices in $G(n,\cV^d_{L,K})$.
\begin{proposition}[Final fraction of infected vertices for graph with finite vertex type (FTBT)]\label{prop:fp_z}
Let $\nu (0,\mathcal{S})>0$ and let $\what{\mathbf{z}}$ be the least fixed point defined in \eqref{finite:fixed:point}. If for some vector $(w^l)_{l \in [L]} > \mathbf{0}$ and $\sum_{l\in[L]} w^l \le 1$, the condition
\begin{align} \label{eq:deriv_zhat}
\sum_{l' \in [L]} w^{l'} \frac{\d \nu_0^l(\what{\mathbf{z}})}{\d z^{l'}}  < 0, \ \forall l \in [L],
\end{align}
holds for the graph with finite vertex type $G(n,\cV_{L,K}^d)$, then it holds
\begin{align}
n^{-1}|\mathbb{D}(G(n,\cV_{L,K}^d))| \xrightarrow[n\rightarrow 
\infty]{p} \sum_{l\in [L]} \what{z}^l.
\end{align}
\end{proposition}

\begin{remark}
If condition \eqref{eq:deriv_zhat} does not hold, then still for every $\varepsilon>0$
\begin{align}
\lim_{n\rightarrow \infty} \mathbb{P}\left( n^{-1}|\mathbb{D}(G(n,\cV_{L,K}^d))|\ge \sum_{l\in[L]} \what{z}^l -\varepsilon \right)=1.
\end{align}
\end{remark}

For a given FTBT random graph sequence $\cV_{L,K}^d =(\kappa^d_L,\nu, \bm{s} (n),\bm{r} (n))$, we may define a partition $\{\cS_L^l\}_{l \in [L]}$ of the compact type space $\cS$ and a measure $\nu: [K]\cup \{0\} \times \cS \to \R_+$ such that $\nu(k, \cS_L^l) = \nu_k^l$ for all $k \in [K]\cup \{0\}$ and $l \in [L]$. Let $\mu(\cdot) = \sum_{k=0}^K \nu(k, \cdot)$. The finite-dimensional description of the graph $G(n,\cV^{d}_{L,K})$ can then be related to the operator $\Psi_{\kappa_L}$ of a SCBT graph $G(n,\cV_{L,K})$ that has a piece-wise constant kernel $\kappa_L: \cS \times \cS \to \R_+$. 

\begin{definition}[Embedding graph]
For a partition $\{\cS_L^l\}_{l \in [L]}$ define the stepwise kernel $\kappa_L(s, s'): \cS \times \cS \rightarrow \mathbb{R}$ and the vertex sequence $\cV_{L,K}$ by
\begin{align}
\kappa_L(s, s')
    :=\sum_{l,l' \in L} I_{\cS^l}(s) I_{\cS^{l'}}(s') \kappa^d_L(l, l'),
&&
\cV_{L,K} := (\kappa_L, \nu, (\bm{s}(n), \bm{r}(n))),
\end{align}
then the resulting graph $G(n, \cV_{L,K})$ is called the \textit{embedding graph} of $G(n, \cV^d_{L,K})$.
\end{definition}
By construction, $G(n, \cV^d_{L,K})$ and $G(n, \cV_{L,k})$ have the same distribution and by coupling the two random graphs, we may even ensure that $\abs{\mathbb{D}(G(n,\cV_{L, K}))}=\abs{\mathbb{D}(G(n, \cV^d_{L,K})}$. This allows us to extend the result of Proposition~\ref{prop:fp_z} for FTBT graphs to SCBT graphs. The functionals \eqref{eq:ode_sol_z_1} for $G(n, \cV^d_{L,K})$ and the functionals \eqref{eq:operator} for $G(n, \cV_{L,k})$ are related in the following sense: For the vector $\mathbf{z}_L=(z_L^1, \dots , z_L^L) \in [0,1]^{L}$, define a finite-dimensional function $f_L(s): \cS \rightarrow [0,1]$ by
\begin{align}
f_L(s)
    := \sum_{l \in [L]} \frac{1}{\mu(\cS_L^l)} z_L^l I_{\cS_L^l}(s),
\label{eq:step}
\end{align}
then by definition of $f_L$, it holds that $P_{\kappa}^k [f_L](s)=p(k,\lambda^l (\mathbf{z}_L))$ for $s\in \cS_L^l$ and any $k \in [K]\cup\{0\}$. This implies
\begin{equation}\label{embed:function}
\begin{aligned}
\nu_0^l(\mathbf{z}_L)
    &= - z_L^l + \sum_{k'=0}^K \nu_{k'}^l(\mathbf{0}) \( 1 - \sum_{k''=0}^{k'-1} p(k'',\lambda^l(\mathbf{z}_L))\) \\
    &=  - \int_{\cS_L^l} f_L(s) \dd \mu(s) +  \int_{\cS_L^l}\Psi_{\kappa_L} [f_L](s) \dd \mu(s), 
\end{aligned}
\end{equation}
for all $l\in [L]$ where $\nu_0^l(\mathbf{z}_L)$ and $\nu_k^l(\mathbf{0})$ are defined in Equation \eqref{eq:ode_sol_z_1}, and we also used that $\nu_{k}^l(\mathbf{0}) =\int_{\cS_L^l} \eta_k (s) \dd \mu(s)$. Because all fixed points of $\Psi_{\kappa_L}$ are step functions by Lemma \ref{lem:fp_conv}, the first zero $\what{\mathbf{z}}_L$ of $(\nu_0^l(\mathbf{z}))_{l\in [L]}$ and the first fixed point $\hat{f}_L$ of $\Psi_{\kappa_L}$ are related by 
\begin{align}
\hat{f}_L(s):= \sum_{l \in [L]} \frac{1}{\mu(\cS_L^l)} \what{z}_L^l I_{\cS_L^l}(s),
&&
\sum_{l \in [L]} \what{z}_L^l \label{eq:dtc_tau}
    &= \int_{\cS} \hat{f}_L (s) \dd \mu(s).
\end{align}
Moreover, for $0<w_l<1, \ \sum_{l \in [L]} w_j = 1$ and $h_L\in \cF_b$ defined by 
\begin{align}
h_L(s) := \sum_{l \in [L]} \tfrac{1}{\mu(\cS_L^l)}w^l I_{\cS_L^l}(s),
\end{align}
it holds that
\begin{equation}\label{eq:dtc_deriv}
\begin{aligned}
\sum_{l' \in [L]} w^{l'} \frac{\d \nu_0^l(\mathbf{z}_L)}{\d z^{l'}}
    &= - w^l +  \(\sum_{l' \in [L]} \kappa_L^d(l',l) w^{l'}\) \nu_1^l(\mathbf{z}_L) \\
    &= - \int_{\cS_L^l} h_L(s) \dd \mu(s) +  \int_{\cS_L^l} D \Psi_{\kappa_L} f_L [h_L] (s) \dd \mu(s).
\end{aligned}
\end{equation}
On the left hand side we have the directional derivative of the multivariate function $(\nu_0^l(\mathbf{z}))_{l\in [L]}$ while on the right hand side we have an expression that involves the Fr\'echet derivative in the direction $h_L$. The identity will allow us to derive stopping criteria for the percolation process in $G(n,\cV_{L,K})$ from the Fr\'echet derivative of $\Psi_{\kappa}$. 

Because of the equalities \eqref{embed:function}, \eqref{eq:dtc_tau}, and \eqref{eq:dtc_deriv} we say that $G(n,\cV^d_{L,K})$ is embedded into $G(n,\cV_{L, K})$. This embedding will allow us to describe the final default fraction $\abs{\mathbb{D}(G(n,\cV_{L, K}))}=\abs{\mathbb{D}(G(n, \cV^d_{L,K})}$ in terms of the functionals \eqref{eq:operator} for $G(n,\cV_{L, K})$. We next apply this observation to analyze the lower and upper bounds of the final fraction of infected vertices in CSBT graphs by constructing its \textit{coupling graphs}. For this we shall first specify a sequence of partitions for $\cS$ based on the following lemma. Let $d: \cS \times \cS \to \R_+$ be the metric of the type space $\cS$, and for subset $\mathcal{R} \subset \cS$, define $\mathrm{diam}(\mathcal{R}):= \sup_{x, y \in \mathcal{R}} d(x,y)$.
\begin{lemma}[Type space partition]
\label{prop:partition}
There exists a sequence, indexed by $m \in \N$, of partitions $\{\cS_{L(m)}^l\}_{l \in [L(m)]}$ of $\cS$ with $L(m)$ the finite number of subsets and with each set $\cS_{L(m)}^l  \subset \mathcal{S}$ Borel. Further, the sequence is such that 
\begin{enumerate}
\item $\{\cS_{L(m')}^l\}_{l \in [L(m')]}$ for each $m' > m$ is a refinement of $\{\cS_{L(m)}^l\}_{l \in [L(m)]}$, i.e. each  $\cS_{L(m)}^l$ is a union $\cup_{j\in J} \cS_{L(m')}^j$ for some index set $J\subset L(m')$.
\item As $m \rightarrow \infty$, $\mathrm{diam}(S_{L(m)}^l) \rightarrow 0$, uniformly for all $l \in [{L(m)}]$.
\end{enumerate}
\end{lemma}

For a CSBT graph $G(n, \cV_K)$, we construct two sequences of coupling SCBT graphs, indexed by $L \in \mathbb{N}$, $G(n, \cV_{L, K}^-)$ and $G(n, \cV_{L, K}^+)$ as follows. In the following we often drop the dependence of $L$ on $m$ to simplify notation.

\begin{definition}[Coupling graphs] \label{def:couple}
Let $\{\cS_{L(m)}^l\}_{l \in [L(m)]}$ be a sequence of partitions for the type space $\cS$ according to Lemma \ref{prop:partition}. In the following we omit $m$ and just write $L$. For each $L$, define lower coupling graph $G(n, \cV_{L, K}^-)$ and an upper coupling graph $G(n, \cV_{L, K}^+)$ as random graphs with vertex sequence $\cV_{L, K}^-:=(\kappa_L^-,\nu, \bm{s} (n),\bm{r} (n))$ and $\cV_{L, K}^+:=(\kappa_L^+,\nu, \bm{s} (n),\bm{r} (n))$, where the stepwise kernels $\kappa_L^{\pm}(s, s')$ are given by
\begin{align} \label{eq:step_kernel}
\kappa_L^+(s, s')  
    &:= \sum_{l,l' \in [L]} I_{\cS_L^l}(s) I_{\cS_L^{l'}}(s') \sup_{x \in \cS_L^l, y \in \cS_L^{l'}} \kappa(x,y), \\
\kappa_L^-(s, s')  
    &:= \sum_{l,l' \in [L]} I_{\cS_L^l}(s) I_{\cS_L^{l'}}(s') \inf_{x \in \cS_L^l, y \in \cS_L^{l'}} \kappa(x,y).
\end{align}
Moreover, edges of $G(n, \cV_{L, K}^{\pm})$ are generated as follows. For any pair of vertices $i,j \in [n]$ in a graph $G$, let $E_{ij}(G)$ be the binary indicator that $i$ connects to $j$. For $i,j \in [n]$,$i\neq j$, let $U_{ij} \sim \text{Uniform([0,1])}$, let $E_{ij}(G(n, \cV_{L, K}^-)) = I_{\{U_{ij} \le \kappa_L^- (s_i, s_j)\}}$, $E_{ij}(G(n, \cV)) = I_{\{U_{ij} \le \kappa(s_i, s_j)\}}$, and $E_{ij}(G(n, \cV_{L, K}^+)) = I_{\{U_{ij} \le \kappa_L^+ (s_i, s_j)\}}$.
\end{definition}
By construction of the coupling graphs it holds by $\kappa_L^- \le \kappa \le \kappa_L^+$ that 
$$\mathbb{P}(E_{ij}(G(n, \cV_{L, K}^-)) \le E_{ij}(G(n, \cV_K)) \le E_{ij}(G(n, \cV_{L, K}^+)))=1$$ as desired, and it follows that
\begin{align} \label{eq:D_ineq}
\mathbb{P}(|\mathbb{D}(G(n, \cV_{L, K}^-))| \le |\mathbb{D}(G(n, \cV_K))| \le |\mathbb{D}(G(n, \cV_{L, K}^+))|)=1.
\end{align}
Note that $G(n, \cV_{L, K}^{\pm})$ embeds $G(n, \cV_{L, K}^{d, \pm})$ and we have $|\mathbb{D}(G(n, \cV_{L, K}^\pm))|=|\mathbb{D}(G(n, \cV_{L, K}^{d,\pm}))|$. Then $|\mathbb{D}(G(n, \cV_L^{d,\pm}))|$ can be determined based on the least joint zeros of the functions in \eqref{eq:ode_sol_z_1}. The embedding relations in \eqref{eq:dtc_tau} and \eqref{eq:dtc_deriv} allows us to connect the finite dimensional discrete description to the operators $\Psi_{\kappa_L^\pm}$. We expect that $$\lim_{L \to \infty}\abs{|\mathbb{D}(G(n, \cV_{L, K}^-))| - |\mathbb{D}(G(n, \cV_{L, K}^+))|}=0.$$ This can be shown with the convergence of the operators $\Psi_{\kappa_L^\pm}$ to $\Psi_{\kappa}$ on $\cF_b$ as $L \rightarrow \infty$, as well as the convergence of their fixed points, provided in the following two lemmas.

\begin{lemma}[Uniform convergence] \label{lem:converge}
For $f, h \in \cF_b$ and a monotone sequence $\{f_L\}_{L\in \mathbb{N}} \in \cF_b$ such that $f_L \rightarrow f$ pointwise as $L \rightarrow \infty$, it holds 
\begin{align}
\Psi_{\kappa_L^\pm} f_L \rightarrow \Psi_{\kappa} f,
&& D \Psi_{\kappa_L^\pm} f_L[h] \rightarrow D \Psi_{\kappa} f[h],
\end{align}
uniformly in $\cS$.
\end{lemma}

\begin{lemma}[Uniform convergence of fixed points] \label{lem:fp_conv}
Assume that a least fixed point $\hat{f}$ of $\Psi_{\kappa}$ exists and that there exists a non-negative continuous function $h \in \cF_b$ and small $\epsilon > 0$ such that $D \Psi_{\kappa} \hat{f} [h] - h < - \epsilon \mathbf{1}$. Then each of the $\Psi_{\kappa^\pm_L}$ has a least fixed point, denoted by $\hat{f}_L^\pm$. The fixed point $\hat{f}_L^\pm$ is a step function that is constant on the sets $\{\cS^l_L \}_{l\in L}$. Moreover, as $L \rightarrow \infty$, the following convergences hold:
\begin{align}
\hat{f}_L^\pm \rightarrow \hat{f}  \text{ uniformly}, &&
D \Psi_{\kappa^\pm_L} \hat{f}_L^\pm [h]\rightarrow  D \Psi_{\kappa} \hat{f} [h] \text{ uniformly}, &&
\int_{\cS}\hat{f}_L^\pm \dd \mu \rightarrow \int_{\cS} \hat{f} \dd \mu.
\label{eq:fz_con}
\end{align}
\end{lemma}
\begin{remark}
Lemma~\ref{lem:fp_conv} shows that the least fixed points of the coupling operators $\Psi_{\kappa_L^{\pm}}$ converge to the least fixed points of $\Psi_{\kappa}$. 
While Lemma \ref{lem:converge} follows relatively straightforward from assumptions on $\kappa$ and the definition of $\Psi$, showing fixed point convergence is more involved. In our case an additional complication arises due to the increasing dimension of the step functions which depend on $L$.
\end{remark}

The following proposition determines the final fraction of infected vertices of $G(n, \cV_K)$.

\begin{proposition}[Final fraction of infected vertices for graph with bounded threshold (CSBT)]\label{prop:fp_K}
Let $\hat{f}_K$ be the least fixed point of $\Psi_{\kappa}$. Assume that for some positive continuous function $h \in \cF_b$ and small $\epsilon>0$, the following derivative condition holds:
\begin{align} \label{eq:gd_cond}
D \Psi_{\kappa} \hat{f}_K [h] - h < - \epsilon \mathbf{1}.
\end{align}
Then, the final fraction of infected vertices in $G(n,\cV_K)$ converges:
\begin{align} \label{eq:fp_fz}
n^{-1}|\mathbb{D}(G(n,\cV_K))|\xrightarrow[n\rightarrow \infty]{p} \int_{\cS} \hat{f}_K \dd \mu.
\end{align}
\end{proposition}

\begin{proof}
For $G(n,\cV_{K})$, build the coupling graphs $G(n,\cV_{L,K}^\pm)$ as specified in Definition \ref{def:couple}. It then holds for the final fraction of infected vertices $|\mathbb{D}(G(n,\cV_{K}))|$ that
\begin{align} \label{eq:D_bound}
\mathbb{P}(|\mathbb{D}(G(n, \cV_{L,K}^-))| \le |\mathbb{D}(G(n, \cV_K))| \le |\mathbb{D}(G(n, \cV_{L,K}^+))|)=1.
\end{align}
Let $\hat{f}^\pm_{L,K}$ be the least fixed points of $G(n,\cV_{L,K}^\pm)$, which by Lemma~\ref{lem:fp_conv} are step functions. Construct the corresponding embedded finite-type graphs $G(n,\cV_{L,K}^{d,\pm})$ and let $\mathbf{\what{z}}^\pm_{L,K}$ be the corresponding \textit{least joint zeros} as defined in \eqref{finite:fixed:point}. The relation between the operator of $G(n,\cV_{L,K}^\pm)$ and that of $G(n,\cV_{L,K}^{d,\pm})$ is shown in equations \eqref{eq:step} - \eqref{eq:dtc_tau} when we replace ($\kappa_L$, $f_L$, $\kappa^d_L$, $\mathbf{z}_L$) with ($\kappa^\pm_L$, $\hat{f}^\pm_{L,K}$, $\kappa^{d, \pm}_L$, $\mathbf{\what{z}}^\pm_{L,K}$). We will use such replacement throughout this proof whenever we refer to those equations.

By Lemma~\ref{lem:fp_conv}, there exists a $L_0 \in \mathbb{N}$ such that $\norm{D\Psi_{\kappa^\pm_L}\hat{f}^\pm_{L,K}[h] - D\Psi_{\kappa} \hat{f}_K[h]}_\infty < \epsilon/2$ for $L \ge L_0$. Then with inequality \eqref{eq:gd_cond} it follows
\begin{align}\label{eq:gd_cond_couple}
D\Psi_{\kappa^\pm_L}\hat{f}^\pm_{L,K}[h] - h
    &= (D\Psi_{\kappa^\pm_L}\hat{f}^\pm_{L,K}[h] - D\Psi_{\kappa} \hat{f}_K [h]) + (D\Psi_{\kappa} \hat{f}_K[h] - h) \\
    &<  (\frac{\epsilon}{2} - \epsilon) \mathbf{1}
    = - \frac{\epsilon}{2} \mathbf{1}.
\end{align}
Let $L \ge L_0$, combining equation \eqref{eq:dtc_deriv} with equation \eqref{eq:gd_cond_couple} yields
\begin{align}
\sum_{l' \in [L]} w^{l'} \frac{\d \nu_0^l(\mathbf{\what{z}}^\pm_{L,K})}{\d z^{l'}} <  - \frac{\epsilon}{2}, \quad \forall l \in [L],
\end{align}
which is the condition that allows us to apply Proposition~\ref{prop:fp_z} to $G(n,\cV_{L,K}^{d,\pm})$. Therefore, the final fraction of infected vertices of the embedded graph with finite vertex types converges
\begin{align} \label{eq:D_couple_conv}
n^{-1}|\mathbb{D}(G(n,\cV^\pm_{L,K}))| \xrightarrow[n\rightarrow \infty]{p} \sum_{l\in[L]} \what{z}^{\pm, l}_{L,K} = \int_{\cS} \hat{f}^\pm_{L,K} \dd \mu,
\end{align}
where we use equation \eqref{eq:dtc_tau} on the right hand side.
Finally, we recall inequality \eqref{eq:D_bound}, and apply the convergence Lemma~\ref{lem:fp_conv} to \eqref{eq:D_couple_conv} to obtain
\begin{align}
 n^{-1}|\mathbb{D}(G(n,\cV_{L,K}^+))| 
    \xrightarrow[n\rightarrow \infty]{p}
    \int_{\cS} \hat{f}^+_{L,K} \dd \mu
    \xrightarrow[L \rightarrow \infty]{}
    \int_{\cS} \hat{f}_K \dd \mu, \\
 n^{-1}|\mathbb{D}(G(n,\cV_{L,K}^-))|
    \xrightarrow[n\rightarrow \infty]{p}
    \int_{\cS} \hat{f}^-_{L,K} \dd \mu
    \xrightarrow[L \rightarrow \infty]{}
    \int_{\cS} \hat{f}_K \dd \mu,
\end{align}
and hence the result follows.
\end{proof} 

Finally, given the above results, we are now able to prove the main Theorem~\ref{thm:fp}.
\begin{proof}[Proof of Theorem~\ref{thm:fp}]
Note that $\sum_{k=1}^{\infty} \nu(k,\cS) = 1$ implies that for all $\epsilon > 0$, there is a $K_0 \in \N$ such that for $K \ge K_0$ we have $\sum_{k=K+1}^{\infty} \nu(k, \cS) < \epsilon$.
Define $\epsilon_{K}(\cdot): \cS \rightarrow [0,1]$ by
\begin{align}
\epsilon_{K}(\cdot)
    := \sum_{k=K+1}^{\infty} \nu (k, \cdot), &&
\epsilon_{K} \equiv \epsilon_{K}(\cS)
\end{align}
and note that $\lim_{K \rightarrow \infty }\epsilon_{K} =0$. We build an \textit{upper coupling graph} $G(n,\overline{\cV}_K)$ and a \textit{lower coupling graph} $G(\lfloor(1-\epsilon_K)n\rfloor,\underline{\cV}_K)$ with vertex sequence $\overline{\cV}_K := (\kappa,\overline{\nu}^K,(\bm{s} (n),\bm{r} (n))$ and $\underline{\cV}_K := (\kappa,\underline{\nu}^K,(\bm{s} (\lfloor(1-\epsilon_K)n\rfloor),\bm{r} (\lfloor(1-\epsilon_K)n\rfloor)$. Their vertex distributions are
\begin{align}
\overline{\nu}^K(k,\cdot)
    =  \begin{cases}
        \nu(k,\cdot), &  0\le k \le K-1 \\
        \sum_{k'=K}^{\infty} \nu(k',\cdot), & k = K
        \end{cases}, &&
\underline{\nu}^K(k,\cdot)
    = \frac{\nu(k, \cdot)}{1 - \epsilon_K}, \quad & 0 \le k \le K.
\end{align}
For $G(n,\overline{\cV}_K)$, we clip the largest thresholds to $K$ by mapping all vertices in $G(n,\cV)$ with threshold $k \ge K$ into vertices with only threshold $K$ in $G(n,\overline{\cV}_K)$. Hence, $G(n,\overline{\cV}_K)$ is more vulnerable to infection than $G(n,\cV)$. For the sequence $\underline{\cV}_K$, we exclude vertices with thresholds larger than $K$. Their proportion is $\epsilon_K$. We then consider the graph $G(\lfloor(1-\epsilon_K)n\rfloor,\underline{\cV}_K)$ with only $(1- \epsilon_K)n$ vertices. This is equivalent to considering a graph with $n$ vertices in which the vertices with large threshold values are considered non-infectable. Hence, in absolute numbers $G(\lfloor(1-\epsilon_K)n\rfloor,\underline{\cV}_K)$ has fewer  infected vertices than $G(n,\cV)$. Consequently, by such a construction and again a coupling argument, we have the following inequality:
\begin{align} \label{eq:D_K_ineq}
    \mathbb{P}(n^{-1}|\mathbb{D}(G(\lfloor(1-\epsilon_K)n\rfloor,\underline{\cV}_K))| \le  n^{-1}|\mathbb{D}(G(n,\cV))| \le  n^{-1}|\mathbb{D}(G(n,\overline{\cV}_K))|)=1.
\end{align}
The operators for $G(n,\overline{\cV}_K)$ and $G(\lfloor(1-\epsilon_K)n\rfloor,\underline{\cV}_K)$ are given by
\begin{align}
\overline{\Psi}_K [f](\cdot) 
    &:= \sum_{k=0}^K \frac{\dd \overline{\nu}^K(k, \cdot)}{\dd \overline{\mu}^K(\cdot)}\(1- \sum_{k'=0}^{k-1}P^{k'} [f](\cdot)\), \\
\underline{\Psi}_K [f](\cdot) 
    &:= \sum_{k=0}^K \frac{\dd \underline{\nu}^K(k, \cdot)}{\dd \underline{\mu}^K(\cdot)}\(1- \sum_{k'=0}^{k-1}P^{k'} [f](\cdot)\).
\end{align}

The graphs $G(n,\overline{\cV}_K)$ and $G(\lfloor(1-\epsilon_K)n\rfloor,\underline{\cV}_K)$ are CSBT random graphs and we can therefore determine $|\mathbb{D}(G(\lfloor(1-\epsilon_K)n\rfloor,\underline{\cV}_K))|$ and $|\mathbb{D}(G(n,\overline{\cV}_K))|$ with Proposition \ref{prop:fp_K}. Working through essentially the same techniques as used in Lemma~\ref{lem:converge} and Lemma~\ref{lem:fp_conv}, we can easily show the existence of the least fixed points $\overline{f}_K$ (resp. $\underline{f}_K$) for $\overline{\Psi}_K$ (resp. $\underline{\Psi}_K$), and further prove the following convergence for $K\rightarrow \infty$:
\begin{equation} \label{eq:fz_con_K}
\begin{aligned}
\overline{f}_K, \underline{f}_K
    \rightarrow \hat{f}  \text{ uniformly}, &&
D \overline{\Psi}_{K} \overline{f}_K[h], D \underline{\Psi}_{K} \underline{f}_K[h]  
    \rightarrow  D \Psi_\kappa \hat{f} [h] \text{ uniformly},
\end{aligned}    
\end{equation}
and 
$$\int_{\cS}\overline{f}_K, \int_{\cS}\underline{f}_K  
    \rightarrow \int_{\cS} \hat{f}.$$
Finally, by condition \eqref{eq:gd_cond_N} on the derivative $D \Psi_\kappa$, the second convergence in \eqref{eq:fz_con_K} implies that $D \overline{\Psi}_{K}$ and $D \underline{\Psi}_{K}$ satisfy such derivative condition for $K$ large. Therefore, applying Proposition~\ref{prop:fp_K}, we obtain
\begin{align}
\frac{|\mathbb{D}(G(\lfloor(1-\epsilon_K)n\rfloor ,\underline{\cV}_K))|}{(1-\epsilon_K)n} 
    \xrightarrow[n\rightarrow \infty]{p} \int_{\cS} \underline{f}_K
    \xrightarrow[K\rightarrow \infty]{} \int_{\cS} \hat{f}
    \xleftarrow[K\rightarrow \infty]{} \int_{\cS} \overline{f}_K
    \xleftarrow[n\rightarrow \infty]{p}
    \frac{|\mathbb{D}(G(n,\overline{\cV}_K))|}{n}.\nonumber 
\end{align}
Then, together with the inequality \eqref{eq:D_K_ineq} and $\epsilon_K \rightarrow 0$, \eqref{eq:fp_fz_N} follows.
\end{proof}

\subsection{Resilience}\label{proofs:explained:resilience}
For the proof of non-resilience, one first notes that the existence of a function $h\in \cF_b$ such that $D \wtilde{\Psi} \mathbf{0}[h] > h$ implies by continuity of the derivative, that there exists $\alpha > 0$ such that $\wtilde{\Psi}[a h] >0$ for $a \in (0, \alpha ]$. We will then see that for the operator $\Psi $ of the infected network it actually  holds that $\Psi [a h] \geq \wtilde{\Psi}[a h] >0$ for $a \in (0, \alpha ]$.The lemma below ensures that a fixed point $\hat{f}$ of $\Psi$ (if any) is such that $\hat{f} \geq \alpha h$, which will provide the lower bound. Recall the definition of  $\cF_1^c$ in Equation~\eqref{eq:Fc_def}.
\begin{lemma} \label{prop:no_fix}
If there exists an $h \in \cF_1^c$ and $a_0 \in [0,1]$ such that $\Psi [a h] > a h$ for all $a \in [0, a_0]$, then, for any $g\in \cF_b$ such that $\Psi [g] = g$, it holds that $g\geq a_0 h$.
\end{lemma}

To show resilience, we use Lemma~\ref{lem:converge} together with the embedding \eqref{eq:dtc_deriv} to derive upper approximating kernels $\kappa_L^+$ and coupling graphs $G(n,\cV_L^+)$. The derivative condition then allows one to show that the expected infection  across all generations of the process is bounded by $C \nu(0, \cS)$, where $C>0$ results from a geometric series. A simple Markov estimate then provides the bound for $|\mathbb{D}(G(n,\cV_L^+))|$ (see Proposition~\ref{prop:remain}).

\begin{proof}[Proof of Theorem~\ref{thm:resilience}]
Note that $\mathbf{0}$ is a fixed point of $\wtilde{\Psi}$. First, for part (1), we will show that, given the inequality condition, there exists another fixed point $\wtilde{f} \in \cF_b$ of $\wtilde{\Psi}$ with $\wtilde{f}> \mathbf{0}$, which provides the lower bound of the total infected proportion for $G(n,\cV)$. For such $h$, we can take a small $\epsilon > 0$ such that
\begin{align} \label{eq:d0_cond}
D \wtilde{\Psi}\mathbf{0}h(\cdot)
    &=\Lambda_{\kappa} [h]  (\cdot)\sum_{k=1}^\infty \wtilde{\eta}_k(\cdot) P_{\kappa}^{k-1}[\mathbf{0}] (\cdot) \\
    &= \Lambda_{\kappa} [h] (\cdot)\wtilde{\eta}_1(\cdot) 
    > (1-\epsilon) \Lambda_{\kappa} [h] (\cdot) \wtilde{\eta}_1(\cdot)
    > h(\cdot),
\end{align}
where $\eta_k$ is the Radon-Nikodym derivative defined in \eqref{eq:margin}. Then, take $a \in (0,1)$, by the mean value theorem, there exists a $\xi \in (0,1)$ such that
\begin{align}
\wtilde{\Psi} [\mathbf{0} + ah] - \wtilde{\Psi} \mathbf{0}
    &= D\wtilde{\Psi} [\xi ah] [ah] = \Lambda_{\kappa} [ah] \sum_{k=1}^\infty\wtilde{\eta}_k  P_{\kappa}^{k-1}[\xi ah]\\
    &= a\Lambda_{\kappa}[h]  \wtilde{\eta}_1    \sum_{k=1}^\infty \frac{\wtilde{\eta}_k }{\wtilde{\eta}_1 } P_{\kappa}^{k-1}[\xi ah] \\
    &\ge a\Lambda_{\kappa}[h] \wtilde{\eta}_1 P_{\kappa}^{0}[\xi ah] = a\Lambda_{\kappa}[h]  \wtilde{\eta}_1 e^{- \xi a \Lambda_\kappa h}.
\end{align}
Since $e^{- \xi 0 \Lambda_\kappa h} = \mathbf{1}$ and $e^{- \xi a \Lambda_\kappa h}$ is continuous in $a$ and $h$ is bounded, there exists a small neighborhood $(0, a_\epsilon] \subset (0, 1)$ such that $e^{- \xi a \Lambda_\kappa h(s)} \ge 1-\epsilon$ for all $a \in (0, a_\epsilon]$ and $s \in \cS$. Therefore, with inequality \eqref{eq:d0_cond}, we obtain 
\begin{align} \label{eq:dh_cond}
\wtilde{\Psi}[ah]
    > a\Lambda_{\kappa}[h] \wtilde{\eta}_1(\cdot) (1- \epsilon)
    > a h, \quad \forall a \in (0, a_\epsilon].
\end{align}

Then, expanding the expression of $\Psi[ah]$ and using the relation of $\eta_k$ and $\wtilde{\eta}_k$ indicated by \eqref{eq:resl_nu0k}, we derive the following inequality:
\begin{align}
\Psi[ah]
    &=  \sum_{k=0}^\infty \eta_k\(1- \sum_{k'=0}^{k-1}P_{\kappa}^{k'} [ah]\) \\
    &=  \sum_{k=0}^\infty (\eta_k - \wtilde{\eta}_k)\(1- \sum_{k'=0}^{k-1}P_{\kappa}^{k'} [ah]\) +  \sum_{k=0}^\infty \wtilde{\eta}_k\(1- \sum_{k'=0}^{k-1}P_{\kappa}^{k'} [ah]\)\\
    &= \eta_0 + \sum_{k=1}^\infty (\eta_k - \wtilde{\eta}_k)\(1- \sum_{k'=0}^{k-1}P_{\kappa}^{k'} [ah]\) +  \wtilde{\Psi}[ah]\\
    &= \sum_{k=1}^\infty (\wtilde{\eta}_k - \eta_k) + \sum_{k=1}^\infty (\eta_k - \wtilde{\eta}_k)\(1- \sum_{k'=0}^{k-1}P_{\kappa}^{k'} [ah]\) +  \wtilde{\Psi}[ah]\\
    &= \sum_{k=1}^\infty (\wtilde{\eta}_k - \eta_k)\sum_{k'=0}^{k-1}P_{\kappa}^{k'} [ah] + \wtilde{\Psi}[ah] \ge \wtilde{\Psi}[ah] > ah.
\end{align}
By Lemma~\ref{prop:no_fix}, we know that $\Psi$ does not have a fixed point between $\mathbf{0}$ and $a_\epsilon h$. Therefore, by Theorem~\ref{thm:fp}, the final fraction of infected vertices of $G(n, \cV)$ must be greater than $\alpha := \int_{\cS} a_{\epsilon} h \dd \mu$, where $\alpha$ is independent of the imposed infection $\nu_0$. Hence, we have shown the non-resilient case (1).

For the resilient case (2), we first observe that 
\begin{align}
D \wtilde{\Psi}\mathbf{0} [h]
    = \Lambda_{\kappa} [h] \wtilde{\eta}_1 \ge \Lambda_{\kappa} [h] \eta_1 =D \Psi\mathbf{0} [h],
\end{align}
which implies that $D \Psi\mathbf{0} [h] < h $ by the assumption in (2). Because $D\Psi \mathbf{0} [h]$ and $h$ are continuous, it follows that $D\Psi \mathbf{0} [h] - h$ is continuous, and therefore attains a maximum. Thus, we can find $\epsilon^*$ sufficiently small such that 
\begin{align} 
D\Psi \mathbf{0} [h] - h < -\epsilon^*.
\end{align}
Now following the approximation in Section~\ref {proof:strat:final:frac}, we may find a sequence $\{\cS_{L}^l\}_{l \in [L]}$ of partitions for the type space $\cS$, and coupling kernels $\kappa^+_L\geq \kappa $, such that  $$\norm{D\Psi_{\kappa^+_L} \mathbf{0} [h] - D\Psi \mathbf{0} [h]}_\infty < \frac{\epsilon^*}{2}$$ for $L\geq L_0$. For each $L$, let $h_L (s) = \sum_{l \in [L]} I_{\cS_L^l}(s)\sup_{x \in \cS_L^l} h(x)$. By continuity of the function $h\mapsto D \Psi\mathbf{0} [h]$ there exists $L_1>0$ such that  $\norm{D\Psi_{\kappa^+_L} \mathbf{0} [h] - D\Psi_{\kappa^+_L} \mathbf{0} [h_L]}_\infty <\frac{\epsilon^*}{4}$ for $L\geq L_1$. It follows that $\norm{D\Psi_{\kappa^+_L} \mathbf{0} [h_L] - D\Psi \mathbf{0} [h]}_\infty < \frac{3\epsilon^*}{4}$ for $L \geq \max \{L_0,L_1 \}$. This implies that 
\begin{align} 
D\Psi_{\kappa^+_L} \mathbf{0} [h_L] - h_L < -\epsilon^*/4,
\end{align}
Now by the embedding results in the previous section, it follows from \eqref{eq:dtc_deriv} for the discrete counterpart $\kappa^{d+}_L$ of the kernel $\kappa^{+}_L$ that
\begin{align}
\sum_{l' \in [L]} w^{l'} \frac{\d \nu_0^l(\mathbf{0}_L)}{\d z^{l'}}
    &= - w^l +  \(\sum_{l' \in [L]} \kappa^+_L(l',l) w^{l'}\) \nu_1^l(\mathbf{z}_L) \\
    &= - \int_{\cS_L^l} h_L(s) \dd \mu(s) +  \int_{\cS_L^l} D \Psi_{\kappa_L} \mathbf{0} [ h_L ] (s) \dd \mu(s).
\end{align}
with $w^l=h_L(s) \mu (\mathcal{S}_L^l)$ for $s\in \mathcal{S}_L^l$ and where $\mathbf{0}_L=(0,\dots,0)$. Now along the lines of the proof of Proposition~\ref{prop:remain}, there exists C>0 s.t. 
\begin{align}
n^{-1}|\mathbb{D}(G(n,\cV_L^+))|  \xrightarrow[\nu(0, \cS) \rightarrow 0]{p} 0,
\end{align} which finishes the proof since 
\begin{align}
\mathbb{P}(|\mathbb{D}(G(n,\cV))| \le |\mathbb{D}(G(n,\cV_L^+))|)=1.    
\end{align}
\end{proof}

\section{Applications}\label{case:study}
In Section \ref{sec:con}, we have established results that allow us to determine for $n\rightarrow \infty$ the final fraction of infected vertices at the end of the percolation process based on the least fixed point $\hat{f}$ of $\Psi_{\kappa}$. In this section, we will provide a numerical case study and compare the theoretical result of Theorem~\ref{thm:fp} with the outcome of simulations for random graphs of moderate size. 

We first propose an algorithm to approximate the least fixed point $\hat{f}$ with neural networks. For the simulations, we provide an additional algorithm that for each random sample of the graph allows us to determine the result of the percolation process in a computationally very efficient way. Then we specify a continuous kernel that satisfies Assumption~\ref{ass:kernel} and a vertex sequence that satisfies the regularity Assumption~\ref{ass:regularity}. For each sample from the random graph, we determine the exact result of the contagion process and compare it with the asymptotic $n\rightarrow \infty$ results that we obtain in this paper. For numerical convenience, in this section, we consider the type space $\cS$ to be a compact subset of the real line $\R$. All the codes are available in the GitHub repository (\href{https://github.com/jmlinx/BPRG}{https://github.com/jmlinx/BPRG}).

\subsection{Neural network fixed point approximation}
An $M$-layer ($M\in \N$) neural network $f^{NN}:\R^{d_x} \rightarrow \R^{d_y}$ maps $d_x$ dimensional input to $d_y$ dimensional output and has the form 
\begin{align}
f^{NN} =  \sigma_M \circ W_M \dots \circ \sigma_1 \circ W_1,
\end{align}
where $\circ$ denotes component-wise composition, for $1 \le m \le M$, $d_0 = d_x$, and $d_M=d_y$, $W_m: \R^{d_{m-1}} \rightarrow \R^{d_m}$ are composable affine maps such that
\begin{align}
W_m(\mathbf{x}) = \bs{\theta}_m \mathbf{x} + \mathbf{b}_m,  &&
\bs{\theta}_m \in \R^{d_{m} \times d_{m-1}}, &&
\mathbf{b}_m \in \R^{d_{m}},
\end{align}
and $\sigma_m: \R^{d_m} \to [0,1]^{d_m}$ are some activation functions. The width $d_m$ for $1 \le m \le M$ of the $m$-th layer of the neural network is a hyperparameter to be specified.

Let $\bs{\theta} = \{\bs{\theta}_m, \mathbf{b}_m\}_{m=1}^M$ collect the parameters, and we denote by $f^{NN}(\cdot \ ; \bs{\theta})$ the neural network with parameters $\bs{\theta}$.
The following proposition ensures that the fixed point $\hat{f}$ can be approximated by a neural network.

\begin{proposition}
Let $\cS$ be a compact subset of $\R$. Assume that the least fixed point $\hat{f}$ of $\Psi_{\kappa}$ exists. Then for any $\epsilon >0$, there exists a neural network $f^{NN}(\cdot;\bs{\theta}): \cS \rightarrow [0,1]$, such that
\begin{align}
    \sup_{s\in\cS} \norm{f^{NN}(s; \bs{\theta}) - \hat{f}(s)} < \epsilon, &&  \sup_{s\in\cS} \norm{f^{NN}(s; \bs{\theta}) - \Psi_{\kappa}f^{NN}(s; \bs{\theta})} < \epsilon.
\end{align}
\end{proposition}
\begin{proof}
Recall that the functions $f \in \cF_1^c$ are continuous and that $\hat{f} \in \cF_1^c$. Then the existence of $f^{NN}$ such that the first inequality holds is a direct consequence of the universal approximation theorem by \cite{cybenko1989approximation}. The second inequality can be ensured by possibly further reducing $\sup_{s\in\cS} \norm{f^{NN}(s; \bs{\theta}) - \hat{f}(s)}$, the fixed point property of $\hat{f}$, and by the continuity of $\Psi_{\kappa}$.
\end{proof}

\noindent We propose to approximate the least fixed point by the following procedure. The corresponding pseudocode is provided in Algorithm~\ref{alg:nn}.
\begin{enumerate}
    \item For $j=0$, initialize a neural network $f^{NN}(\cdot;\bs{\theta}_0)$ with parameter $\bs{\theta}_0$.
    \item
    Evaluate the objective function:
    \begin{align}\label{eq:obj}
        J(\bs{\theta}_j) = \frac{1}{N}\sum_{i=1}^N \abs{f^{NN}(x_i; \bs{\theta}_j) - \Psi_{\kappa}^{NN}f_j^{NN}(x_i; \bs{\theta}_j)} + \gamma \int_\cS f^{NN}(s;\bs{\theta}_j)\dd \mu(s),
    \end{align}
    where $0 < \gamma < 1$ is a parameter and $\{x_0, \dots, x_N\}$ are data points sampled from $\cS$. The first term is the mean absolute error between the neural network $f^{NN}$ and $\Psi_{\kappa}^{NN}f^{NN}$ (to be defined below) over the data points.  For small $\gamma$, minimizing the first term ensures convergence to a fixed point. The second term approximates the integrated value of the fixed point. Adding it to the loss function penalizes large values and steers the algorithm to converge towards the least fixed point. The coefficient $\gamma$ should be chosen small enough such that the algorithm will prioritize the convergence to zero of the first term.
    \item Evaluate the gradient $\nabla_{\bs{\theta}} J(\bs{\theta}_j)$ and obtain the new parameter $\bs{\theta}_{j+1}$ with gradient descent
    \begin{align}\label{eq:descent}
        \bs{\theta}_{j+1} = \bs{\theta}_j - \alpha_j \nabla_{\bs{\theta}} J(\bs{\theta}_j), && j \leftarrow j+1,
    \end{align}
    where $\alpha_j$ is the learning rate, and $\nabla_{\bs{\theta}} J(\bs{\theta}_j)$ is calculated by automatic differentiation (\cite{baydin2018automatic}).
    \item Repeat step 2 and step 3 until the stopping criterion $J(\bs{\theta}_{\cdot}) < \epsilon$ is reached for some $\epsilon > 0$. We obtain an approximation $\hat{f}_{\epsilon}^{NN}$ to the least fixed point $\hat{f}$.
\end{enumerate}

\begin{algorithm}[t]
\caption{Neural network approximation of the least fixed point}
\begin{algorithmic}[1]
\Require{Type space $\cS$, kernel $\kappa$, measures $\nu$ and $\mu$.}
\State{Initialization: Initialize neural network $f^{NN}$ with parameter $\bs{\theta}_0$, and sample $N$ points $x_1, \cdots, x_N\ \in \cS$ uniformly. Choose a small stopping criterion $\epsilon>0$.}
\For{$j = 0, 1, \dots$}
    \State{Calculate the objective function $J(\bs{\theta}_j)$ as defined in equation~\eqref{eq:obj}.}
    \State{\textbf{Update} parameters with gradient descent $\bs{\theta}_{j+1} = \bs{\theta}_j - \alpha_j \nabla_{\bs{\theta}} J(\bs{\theta}_j)$ as in equation~\eqref{eq:descent}.}
    \If{$J(\bs{\theta}_j) \le \epsilon$}
        \State{Break.}
    \EndIf
\EndFor
\end{algorithmic}
\label{alg:nn}
\end{algorithm}

\begin{remark}
In the above proposition and algorithm, we explained how a single-layer neural network approximates the least fixed point with a simple gradient descent algorithm. In practice, many specifications can be adapted to speed up the learning process, such as using neural networks multiple layers rather than a single layer, and updating the parameters with more advanced methods such as the Adam algorithm (\cite{kingma2014adam}) instead of standard gradient descent.
\end{remark}
\begin{remark}
Calculating $\Psi_{\kappa}f^{NN}(\cdot)$ requires $\Lambda_{\kappa}f^{NN}(\cdot) = \int_{s \in \cS} \kappa(s, \cdot) f^{NN}(s) \dd \mu(s)$. The integral can be approximated numerically with Riemann sums in the following manner. Consider a data grid of $\cS$ with $M_\cS$ points, denoted by $\cX_\cS = \{s_0, s_1, \dots, s_{M_\cS}\}$. Then calculate
\begin{align}
\Lambda_{\kappa}^{NN} f^{NN}(\cdot; \bs{\theta})
    := \sum_{m=0}^{M_\cS-1} \kappa(s_m, \cdot) f^{NN}(s_m; \bs{\theta})  \mu([s_m, s_{m+1}]) 
    \approx \Lambda_{\kappa} f^{NN}(\cdot; \bs{\theta}),
\end{align}
where the $\Lambda_{\kappa}^{NN}$ operator approximates the integral operator $\Lambda_{\kappa}$ in equation \eqref{eq:operator} by a Riemann sum. We denote by $\Psi_{\kappa}^{NN}$ the operator obtained by replacing the operator $\Lambda_{\kappa}$ in the definition of $\Psi_{\kappa}$ by $\Lambda_{\kappa}^{NN}$. For a sufficiently fine grid it follows that $\Psi_{\kappa}^{NN}f^{NN}(\cdot; \bs{\theta}) \approx \Psi{\kappa}f^{NN}(\cdot; \bs{\theta})$. The integral $\int_\cS \hat{f}_{\epsilon}^{NN}(s;\bs{\theta})\dd \mu(s)$ approximates the $n\rightarrow \infty$ limit of the final fraction of infected vertices and can be calculated via a simple Riemann sum.
\end{remark}

\subsection{Simulations}

To illustrate numerically the convergence of the simulation results to the theoretical results of Theorem~\ref{thm:fp}, we generate random graph realizations, compute the final fraction of infected vertices in each realization, and compare the resulting empirical distribution with the fixed point $\hat{f}$ and its integral $\int_{\cS} \hat{f} \dd \mu$. To this end, we first introduce a new Monte Carlo algorithm for simulating random graphs and computing the outcome of the percolation process. The algorithm is based on straightforward matrix operations and can be implemented efficiently. The details are given in Algorithm~\ref{alg:mc}.

\begin{algorithm}[t]
\caption{Simulation of the bootstrap percolation process}
\begin{algorithmic}[1]
\Require{Type space $\cS$, kernel $\kappa$, maximum threshold $K$, vertex number $n$, measures $\nu$.}
\State{Initialize the type vector $\bL = \(\bL_i\)_{i\in[n]} \in \cS$ and the threshold vector $\bK^0 = \(\bK_i^0\)_{i\in[n]} \in [K]$ where $(\bK_i, \bL_i) \sim \nu$.}
\State{Initialize the newly infected indicator vector $\bI^0 = \(\bI^0_i\)_{i\in[n]} = \(I_{\{0\}}(\bK^0_i)\)_{i\in[n]}$}
\State{Initialize the explored infection indicator vector $\bE^0 = \(\bE^0_i
\)_{i\in[n]} = \(0\)_{i\in[n]} $}
\For{$m = 1, \cdots, M$, where $M$ is the number of realizations,}
    \State{Sample the binary adjacency matrix $\bA = \(\bA_{ij}\)_{i,j\in[n]} \in \{0,1\}$ by $\bA_{ij} = I_{\{U_{ij} \le \kappa(\bL_i, \bL_j)/n\}}$, where $U_{ij} \sim \text{Uniform}(0,1)$ and $I$ denotes the indicator.}
    \For{$p = 0, 1, \dots$} 
        \State{Update the thresholds: $\bK^{p+1} = \(\(\bK^{p} - \bA^T \bI^{p}\)_i\vee 0 \)_{i \in [n]}$.}
        \State{Update the explored infection: $\bE^{p+1} = \bE^{p} + \bI^{p}$.}
        \State{Update the new infection: $\bI^{p+1} = \(I_{\{0\}}(\bK^{p+1}_i)\)_{i\in[n]} - \bE^{p+1}$.}
        \If{no new infection, $\bI^{p+1} = \boldsymbol{0}$, }
            \State{Break.}
        \EndIf
    \EndFor
\State{Compute the final number of infected vertices of type $s$: $\mathbb{D}_m^p(s)=\sum_{i\in[n]}I_{\{(0, s)\}}(\bK^p_i, \bL_i)$.}
\EndFor
\end{algorithmic}
\label{alg:mc}
\end{algorithm}

Let $\mathbb{D}_m(\cS)/n$ be the final fraction of infected vertices for simulation $m$. By the results in Theorem~\ref{thm:fp} we expect this fraction to be close to $\int_{\cS} \hat{f}(s)\dd\mu(s)$ for all $m$. In addition, the empirical result in the next section shows that, for a subset $A \subset \cS$, the simulated final fraction of infected vertices with type in $A$ is approximately $\int_A \hat{f}(s)\dd\mu(s)$. Based on this observation, we believe that Theorem~\ref{thm:fp} can be extended to determine the final fraction of infected vertices with type in a certain subset $A \subset \cS$.

\subsection{Numerical results}
For the numerical experiments, we consider the following three kernel functions defined on the type space $\cS=[0,1]$, with $\mu$ being the uniform measure on $\cS$:
\begin{equation} \label{eq:kernels}
\begin{aligned} 
\kappa_1(x,y) &= \frac{10\sqrt{x^2+y^2}}{1+\sqrt{\abs{x-y}}}, \\
\kappa_2(x, y) &= \frac{5(e^{x+0.5y} - 1)}{1 + \sqrt{\abs{x-y}}}, \\
\kappa_3(x, y) &= \frac{x+y}{\abs{x-y} + \abs{x-1} + \abs{y-1}}.
\end{aligned}
\end{equation}
Figure~\ref{fig:kernels} visualizes the three kernels. All three kernels, $\kappa_1$, $\kappa_2$, and $\kappa_3$, exhibit the desired properties of \emph{heterogeneity} and \emph{blocking}. Heterogeneity arises from the fact that the connection probability varies across vertex types, as illustrated by the non-flat surfaces in the figures. Blocking exhibits the phenomena that vertices of similar types have higher connection probabilities, indicated by the ridge structure in each surface. In particular, $\kappa_2$ exhibits stronger heterogeneity than $\kappa_1$ due to its exponentially increasing numerator and also introduces asymmetry through the term $x + 0.5y$. The kernel $\kappa_3$ behaves even more extreme, exploding at the point $(1,1)$.
\begin{figure}[h]
\begin{center}
\begin{subfigure}{0.32\linewidth}
  \includegraphics[trim=4.5cm 3cm 4.5cm 1cm, clip, width=\linewidth]{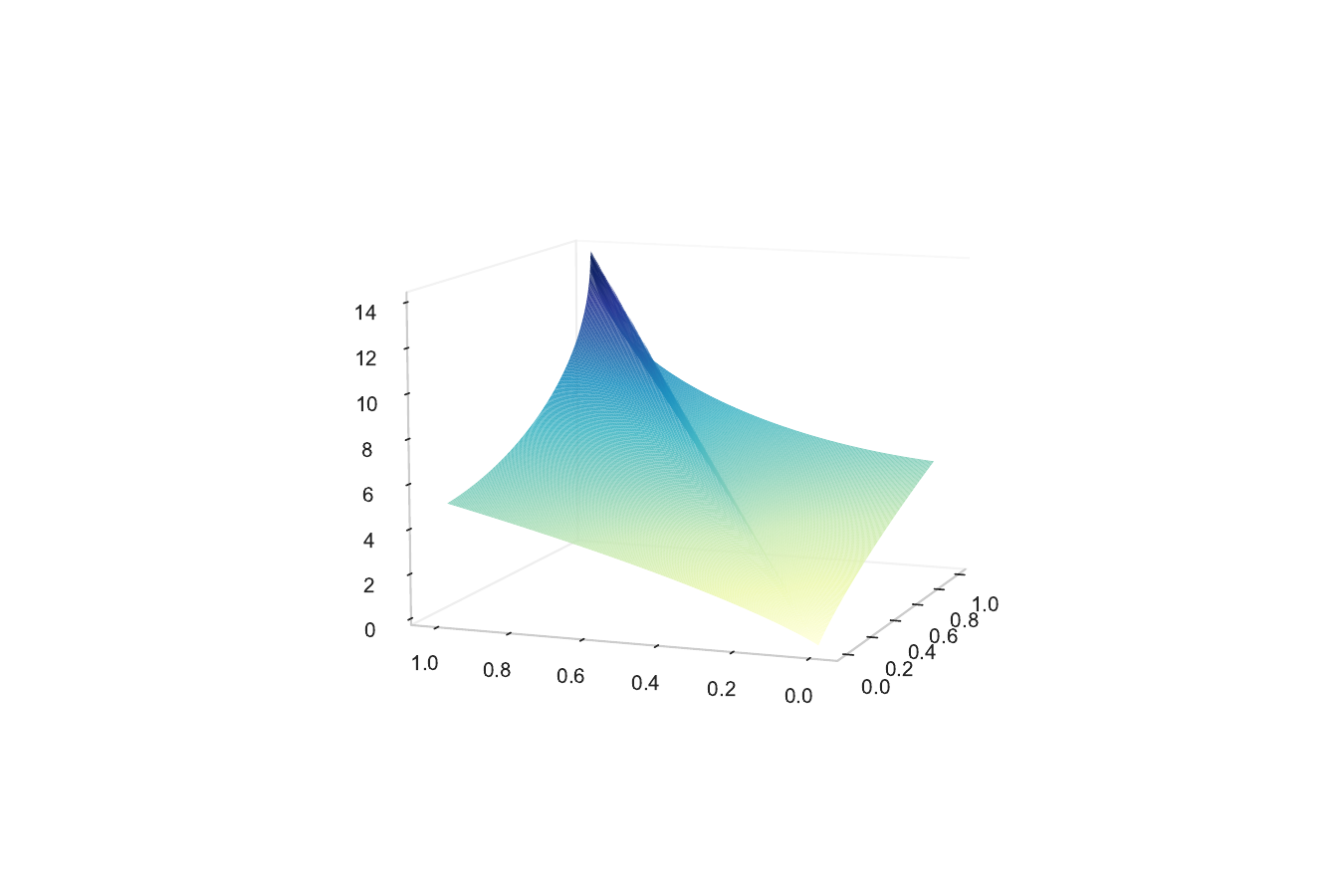}
  \caption{}
\end{subfigure}
\begin{subfigure}{0.32\linewidth}
  \includegraphics[trim=4.5cm 3cm 4.5cm 1cm, clip, width=\linewidth]{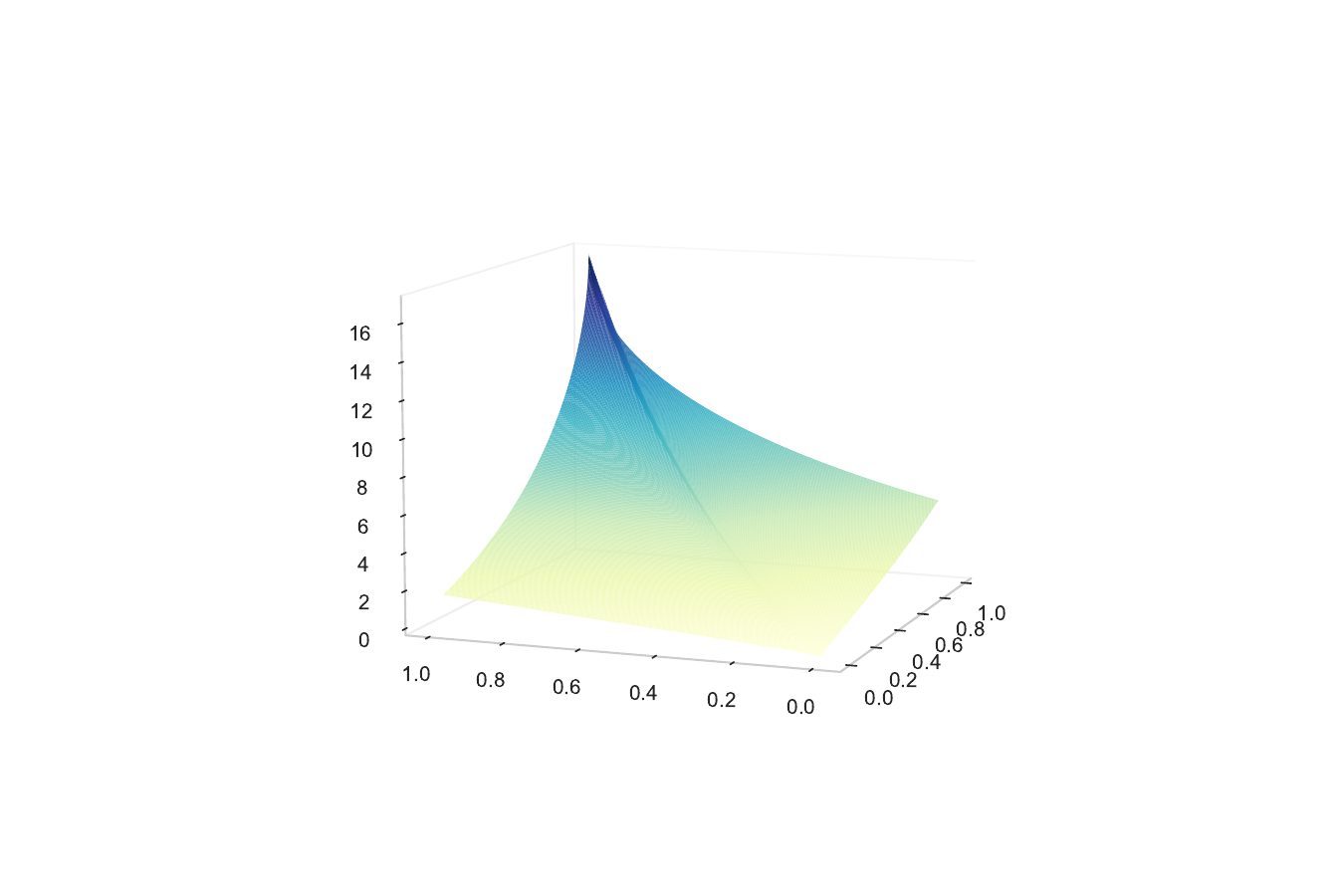}
  \caption{}
\end{subfigure}
\begin{subfigure}{0.32\linewidth}
  \includegraphics[trim=4.5cm 3cm 4.5cm 1cm, clip, width=\linewidth]{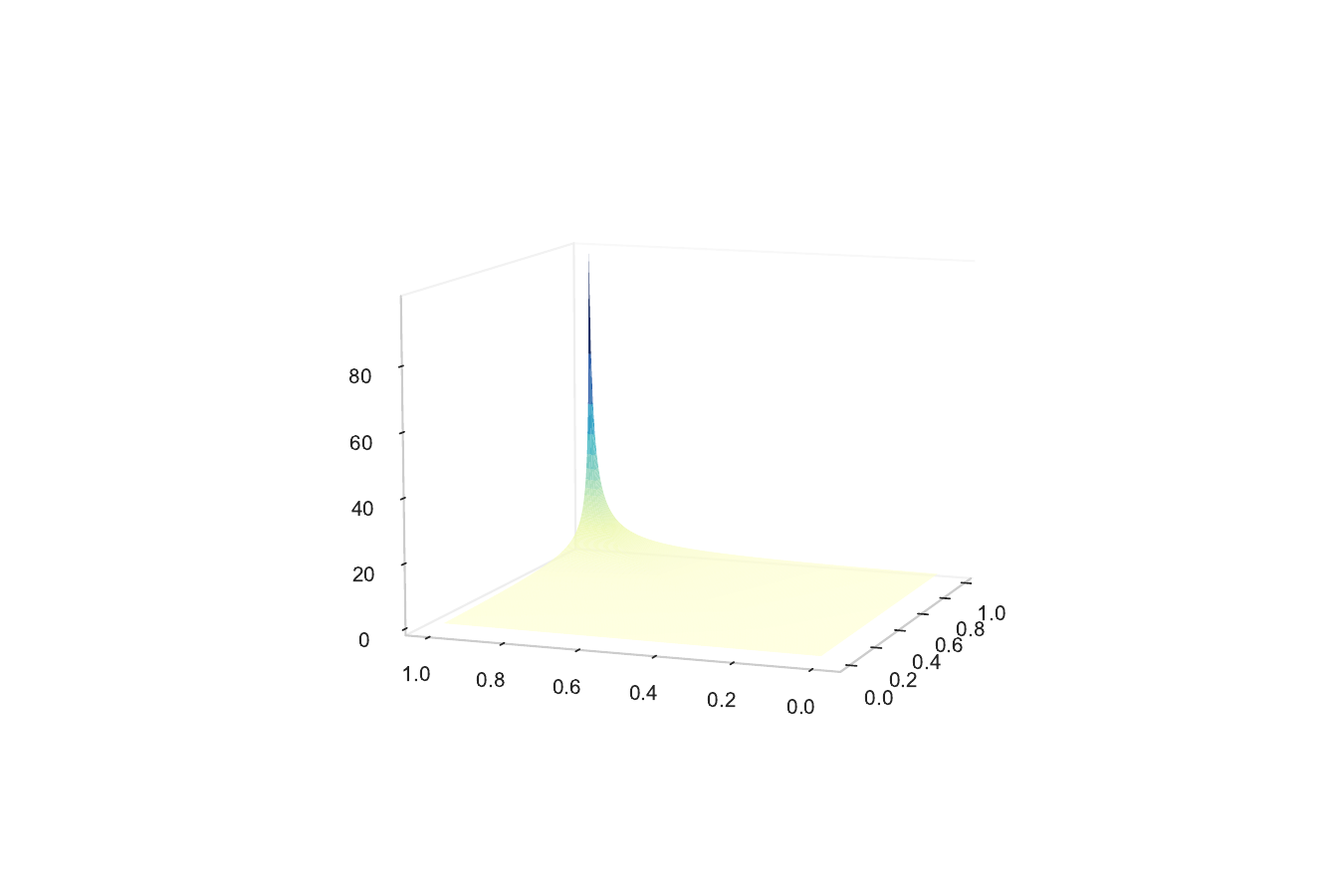}
  \caption{}
\end{subfigure}
\caption{
Kernel functions in Equation~\eqref{eq:kernels}: (a) $\kappa_1$, (b) $\kappa_2$, and (c) $\kappa_3$.
}
\label{fig:kernels}
\end{center}
\end{figure}
The bootstrap percolation is considered in a setting where $10\%$ of vertices are initially infected and all remaining vertices have threshold $2$. The initial infection is uniformly distributed among the vertices and independent of the type, which corresponds to
\begin{align}
\nu(k, [a,b]) = I_{\{k=0\}}\frac{1}{10}\abs{b-a} + I_{\{k=2\}}\frac{9}{10}\abs{b-a}.
\end{align}

We first report the numerical results for $\kappa_1$. The fixed point is computed using Algorithm~\ref{alg:nn}. The neural network consists of two hidden layers with $20$ nodes each, followed by hyperbolic tangent activation function. The parameters are optimized with Adam algorithm. In Figure~\ref{fig:func_lips}(a), the overlap of the functions $f^{NN}$ and $\Psi_{\kappa}^{NN} f^{NN}$ indicates that the fixed point has been accurately identified. We compare the fraction of infected vertices obtained from the fixed point with the fraction obtained by Monte Carlo simulation via Algorithm~\ref{alg:mc}. We generate $M=1000$ random graphs with $n=3000$ vertices and determine for each sample the result of the bootstrap percolation process. The type space is discretized into $L=1000$ equal bins, and the final fraction of infected vertices is calculated for each type. Figure~\ref{fig:func_lips}(b) compares the scatter plot of simulated fractions $f^{MC}$ with the fixed point function $f^{NN}$. Each point corresponds to the infected fraction (scaled by the width of the bin) of one of the $L$ types of one of the $M$ simulations. The simulated $f^{MC}$ scatters align with the function $f^{NN}$, indicating a strong agreement between the simulations and the theoretical asymptotic result given by Theorem~\ref{thm:fp}.

\begin{figure}[h]
\begin{center}
\begin{subfigure}{0.48\linewidth}
  \includegraphics[clip, width=\linewidth]{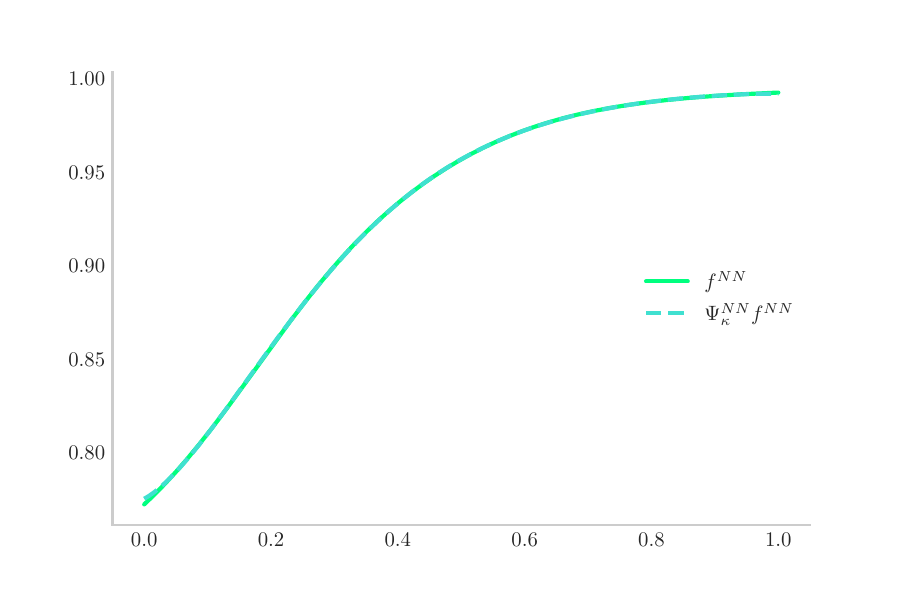}
  \caption{}
\end{subfigure}
\begin{subfigure}{0.48\linewidth}
  \includegraphics[clip, width=\linewidth]{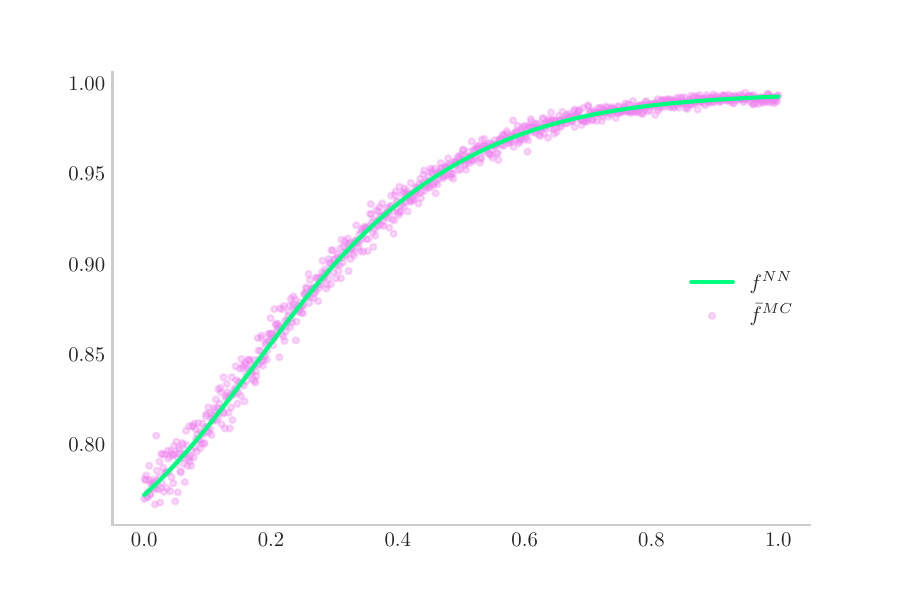}
  \caption{}
\end{subfigure}
\caption{
Fixed point, its image, and simulated final fractions for $\kappa_1$. (a) Fixed point $f^{NN}$ approximated with Algorithm~\ref{alg:mc} and its image $\Psi_{\kappa}^{NN} f^{NN}$. (b) $f^{NN}$ and simulated final fractions $f^{MC}$. Horizontal axis: vertex type; vertical axis: function value or infected fraction.
}
\label{fig:func_lips}
\end{center}
\end{figure}

Figures~\ref{fig:func23}(a) and (b), respectively, report the corresponding results for $\kappa_2$ and $\kappa_3$, following the same procedure as used to generate Figure~\ref{fig:func_lips}. For conciseness, we overlay the neural fixed point, its image, and the simulated fractions in each plot. In both cases, the theoretical fixed point aligns well with the simulation results. At first glance, the scatter in Figure~\ref{fig:func23}(b) appears more dispersed than in Figure~\ref{fig:func_lips}(b) and Figure~\ref{fig:func23}(a). This effect is due to lower connection probabilities and the resulting smaller final infected fraction for the kernel $\kappa_3$. Taking into account the scale of the vertical axis, we observe that the absolute deviation of the random outcomes from the $n\rightarrow \infty$ result is less than 2\% across all three kernels. This highlights that the asymptotic results derived in Theorem~\ref{thm:fp} approximate the outcome of the infection process very well even for random graphs of moderate size.

\begin{figure}[h]
\begin{center}
\begin{subfigure}{0.48\linewidth}
  \includegraphics[clip, width=\linewidth]{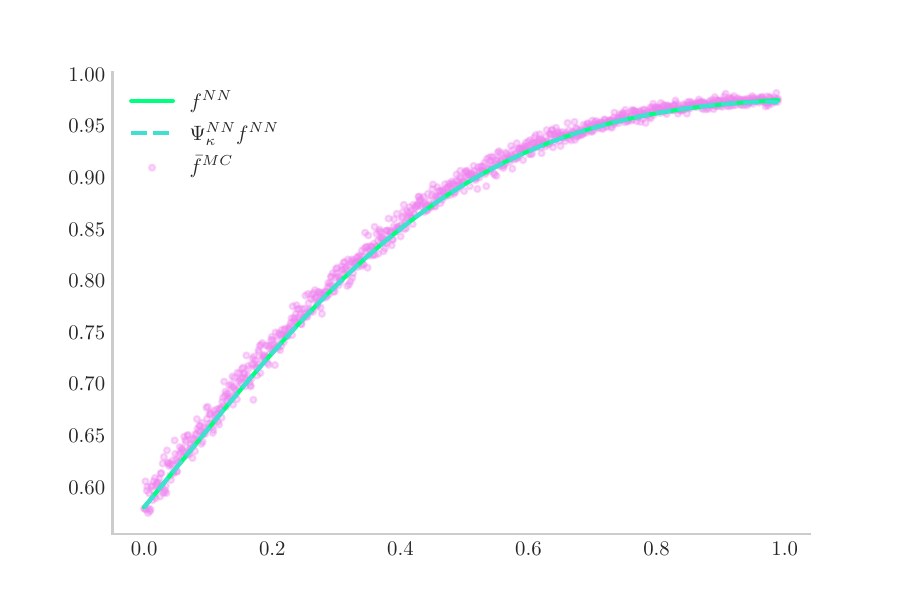}
  \caption{}
\end{subfigure}
\begin{subfigure}{0.48\linewidth}
  \includegraphics[clip, width=\linewidth]{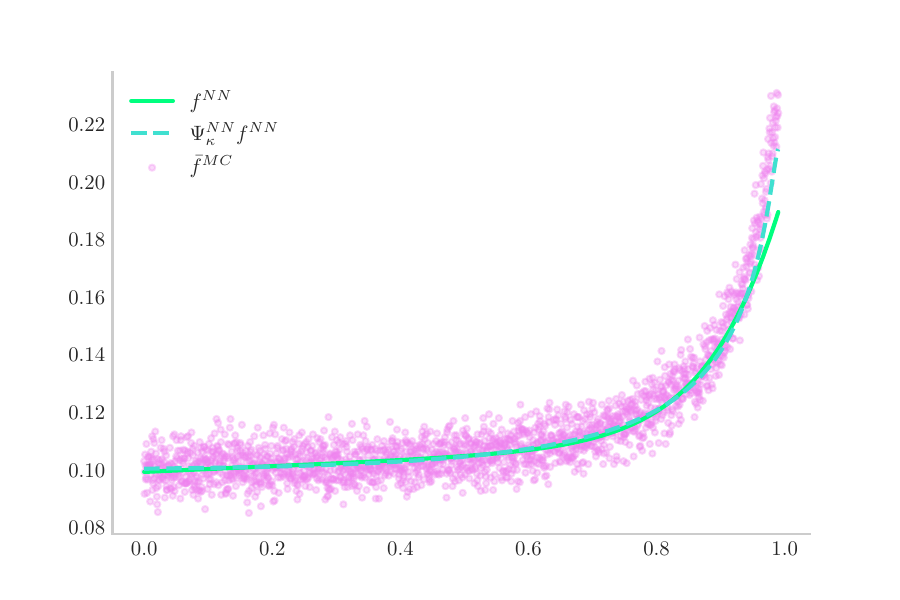}
  \caption{}
\end{subfigure}
\caption{
Fixed point, its image, and simulated final fractions for (a) $\kappa_2$ and (b) $\kappa_3$. Horizontal axis: vertex type; vertical axis: function value or infected fraction.
}
\label{fig:func23}
\end{center}
\end{figure}

Finally, we report the distribution of simulated final fractions of infected vertices for $\kappa_1$; the results for $\kappa_2$ and $\kappa_3$ are similar and thus omitted. We compute the integral of the fixed point as $\hat{\tau} := \int_{\cS} \hat{f}\dd\mu \approx \tau^{NN} :=\int_\cS f^{NN} \dd \mu=0.9273$. By Theorem~\ref{thm:fp}, the final fraction of infected vertices in a large graph is therefore expected to be $92.73\%$. Figure~\ref{fig:ci1}(a) shows the distribution of final fractions obtained from the $M=1000$ simulations with $n=3000$ vertices. The results tightly concentrate between $90\%$ and $94\%$ and are centered near the theoretical value $92.73\%$. The average simulated final fraction is $92.70\%$.

Since $\hat{\tau}$ quantifies the asymptotic limit as $n \rightarrow \infty$, we further perform simulations with graph size $n=200, 400, \dots, 10000$, using $M=1000$ realizations for each size. Figure~\ref{fig:ci1}(b) reports the average simulated final fractions of infected vertices $\tau^{MC}$ together with its $95\%$ confidence band. As expected, the confidence band narrows as the graph size increases.

\begin{figure}[h]
\begin{center}
\begin{subfigure}{0.48\linewidth}
  \includegraphics[clip, width=\linewidth]{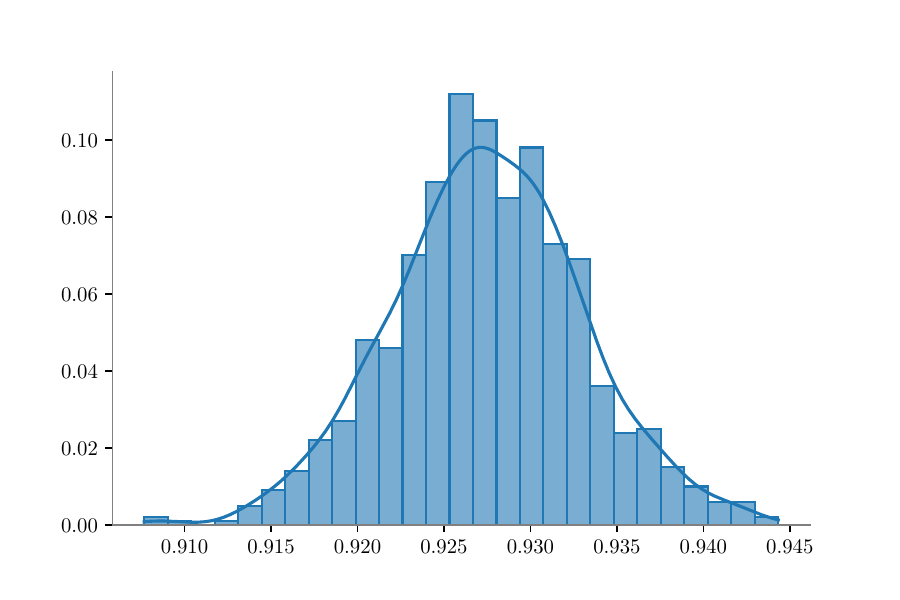}
  \caption{}
\end{subfigure}
\begin{subfigure}{0.48\linewidth}
  \includegraphics[clip, width=\linewidth]{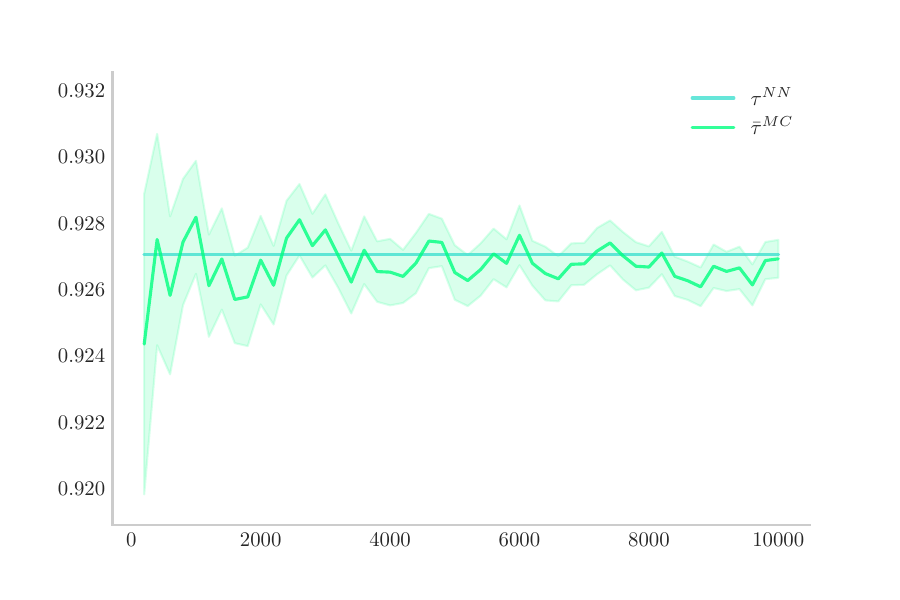}
  \caption{}
\end{subfigure}
\caption{(a) Histogram and kernel density estimation of $M=1000$ simulated final fractions of infected vertices for graph of size $n=3000$. Horizontal axis: final fraction; vertical axis: frequency. (b) average simulated final fractions of infected vertices and $95\%$ confidence band for graphs with sizes $n = 200, 400, \dots, 10000$, based on $M=1000$ simulations per size. Horizontal axis: graph size; vertical axis: final fraction.
}
\label{fig:ci1}
\end{center}
\end{figure}

In conclusion, through numerical experiments with the three kernel functions, we validate that the final fraction of infected vertices formulated by the fixed point in Theorem~\ref{thm:fp} and computed via Algorithm~\ref{alg:nn} agrees closely with Monte Carlo simulation results obtained using Algorithm~\ref{alg:mc} already for random graphs of moderate size. Moreover, the fixed point approach is computationally more efficient for obtaining an approximation of the default fraction compared to a Monte Carlo simulation and avoids the sampling error.

\appendix
\section{Random graphs with a finite number of vertex types} \label{sec:dis}
In this section, we study the FTBT graph with a finite number of vertex types and finite thresholds introduced in Definition \ref{def:dis_sys} and we prove Proposition \ref{prop:fp_z}. Let $K \in \N$ and we assume that $r_i(n) \in [K] \cup \{0\}$ for all $i\in [n]$. We denote the vertex sequence by $\cV_{L,K}^d:=(\kappa^d_L,\nu, \bm{s} (n),\bm{r} (n))$ with $K$ explicit in the subscript and the corresponding graph by $G(n,\cV_{L,K}^d)$. For a neater notation, across this section we use $\kappa$ instead of $\kappa^d$ but it should be kept in mind that we have a discrete kernel.

Instead of exploring the final set of infected vertices $\mathbb{D} (G(n,\cV_{L,K}^d))$ by generations as described in Section~\ref{sec:con}, we use a \textit{sequential process} that results in the same set of infected vertices. The idea is that in each step we only explore the effect on the system triggered by one infected vertex. For this, at the beginning of each iteration, we uniformly select an infected vertex from the set of unexplored infected vertices. Denote this vertex by $i$. We then explore all edges that vertex $i$ sends to uninfected vertices and reduce the threshold value of each receiving vertex by $1$. If the threshold of a receiving vertex reaches $0$, then this vertex is infected and we include it in the set of unexplored infected vertices. Then, after we have updated the threshold of all vertices that received an edge from vertex $i$, we remove the vertex $i$ from the set of unexplored infected vertices. The effect of this vertex has now been explored. We repeat the iteration until there are no more unexplored infected vertices left. 

{\em Sequential exploration process:} To formulate the sequential process, we introduce $t \in \mathbb{N}_0$ to index the $t$-th iteration. We first define for $l\in [L]$ and $k\in [K]\cup \{0\}$ the sets 
\begin{align}
U_k^l(0) := \{i \in \[n\] | s_i = l, r_i(0) = k\},  &&
u_k^l(0) := |U_k^l(0)|,
\end{align}
where $U_k^l(0)$ is the (initial) set of type $l$ vertices with threshold $k$. Note that there is a slight abuse of notation in relation to the sets $U^{(n)}_k(A)$ defined in Assumption~\ref{ass:regularity}. We further denote by $u_k^l(0)$ their size. Moreover, by Definition \ref{def:dis_sys}, it holds that the initial proportion $u_k^l(0)/n$ converges to a limit $\nu_k^{l}$:
\begin{align}\label{conv:finite:type}
u_k^l(0)/n \xrightarrow[n\rightarrow \infty]{} \nu_k^{l},
\end{align}
where $\sum_{l \in [L]}\sum_{k=0}^K\nu_k^{l}=1$.

We keep track of the evolution of the thresholds of the vertices throughout the iterations. For any vertex $i \in [n]$ at iteration $t$, its threshold $0 \le r_i(t) \le r_i(0)$ equals to $r_i(0)$ minus the number of edges it has received from explored infected vertices. For each $t\geq 1$ we therefore group vertices by their types and (current) thresholds:
\begin{align}
U_k^l(t) := \{i \in \[n\] | s_i = l, r_i(t) = k\},  &&
u_k^l(t) := |U_k^l(t)|,
\end{align}
We also use $U_k(t)$ to denote all the vertices with threshold $k$ at iteration $t$, and $u_k(t)$ to denote their size:
\begin{align}
U_k(t) = \bigcup_{l \in [L]} U_k^l(t),  &&
u_k(t) = \sum_{l\in[L]}u_k^l(t).
\end{align}
We stress that for $t>0$, the sets $U_k^l(t), U_k(t)$ and their size $u_k^l(t)$ and $u_k(t)$ are actually random quantities. Note that after each iteration, we drop the selected and explored infected vertex from the set $U_0(t)$, i.e. if $i\in U_0(t)$ is selected, then $i \notin U_0(t')$ for all $t'>t$. It holds that $u_0(t+1) \ge  u_0(t) - 1$.

Additionally, we track the random sets
\begin{align}
\mathbb{D}(t, G(n,\cV_{L,K}^d))
    &:= \{\text{infected vertices explored up to iteration } t\}, \\
\mathbb{D}(G(n,\cV_{L,K}^d))
    &:= \{\text{infected vertices at the end of the process}\}.
\end{align}
By the nature of the sequential process we know that $u_0(0) \le |\mathbb{D}(G(n,\cV_{L,K}^d))|\le n$.

As explained above, at iteration $t$, we uniformly select an unexplored infected $i\in U_0(t)$. Its threshold is $r_i(t)=0$. The probability that the type $s_i$ of vertex $i$ is equal to $l$ is given by $\tfrac{u_0^l(t)}{u_0(t)}$. Then, we reveal all the vertices that receive an edge from $i$. If a vertex $j$ with type $s_j$ and threshold $r_j(t)=k>0$ receives an edge from $i$, which happens with probability $\kappa(s_i, s_j)/n$, then we set $r_j(t+1) := k - 1$. Henceforth in the next iteration $t+1$, $j$ has moved to the set $U_{r_j(t+1)}^{s_j}(t+1)=U_{k-1}^{s_j}(t+1)$. After all other vertices receiving an edge from $i$ are examined, we consider the vertex $i$ as explored and add it to the set $\mathbb{D}(t, G(n,\cV_{L,K}^d)) = \mathbb{D}(t-1, G(n,\cV_{L,K}^d)) \cup \{i\}$ and remove it for future explorations: $i \notin U_0(t+1)$.

Let $h(t):= (u_k^l(t))_{k \in [K]\cup \{0\}, l \in [L]}$ describe the state of the entire percolation process at iteration $t$. According to the algorithm discussed above, for each type $l \in [L]$, we can write down the expected change of the sizes of the vertex sets by the following equations
\begin{equation}\label{eq:sys_all}
\begin{aligned} 
\Eb[u_0^l(t+1) - u_0^l(t) | h(t)]
    &= - \frac{u_0^l(t)}{u_0(t)} + \(\sum_{l' \in [L]} \frac{u_0^{l'}(t)}{u_0(t)}  \frac{\kappa(l',l)}{n}\)u_1^{l}(t), \\
 \Eb[u_k^l(t+1) - u_k^l(t) | h(t)]
    &=  \(\sum_{l'\in[L]} \frac{u_0^{l'}(t)}{u_0(t)} \frac{\kappa(l',l)}{n}\) \(u_{k+1}^{l}(t) - u_k^{l}(t)\), \ \forall k \in [K-1]\\
\Eb[u_K^l(t+1) - u_K^l(t) | h(t)]
    &= - \(\sum_{l'\in[L]} \frac{u_0^{l'}(t)}{u_0(t)} \frac{\kappa(l',l)}{n}\) u_K^{l}(t).
\end{aligned}
\end{equation}
Consequently, the expected change of the sizes of the entire vertex sets across each threshold can be summarized by
\begin{equation}
\begin{aligned}
\Eb [u_0(t+1) - u_0(t) | h(t)]
    &= \sum_{l \in [L]} \Eb[u_0^l(t+1) - u_0^l(t) | h(t)], \\
\Eb [u_k(t+1) - u_k(t) | h(t)]
    &= \sum_{l \in [L]} \Eb[u_k^l(t+1) - u_k^l(t) | h(t)], \ \forall k \in [K-1], \\
\Eb [u_K(t+1) - u_K(t) | h(t)]
    &= \sum_{l \in [L]} \Eb[u_K^l(t+1) - u_K^l(t) | h(t)].
\end{aligned}
\end{equation}

To understand the system of equations \eqref{eq:sys_all}, let us focus on vertices of type $l \in [L]$ and the step from iteration $t$ to $t+1$. For the first equation, the expected change of the number of infected vertices of type $l$ consists of two parts. First, one already infected vertex is picked uniformly at $t$ and this vertex is of type $l$ with probability $\frac{u_0^l(t)}{u_0(t)}$. This vertex is explored and then excluded from the set $U_0^l(t+1)$, which corresponds to the term $-\frac{u_0^l(t)}{u_0(t)}$. The second positive term accounts for the newly infected vertices of type $l$ in iteration $t$, which are those vertices in $U_1^l(t)$ that get infected by receiving an edge from the currently explored vertex. As the sequential process reduces thresholds by at most one in each iteration, only the set $U_1^l(t)$ contributes to new infections. The probability to select a vertex of type $l'$ for exploration in iteration $t$ is $\frac{u_0^{l'}(t)}{u_0(t)}$. Conditioning on the type of the vertex selected being of type $l'$, the probability for it to connect to a vertex in the set $U_1^l(t)$ is $\frac{\kappa(l', l)}{n}$. Hence, by summing over all vertices in $U_1^l(t)$ we obtain the second term. For threshold $k \in [K-1]$, the change of the set $U_k^l(t)$ results from the vertices in $U_{k+1}^l(t)$, which receive an edge from the explored vertex and are therefore added to $U_k^l(t+1)$. The negative term comes from those vertices in $U_k^l(t)$ that are added to $U_{k-1}^l(t+1)$ because they receive an edge from the currently explored vertex. Lastly, the number of vertices in $U_K(t)$ decreases by the number of vertices in $U_K(t)$ that receive an edge from the vertex currently being explored.

{\em Approximation with ODEs:}
We will approximate the components of $h(t)/n$ in the system \eqref{eq:sys_all} using the method proposed in \cite{wormald_differential_1995}. For this, let the vector $(\rho_k^l(t/n))_{k \in [K]\cup \{0\}, l \in [L]}$ of functions, where $\rho_k^l: [0,1] \rightarrow [0,1]$ solves the following system of ordinary differential equations:
\begin{equation} \label{eq:ode}
\begin{aligned}
\frac{\dd \rho_0^l(\tau)}{\dd \tau}
    &= - \frac{\rho_0^l(\tau)}{\rho_0(\tau)} + \(\sum_{l' \in [L]}  \frac{\rho_0^{l'}(\tau)}{\rho_0(\tau)} \kappa(l',l)\) \rho_1^l(\tau),\\
\frac{\dd \rho_k^l(\tau)}{\dd \tau}
    &=  \(\sum_{l'\in[L]} \frac{\rho_0^{l'}(\tau)}{\rho_0(\tau)} \kappa(l',l)\) \(\rho_{k+1}^{l}(\tau) - \rho_k^{l}(\tau)\), \ \forall k \in [K-1], \\
\frac{\dd \rho_K^l(\tau)}{\dd \tau}
    &= - \(\sum_{l'\in[L]} \frac{\rho_0^{l'}(\tau)}{\rho_0(\tau)} \kappa(l',l)\) \rho_K^{l}(\tau),
\end{aligned}
\end{equation}
with the initial condition $\rho_k^l(0) = u_k^l(0)/n$ for $k \in [K]\cup\{0\}$.

The solution for the system \eqref{eq:ode} is given by
\begin{equation} \label{eq:ode_sol}
\begin{aligned}
\rho_0^l(\tau)
    &= - \beta^l(\tau) + \rho_{0}^l(0) + \sum_{k' \in [K]} \rho_{k'}^l(0) \( 1 - \sum_{k''=0}^{k'-1} p(k'',\phi^l(\tau)) \),\\
\rho_k^l(\tau)
    &= \sum_{k'=k}^K \rho_{k'}^l(0) p(k' -k,\phi^l(\tau)), \ \forall k \in [K-1], \\
\rho_K^l(\tau)
    &= \rho_K^l(0) p(0, \phi^l(\tau)),
\end{aligned}
\end{equation}
where
\begin{align}
p(k, \phi) := \frac{\phi^k}{k!}e^{-\phi}, &&
\beta^l(\tau) := \int_0^{\tau} \frac{\rho_0^{l}(s)}{\rho_0(s)} \dd s, &&
\phi^l (\tau) := \sum_{l' \in [L]} \kappa(l',l) \beta^{l'}(\tau).
\end{align}
The following lemma asserts the condition to apply \cite[Theorem~2]{wormald_differential_1995}, which, after some additional steps, will allow us to approximate the quantities $(u_k^l(t))_{k \in [K]\cup \{0\}, l \in [L]}$ of the sequential exploration process for large $n$ with high probability by the solution of the differential equation system \eqref{eq:ode}.
\begin{lemma} \label{prop:wormald_cond}
Let $\rho (0,\mathcal{S})>0$ and $f_k^l := f_k^l(\tau, (\rho_k^l)_{k \in [K]\cup \{0\}, l \in [L]}):= \dd \rho_k^l / \dd \tau$ for $l\in [L],k \in [K]\cup\{0\}$. The following holds for $n$ large enough:
\begin{enumerate}
\item For $\delta>0$, the functions $f_k^l,l\in [L],k \in [K]\cup\{0\}$ defining the right hand side of \eqref{eq:ode} fulfill a Lipschitz condition on the domain
\begin{align} \label{eq:lipschitz}
\mathcal{D}_{\delta} = \left\{\( \tau, \(\rho_k^l\)_{\substack{k \in [K]\cup \{0\} \\ l \in [L]}} \)\in \R^{(K+1)L + 1}: 0<\tau<1, 0< \rho_k^l(\tau) < 1,  \sum_{l \in [L]} \rho_0^{l}(\tau) > \delta  \right\}.
\end{align}
\item There exist functions $\omega = \omega(n)$ and $\gamma=\gamma(n)$ with $\gamma \rightarrow \infty$ as $n \rightarrow \infty$ and $\gamma^4 \log n < \omega < n^{2/3}/\gamma$ such that
\begin{align}
\Pb\(\abs{u_k^l(t+1) - u_k^l(t)} > \frac{\sqrt{\omega}}{\gamma^2 \sqrt{\log n}} \mid h(t) \) = o(n^{-3})
\end{align}
for all $k \in [K]\cup\{0\}$ and $l \in [L]$.
\end{enumerate}
\end{lemma}
\begin{proof}
Note that 1. follows directly from \eqref{eq:ode} and the fact that $\kappa$ is bounded and $\rho_0 > \delta$ for $n$ large by the assumption $\rho (0,\mathcal{S})>0$.

For 2. As in \cite{meyer-brandis_bootstrap_2019}, choose $\omega(n) = B^2 n^{25/48}$ with $B > 0$ constant and $\gamma(n) = n^{1/8}$, then 
\begin{align}
\frac{\sqrt{\omega}}{\gamma^2 \sqrt{\log n}} = \frac{B n^{1/96}}{\sqrt{\log n}}.
\end{align}
A rough bound can be derived from the maximal degree of all vertices in the graph. This gives
\begin{align}
\Pb\(\abs{u_k^l(t+1) - u_k^l(t)} \ge d \mid h(t)\) \le n \binom{n-1}{d} \(\frac{M}{n}\)^d \le n \frac{M^{2d}}{d!}
\end{align}
where $M:= \sup_{l\in[L]} \sum_{l' \in [L]} \kappa(l', l)$. Noting that $d! \ge n^5 M^{2d}$ for $d \ge n^{1/100}$ and large $n$ completes the proof.
\end{proof}

Therefore by \cite{wormald_differential_1995}, it holds for the system \eqref{eq:sys_all} that
\begin{align}
u_k^l(t) = n \rho_k^l(t/n) + o(n), \ \forall k \in [K]\cup\{0\}, \ l \in [L],
\label{eq:approx}
\end{align}
with probability $1-o(1)$ within the domain $\mathcal{D}_{\delta}$ where $0 <t/n < 1$ and $\sum_{l \in [L]} \rho_0^{l}(t/n) > \delta$.

A heuristic interpretation of the solution \eqref{eq:ode_sol} is as follows: At the iteration $\tau$ of the percolation, for a vertex of type $l \in [L]$, its degree to the set of infected vertices is approximately Poisson distributed with parameter $\phi^l(\tau)$. Note that $p(k, \phi)$ is exactly defined in the form of the probability mass function of a Poisson random variable with parameter $\phi$. In other words, for $0 \le k \le k' \le K$,  a vertex of type $l$ and initial threshold of $k$ has probability $p(k' - k, \phi^l(\tau))$ to be updated to a threshold of $k'$ at iteration $\tau$. Therefore, at iteration $\tau$, the fraction of vertices of type $l$ and threshold $k$ is the sum over the fractions of all other vertices with initially threshold $k' \ge k$ that receive exactly $k'-k$ infectious connections.

We are interested in the iteration $\hat{t}$ when the sequential exploration stops, i.e. $U(\hat{t})=\emptyset $ and $U(t)\neq \emptyset $ for $t<\hat{t}$. It is clear that $\hat{t} = |\mathbb{D}(G(n,\cV_{L,K}^d))|$ is the nature of the exploration process and we again stress that $\hat{t}$ is actually a random time. Because the approximation only holds in $\mathcal{D}_\delta$ and thus stops to hold before the sequential exploration comes to the end, we shall first study the iteration index when the number of unexplored infected vertices reaches a small proportion $\delta > 0$. Then, we derive a condition that ensures that when the fraction of unexplored infected vertices approaches zero, the process actually ends - it cannot rebound and trigger new infections. That is, with $\hat{t}^\delta$ the first time when $u_0(\hat{t}^\delta)/n \leq \delta$, we need to ensure that $\hat{t}^\delta/n\xrightarrow[]{p}\hat{t}/n$ as $\delta \rightarrow 0$.

We denote by $\what{\tau}$ the first time when all the $(\rho_0^l)_{l\in [L]}$ in the solution to the ODE system \eqref{eq:ode} reach zero. In addition we denote by $\what{\tau}^\delta$ the first time when $\sum_{l \in [L]} \rho_0^l$ reaches a given positive $\delta$:
\begin{align} \label{eq:stop}
\what{\tau} :=  \inf \left\{\tau:  \( \rho_0^l(\tau) \)_{l \in [L]} = \mathbf{0} \right\},
&& \what{\tau}^{\delta} := \inf \left\{ \tau: \sum_{l \in [L]} \rho_0^l(\tau) = \delta \right\}.
\end{align}

Note that the solution \eqref{eq:ode_sol} is not explicit due to the integral term $\beta^l(\tau)$. Observe that $\beta^l(\tau) \in \[0,1\]$ and $\beta^l$ is monotonically increasing in $\tau \in [0, \what{\tau}]$ and that $\sum_{l \in [L]} \beta^l(\tau) = \tau$, which we will use later. 

Now, define the functions $\nu: \R^L \rightarrow [0,1]$ and $\lambda: \R^L \rightarrow \R$ by
\begin{equation}\label{eq:ode_sol_z}
\begin{aligned}
\nu_0^l(\mathbf{z})
    &= - z^l + \sum_{k'=0}^K \nu_{k'}^l(\mathbf{0})\( 1 - \sum_{k''=0}^{k'-1} p(k'',\lambda^l(\mathbf{z})) \), \\
\nu_k^l(\mathbf{z})
    &= \sum_{k'=k}^K \nu_{k'}^l(\mathbf{0})  p(k'-k,\lambda^l(\mathbf{z})) , \ \forall k \in [K-1], \\
\nu_K^l(\mathbf{z})
    &= \nu_K^l(\mathbf{0})  p(0, \lambda^l(\mathbf{z})),
\end{aligned}
\end{equation}
where
\begin{align}
\mathbf{z}:= \(z^l\)_{l \in [L]} \in \[0,1\]^L, &&
\lambda^l(\mathbf{z}) := \sum_{l' \in [L]} \kappa(l',l) z^{l'}.
\end{align}
We remark that $\nu_k^l(\boldsymbol{z})=\rho_k^l(\tau)$ for $\boldsymbol{z}=(\boldsymbol{\beta}^1(\tau), ..., \boldsymbol{\beta}^L(\tau))$ for $k\in [L]$ and $l\in [L]$.
The partial derivative of $\nu_0^l(\mathbf{z})$ is given by
\begin{align} \label{eq:deriv_z}
\frac{\d \nu_0^l(\mathbf{z})}{\d z^{l'}}
    &= -\delta_{l,l'} + \kappa(l',l) \nu_1^l(\mathbf{z})
\end{align}
where $\delta_{l,l'}$ is the Kronecker delta.

Define now
\begin{equation}\label{eq:stop_z}
\begin{aligned} 
\mathcal{Z}_0 := \left\{ \mathbf{z}: \boldsymbol{\nu_0}(\mathbf{z}) = \mathbf{0} \right\},
&&
\mathcal{Z}_\delta := \left\{ \mathbf{z}: \sum_{l\in[L]} \nu_0^l(\mathbf{z}) = \delta \right\}, \\
\what{\mathbf{z}} := \min \left\{\mathcal{Z}_0 \right\},
&&
\what{\mathbf{z}}^\delta := \min \left\{ \mathcal{Z}_\delta \right\}.
\end{aligned}
\end{equation}
The existence of the last two component-wise minimal $\what{\mathbf{z}}$ and $\what{\mathbf{z}}^\delta$ can be verified by rewriting $\boldsymbol{\nu_0}(\mathbf{z}) = \mathbf{0}$ and $\nu_0^l(\mathbf{z}) = \delta$ as a fixed point problem using the definitions in \eqref{eq:ode_sol_z}. Then by properties of the functions $\nu^l_0$ similar to those of \eqref{eq:operator} used in the proof of Lemma~\ref{prop:fp_exist}, we can apply the Knaster-Tarski fixed point theorem to show existence. Therefore, it holds that
\begin{align}
\sum_{l\in[L]} \what{z}^l = \min_{\mathbf{z} \in \mathcal{Z}_0 }\left\{ \sum_{l\in[L]} z^l\right\}, &&
\sum_{l\in[L]} \what{z}^{\delta,l} = \min_{\mathbf{z} \in \mathcal{Z}_\delta }\left\{ \sum_{l\in[L]} z^l\right\}.
\end{align}
\begin{proposition}[Least joint zeros] \label{prop:first_zero}
For $\what{\tau}$ and $\what{\tau}^\delta$ as defined in equation \eqref{eq:stop}, and $\what{\mathbf{z}}$ and $\what{\mathbf{z}}^{\delta}$ defined in equation \eqref{eq:stop_z} the following holds true:
\begin{align}
\boldsymbol{\beta}(\what{\tau}) = 
\what{\mathbf{z}},
&& \what{\tau} = \sum_{l\in[L]} \what{z}^l,
&& \boldsymbol{\beta}(\what{\tau}^\delta) = \what{\mathbf{z}}^\delta,
&& \what{\tau}^\delta = \sum_{l\in[L]}  \what{z}^{\delta,l}.
\end{align}
\end{proposition}
\begin{proof}

We show the first two equalities; the others can be shown with the same arguments. First, we have $\boldsymbol{\beta}(\what{\tau}) \ge \what{\mathbf{z}}$ as $\boldsymbol{\beta}(\what{\tau}) \in \mathcal{Z}_0$ and by the definition of $\what{\mathbf{z}}$. To show the equality, suppose $\boldsymbol{\beta}(\what{\tau}) > \what{\mathbf{z}}$. Since the curve $\boldsymbol{\beta}$ is continuous and non-decreasing in $[0, \what{\tau}]$ and $\boldsymbol{\beta}(0)=\boldsymbol{0}$,
 there exists $\bar{\tau} < \what{\tau}$ at which point for the first time $\beta^l(\bar{\tau}) = \what{z}^l$ for some  $l \in [L]$. It holds that $\boldsymbol{\beta}(\bar{\tau}) \le \mathbf{\what{z}}$. If actually $\beta^{l'}(\bar{\tau}) = \what{z}^{l'}$ for all $l' \in [L]$, then $\boldsymbol{\nu_0}(\boldsymbol{\beta}(\bar{\tau})) =  \boldsymbol{\nu_0}(\mathbf{\what{z}}) = \mathbf{0}$, which contradicts with the definition of $\what{\tau}$. Hence, there should exist at least one $l' \in [L], \ l' \neq l$ such that $\beta^{l'}(\bar{\tau}) < \what{z}^{l'}$. However, with $\tfrac{\d \nu_0^l(\mathbf{z})}{\d z^{l'}} > 0$ for all $l'\neq l$, it holds that $\nu_0^l(\boldsymbol{\beta}(\bar{\tau})) < \nu_0^l (\mathbf{\what{z}}) = 0$, which contradicts with $\rho_0^l(\tau) \ge 0$ for $0 < \tau < \what{\tau}$. Therefore, it follows $\boldsymbol{\beta}(\what{\tau}) = \what{\mathbf{z}}$ and $\what{\tau} = \sum_{l\in[L]} \beta^l(\what{\tau}) =\sum_{l\in[L]} \what{z}^l$.
\end{proof}

{\em Final phase of the process:}
The approximation of the process with the functions \eqref{eq:ode_sol_z} works only in the domain $\mathcal{D}_{\delta}$ before the process reaches $\what{\tau}^{\delta}$. At this time, a proportion $\delta+o(1)$ of infected vertices remains, which still need to be explored. We need to know whether the percolation arising from these remaining $\delta n+o(n)$ unexplored vertices is negligible as $\delta \rightarrow 0$, or whether they still trigger a large number of additional infections. Thus, we study how the process triggered by the remaining small proportion $\delta$ of infected vertices evolves. We explore the final phase of the process starting from the iteration $\lfloor \what{\tau}^{\delta} n \rfloor$ by generations instead of the sequential procedure used so far. Then, we show that, under an appropriate condition, the proportion of infected vertices in the final phase will converge to zero as $\delta \rightarrow 0$. 

For this we group the remaining uninfected vertices at iteration $\lfloor \what{\tau}^{\delta} n \rfloor$ into \textit{weak} vertices and \textit{strong} vertices by their threshold as follows:
Let
\begin{align}
    W^l :=  U_1^l(\lfloor \what{\tau}^{\delta} n\rfloor),
    && W :=  \bigcup_{l\in[L]} W^l
\end{align}
denote sets of weak type $l$ vertices, which have threshold equal to $1$, and their union set of all  weak vertices; and
\begin{align}
    S^l := \bigcup_{k = 2}^K U_k^l(\lfloor \what{\tau}^{\delta} n\rfloor),
    && S := \bigcup_{l\in[L]} S^l
\end{align}
denote sets of strong type $l$ vertices that have threshold greater or equal to $2$, and their union set of all strong vertices.
To track the iterative exploration, we use subscript $j \in \mathbb{N}$ for the $j$-th iteration of our exploration process. Let $U_0^l := U_0^l(\lfloor \what{\tau}^{\delta} n\rfloor)$ denote the initial set of infected vertices of type $l \in [L]$ in the final phase of the process; $W_j^l$ and $S_j^l$ denote the type $l$ weak vertices and strong vertices in the $j$-th iteration.

The final phase is explored as follows: in its first round, we examine all the weak vertices and strong vertices that are infected by vertices in $U_0$, hence we obtain corresponding sets $W_1^l$, $S_1^l$, $W_1 \cup S_1$; in the second round, we examine the newly infected vertices  $W_2^l$, $S_2^l$, $W_2 \cup S_2$, which have been infected through edges from $W_1^l$ and $S_1^l$; and so on. Iterating via such exploration, we obtain the set of all the infected vertices in the final phase $\bigcup_{m=1}^{\infty}W_m \cup S_m$. 

\begin{proposition} [Convergence of final phase of the process] \label{prop:remain}
Let there exist a vector $(w^l)_{l \in [L]} > \mathbf{0}$ and $\sum_{l\in[L]} w^l \le 1$, such that
\begin{align} \label{eq:deriv_zhat_supp}
\sum_{l' \in [L]} w^{l'} \frac{\d \nu_0^l(\what{\mathbf{z}})}{\d z^{l'}}  < 0, \ \forall l \in [L].
\end{align}
Then
\begin{align}
 n^{-1}\abs{\bigcup_{m=1}^\infty W_m \cup S_m}\xrightarrow[\delta \rightarrow 0]{p} 0.
\label{eq:remain}
\end{align}
\end{proposition}
\begin{proof}
For the fixed $\mathbf{w}=(w^l)_{l \in [L]}$ such that the condition \eqref{eq:deriv_zhat_supp} holds,
it follows by the continuity of $\d \nu_0^l$ that there exists $\delta' >0$ and $\epsilon > 0$ such that for all $0< \delta$ small enough it holds that
\begin{align}  \label{eq:deriv_zdelta}
\sum_{l' \in [L]} w^{l'} \frac{\d \nu_0^l(\what{\mathbf{z}}^{\delta})}{\d z^{l'}}  < -\epsilon, \ \forall l \in [L].
\end{align}
where $\what{\mathbf{z}}^{\delta}$ is as defined in \eqref{eq:stop_z}.

Consequently, it follows by \eqref{eq:deriv_z} that
\begin{equation} \label{eq:deriv_ineq}
\begin{aligned}
&\sum_{l' \in [L]} w^{l'} \kappa(l',l) \nu_1^l(\what{\mathbf{z}}^{\delta}) < c_1^l w^l < w^l - \epsilon, \\
&\sum_{l\in[L]}\sum_{l' \in [L]} w^{l'} \kappa(l',l) \nu_1^l(\what{\mathbf{z}}^{\delta}) < \sum_{l\in[L]} c_1^l w^l <  c_1 \sum_{l\in[L]} w^l \le c_1,
\end{aligned}    
\end{equation}
where $0 < c_1^l < 1$ for $l\in [L]$ and $c_1 = \max_{l\in [L]} c_1^l <1$ are constants depending on $\epsilon$ but not on $\delta$.
 With $\delta'=2\frac{\delta}{\bar{w}}, \bar{w}=\min \{ w^1,\dots. w^L\}$, it holds that  
\begin{align} \label{eq:delta_ineq}
    \nu_0^l(\what{\mathbf{z}}^{\delta}) < w^l \delta', \quad \forall l \in [L].
\end{align}
Recall that $\nu_k^l(\mathbf{z})=\rho_k^l(\tau)$ for $\mathbf{z}=(\boldsymbol{\beta}^1(\tau), ..., \boldsymbol{\beta}^L(\tau))$ for $k\in [L]$ and $l\in [L]$. By the approximation 
\eqref{eq:approx}, and because $\boldsymbol{\beta}(\what{\tau}^\delta) = \what{\mathbf{z}}^\delta$ by Proposition~\ref{prop:first_zero},  there exists a $\sigma (h(\lfloor \what{\tau}^{\delta} n\rfloor)$-measurable event $\mathcal{A}_n^\delta$ such that $\lim_{n\rightarrow \infty} \mathbb{P}(\mathcal{A}_n^\delta)=1$ and on which 
\begin{align}
&\sum_{l' \in [L]} w^{l'} \kappa(l',l) u_1^l(\lfloor \what{\tau}^{\delta} n\rfloor)/n < c_1^l w^l < w^l - \epsilon,\label{eq:deriv_true:sets:1} \\
&\sum_{l\in[L]}\sum_{l' \in [L]} w^{l'} \kappa(l',l) u_1^l(\lfloor \what{\tau}^{\delta} n\rfloor)/n < \sum_{l\in[L]} c_1^l w^l <  c_1 \sum_{l\in[L]} w^l \le c_1,\label{eq:deriv_true:sets:2}
\end{align}    
and 
\begin{equation}\label{approx:true:proc:1}
    u_0^l(\lfloor \what{\tau}^{\delta} n\rfloor)/n < w^l \delta', \quad \forall l \in [L].
    \end{equation}
holds.
We use induction to show that
\begin{align} 
I_{\mathcal{A}_n^\delta} \Eb\[ |W^l_m|\;\middle|\; \sigma (h(\lfloor \what{\tau}^{\delta} n\rfloor)) \]
    \le c_1 c^{m-1} w^l \delta' n,&&\label{eq:induction:1}\\  
I_{\mathcal{A}_n^\delta}\Eb\[|S^l_m|\;\middle|\; \sigma (h(\lfloor \what{\tau}^{\delta} n\rfloor)) \]
    \le c_2 c^{m-1} w^l \delta' n,&&\label{eq:induction:2}
\end{align}
where $0 < c_1 < c < 1$, $c_2 = c - c_1$. This will then imply that
\begin{equation} 
\sum_{l\in[L]}\sum_{m=1}^\infty I_{\mathcal{A}_n^\delta} \Eb\[ |W^l_m|+|S^l_m|\;\middle|\; \sigma (h(\lfloor \what{\tau}^{\delta} n\rfloor)) \]
    \le \frac{c \delta'}{1-c}  n\leq  \frac{c \delta}{(1-c)\bar{w}}  n,
\end{equation}
and we can then conclude with a simple Markov bound. 
First, for $m=1$, let $x \in U_0^{l'}$ be an infected vertex and $y \in W^l$ be a weak vertex, then the probability that $y$ gets infected by $x$ is $\kappa(l',l)/n$. The overall probability that $y \in W^l$ is infected by a vertex $x \in U_0$ is then
\begin{align}
p_w^l
    &:= \Pb \( \cup_{l' \in [L]} \cup_{x \in U_0^{l'}} \{E_{xy} = 1 \} \;\middle|\; y \in W^l , \sigma (h(\lfloor \what{\tau}^{\delta} n\rfloor)) \)
    \le \sum_{l'\in[L]}\frac{\kappa(l',l)}{n} |U_0^{l'}|.
\end{align}
Conditional on $\sigma (h(\lfloor \what{\tau}^{\delta} n\rfloor))$, on the set $\mathcal{A}_n^\delta$, the expected number of  weak vertices of type $l$ infected in the first round is bounded by:
\begin{align}
I_{\mathcal{A}_n^\delta} \Eb\[ |W_1^l|\;\middle|\; \sigma (h(\lfloor \what{\tau}^{\delta} n\rfloor) \]
    &\le  |W^l| p_w^l  \le |W^l|  \sum_{l' \in [L]} \frac{\kappa(l',l)}{n} |U_0^{l'}| \\
    &\le u_1^l(\lfloor \what{\tau}^{\delta} n\rfloor) \sum_{l' \in [L]} \kappa(l',l)  u_0^{l'}(\lfloor \what{\tau}^{\delta} n\rfloor)/n \le c_1 \delta' n  w^l,
\end{align}
where we use inequalities \eqref{eq:deriv_true:sets:1} and \eqref{approx:true:proc:1} in the last two lines. 

Conditional on $\sigma (h(\lfloor \what{\tau}^{\delta} n\rfloor)$, the probability that a strong vertex $y \in S^l$ is connected to two infected vertices in the first round is
\begin{align}
p_s^l
    &:= \Pb \( \cup_{l', l'' \in [L]} \cup_{x_1 \in U_0^{l'}, x_2 \in U_0^{l''}} \{E_{x_1y} = 1, E_{x_2y}=1 \} \;\middle|\; \sigma (h(\lfloor \what{\tau}^{\delta} n\rfloor) , y \in S^l \) \\
    &\le \sum_{l', l'' \in[L]} \frac{\kappa(l',l)\kappa(l'',l)}{n^2} |U_0^{l'}||U_0^{l''}|.
\end{align}
The expected number of type $l$ strong vertices that are infected in the first round is bounded on $\mathcal{A}_n^\delta$ by
\begin{align}
I_{\mathcal{A}_n^\delta}\Eb \[ |S_1^l|\;\middle|\; \sigma (h(\lfloor \what{\tau}^{\delta} n\rfloor) \]
    &\le  |S^l| p_s^l  \\
    &\le \sum_{l', l'' \in[L]}\frac{\kappa(l',l)\kappa(l'',l)}{n^2} |U_0^{l'}||U_0^{l''}| |S^l| \\
    &\le  \frac{\kappa_{max}^2}{n^2} |U_0|^2 |S^l| \le \kappa_{max}^2 (\sum_{l\in [L]}u_0^l(\lfloor \what{\tau}^{\delta} n\rfloor))^2 n \\
    &\le \kappa_{max}^2 (w^l\delta')^2 n \le  c_2 (w^l \delta') n 
\end{align}
for $\delta'$ small enough, and with $\kappa_{max} = \max_{l'}\{\sum_{l \in [L]} \kappa(l', l)\} < \infty$.

Next, by induction, assume inequality \eqref{eq:induction:1} holds for $m=M$, $M \ge 1$, then for round $m=M+1$, the expected number of remaining weak vertices from $W \backslash \cup_{m=1}^M W_m$ that are infected by infected vertices in $W_M \cup S_M$ from round $M$ is bounded by
\begin{align}
&I_{\mathcal{A}_n^\delta} \Eb \[|W_{M+1}^l|\;\middle|\; \sigma (h(\lfloor \what{\tau}^{\delta} n\rfloor) \]\\
    &\le I_{\mathcal{A}_n^\delta} \Eb\[|W^l \backslash \cup_{m=1}^M W^l_m| \sum_{l' \in [L]} \frac{\kappa(l',l)}{n}  |W^{l'}_M \cup S^{l'}_M|\;\middle|\; \sigma (h(\lfloor \what{\tau}^{\delta} n\rfloor)\] \\
    &\le |W^l| \sum_{l' \in [L]} \frac{\kappa(l',l)}{n}  I_{\mathcal{A}_n^\delta}  \Eb\[|W^{l'}_M \cup S^{l'}_M|\;\middle|\; \sigma (h(\lfloor \what{\tau}^{\delta} n\rfloor)\] \\
    &\le u_1^l(\lfloor \what{\tau}^{\delta} n\rfloor) \sum_{l' \in [L]} \kappa(l',l)  c^M \delta' w^{l'} \le  n c_1 c^M \delta' w^l,
\end{align}
where we used \eqref{approx:true:proc:1} in the last line. This shows inequality \eqref{eq:induction:1}.

For a strong vertex to become infected in round $M+1$, it needs at least one connection to $W_M \cup S_M$ and one to $\cup_{m=1}^M  (W_m \cup S_m)$. The expected number of strong vertices in $S \backslash \cup_{m=1}^M S_M$ that are infected is bounded on $I_{\mathcal{A}_n^\delta}$ by
\begin{align}
&I_{\mathcal{A}_n^\delta}\Eb\[|S^l_{M+1}|\;\middle|\; \sigma (h(\lfloor \what{\tau}^{\delta} n\rfloor) \]\\
    &\le  I_{\mathcal{A}_n^\delta}\Eb\[\sum_{l', l'' \in [L]} \frac{\kappa(l',l)\kappa(l'',l)}{n^2}  |W^{l'}_M \cup S^{l'}_M| |\cup_{m=1}^M  (W^{l''}_m \cup S^{l''}_m) |\;\middle|\; \sigma (h(\lfloor \what{\tau}^{\delta} n\rfloor)\] \\
    &\le  \kappa_{max}^2 \frac{c}{1-c} c^{M} \delta'^2 n \le c_2 c^M \delta' n w^l,
\end{align}
for $\delta'$ small enough. So the inequality \eqref{eq:induction:2} holds and with $\sum_{l\in[L]} w^l \le 1$ we obtain
\begin{align}
I_{\mathcal{A}_n^\delta}\Eb\[|W_m \cup S_m|\;\middle|\; \sigma (h(\lfloor \what{\tau}^{\delta} n\rfloor) \]
    \le \sum_{l \in [L]} I_{\mathcal{A}_n^\delta} \Eb\[|W^l_m \cup S^l_m|\;\middle|\; \sigma (h(\lfloor \what{\tau}^{\delta} n\rfloor) \]
    \le c^m \delta' n.
\end{align}
Consequently, for any $a > 0$ the Markov inequality yields
\begin{align}
&\Pb \(I_{\mathcal{A}_n^\delta}\frac{|\bigcup_{m=1}^\infty W_m \cup S_m|}{n} > a \;\middle|\; \sigma (h(\lfloor \what{\tau}^{\delta} n\rfloor) \)\\
    &\le a^{-1} \Eb\[I_{\mathcal{A}_n^\delta}\frac{|\bigcup_{m=1}^\infty W_m \cup S_m|}{n}\;\middle|\; \sigma (h(\lfloor \what{\tau}^{\delta} n\rfloor)\]   \\
    &\le a^{-1} \sum_{l\in[L]}\sum_{m=1}^\infty I_{\mathcal{A}_n^\delta} \Eb\[ |W^l_m|+|S^l_m|\;\middle|\; \sigma (h(\lfloor \what{\tau}^{\delta} n\rfloor)) \] \le  \frac{c\delta}{(1-c)\bar{w}a}   n,
\end{align}
which together with $\lim_{n\rightarrow \infty} \mathbb{P}(\mathcal{A}_n^\delta)=1$ leads to the result \eqref{eq:remain}.
\end{proof}

The final fraction of infected vertices for graph with finite vertex type is determined as Proposition~\ref{prop:fp_z}.
\begin{proof}[Proof of Proposition~\ref{prop:fp_z}]
The ODE approximation \eqref{eq:ode} holds in the domain $\mathcal{D}_{\delta}$ until $\what{\tau}^{\delta}$ is reached. At the time $\lfloor n\what{\tau}^{\delta} \rfloor$,  the number of infected vertices which have been explored is exactly $\lfloor n\what{\tau}^{\delta} \rfloor$. 
It holds that 
$$|\mathbb{D}(G(n,\cV_{L,K}^d))|=\lfloor n\what{\tau}^{\delta} \rfloor + \abs{\bigcup_{m=1}^\infty W_m \cup S_m}.$$
Note that $\nu_0(\tau)$ is decreasing in a neighborhood of  $ \what{\tau}$, which implies that $\what{\tau}^{\delta} \rightarrow \what{\tau}$. By Proposition~\ref{prop:remain} together with Proposition~\ref{prop:first_zero} the result follows.
\end{proof}

\section{Proofs of auxiliary results} \label{sec:aux}

\begin{proof}[Proof of Lemma \ref{prop:psi_mono}]
The statement becomes obvious once we rewrite \eqref{eq:operator} as 
$$\Psi_{\kappa} [\boldsymbol{\cdot}](\cdot) = \sum_{k=0}^\infty \eta_k(\cdot)\Pb\left\{\text{Poi}(\Lambda_{\kappa} [\boldsymbol{\cdot}](\cdot)) \ge k)\right\},$$ where $\text{Poi}(\lambda)$ is a Poisson random variable with intensity $\lambda$. For $f \le g$ we have $\Lambda_{\kappa} [f](\cdot) \le \Lambda_{\kappa} [g](\cdot)$, hence $\Pb\left\{\text{Poi}(\Lambda_{\kappa} [f](\cdot)) \ge k)\right\} \le \Pb\left\{\text{Poi}(\Lambda_{\kappa} [g](\cdot)) \ge k)\right\}$. The result follows.
\end{proof}

\begin{proof}[Proof of Lemma \ref{prop:psi_lipschitz}]
We prove the second case where $\kappa$ is Lipschitz. If $\kappa$ is only continuous, then by compactness of $\mathcal{S}\times \mathcal{S}$, it follows from the Heine–Cantor theorem that $\kappa$ is uniformly continuous on $\mathcal{S}\times \mathcal{S}$. It is then easy to see that one can adapt the following proof by replacing the Lipschitz constant $L_{\kappa}$ by an $\epsilon$-$\delta$ argument. Note that $\Lambda_{\kappa} [f] \le M_{\kappa}$ as $f$ is bounded, and that $ \mathbf{0} \le P_{\kappa}^k [f] \le \mathbf{1}$, $ \mathbf{0} \le \Psi_{\kappa} [f] \le \mathbf{1}$ as they are in the form of a Poisson probability. For the Lipschitz continuity, we examine each of the operators as follows:
\begin{align}
|\Lambda_{\kappa}f(x) - \Lambda_{\kappa}f(y)|
    &= \left|\int_{s \in \cS} \kappa(s, x) f(s) \dd \mu(s) - \int_{s \in \cS} \kappa(s, y) f(s) \dd \mu(s)\right| \\
    &\le \int_{s \in \cS} |\kappa(s, x) - \kappa(s, y)| f(s) \dd \mu(s) \\
    &\le L_{\kappa} |x-y|. \\
\end{align}
Next, for the $P_{\kappa}^k$ operators and $k \in \mathbb{N}_0$, it follows that
\begin{align}
|P_{\kappa}^k f(x) - P_{\kappa}^k f(y)|
    &= \left|\frac{\(\Lambda_{\kappa} f(x)\)^k}{k !} e^{-\Lambda_{\kappa}f(x)} - \frac{\(\Lambda_{\kappa} f(y)\)^k}{k !} e^{-\Lambda_{\kappa}f(y)}\right| \\
    &\le \frac{e^{-\Lambda_{\kappa}f(x)}}{k!} \left|\(\Lambda_{\kappa} f(x)\)^k - \(\Lambda_{\kappa} f(y)\)^k\right| + \frac{\(\Lambda_{\kappa} f(y)\)^k}{k !} \left|e^{-\Lambda_{\kappa}f(x)} - e^{-\Lambda_{\kappa}f(y)}\right| \\
    &\le \frac{e^{-\Lambda_{\kappa}f(x)} \sum_{i=0}^{k-1}\(\Lambda_{\kappa}f(x)\)^{i}\(\Lambda_{\kappa}f(y)\)^{k-1-i}}{k!} \left|\Lambda_{\kappa} f(x)- \Lambda_{\kappa} f(y)\right| \\
    &\quad +  \frac{\(\Lambda_{\kappa} f(y)\)^k}{k !} \left|\Lambda_{\kappa}f(x) -\Lambda_{\kappa}f(y)\right|,
\end{align}
where in the last line the first term applies the equality $1 - a^k = (1-a)(1+a+\dots+a^{k-1})$ with $a=\Lambda_{\kappa}f(x)/\Lambda_{\kappa}f(y)$, and the second term uses that $|e^{-x} - e^{-y}| \le |x - y|$ for $x, y \ge 0$. Then, recalling the Lipschitz continuity of $\Lambda_{\kappa}$ and the upper bound of $\kappa$, we have the following inequality for $k \ge 1$:
\begin{align}
|P_{\kappa}^k f(x) - P_{\kappa}^k f(y)|
    &\le \frac{M_{\kappa}^{k-1}}{(k-1)!}\left|\Lambda_{\kappa}f(x) - \Lambda_{\kappa}f(y)\right| + \frac{M_{\kappa}^k}{k !} \left|\Lambda_{\kappa}f(x) - \Lambda_{\kappa}f(y)\right| \\
    &= \frac{(1+M_{\kappa}/k)M_{\kappa}^{k-1}}{(k-1)!}\left|\Lambda_{\kappa}f(x) - \Lambda_{\kappa}f(y)\right| \\
    &\le \frac{L_{\kappa} (1+M_{\kappa}) M_{\kappa}^{k-1}}{(k-1)!} |x-y|.
\end{align}

We hitherto conclude the Lipschitz continuity of the operator $\Psi_{\kappa}$ by
\begin{align}
\left|\Psi_{\kappa} [f](x) - \Psi_{\kappa} [f](y) \right|
    &= \left|\sum_{k=0}^\infty \eta_k(x)\(1- \sum_{k'=0}^{k-1}P_{\kappa}^{k'} f(x)\)\right.   \left.- \sum_{k=0}^\infty \eta_k(y)\(1- \sum_{k'=0}^{k-1}P_{\kappa}^{k'} f(y)\) \right| \\
    &\le  \sum_{k=1}^\infty \sum_{k'=0}^{k-1} \left| P_{\kappa}^{k'}f(x) - P_{\kappa}^{k'} f(y)\right| \\
    &\le (1+M_{\kappa}) \sum_{k=1}^\infty \sum_{k'=0}^{k-1} \frac{L_{\kappa} M_{\kappa}^{k'-1}}{(k'-1)!} |x-y| \\
    &= L_{\kappa} (1+M_{\kappa}) e^{M_{\kappa}} |x-y| \sum_{k=1}^\infty \sum_{k'=k}^{\infty} \frac{M_{\kappa}^{k'}}{k'!} e^{-M_{\kappa}}\\
    &= L_{\kappa} (1+M_{\kappa}) e^{M_{\kappa}}|x-y| \sum_{k=1}^\infty \Pb\(\text{Poi}(M_{\kappa}) \ge k\) \\
    &= L_{\kappa} (1+M_{\kappa}) e^{M_{\kappa}}|x-y| \Eb[\text{Poi}(M_{\kappa})] \\
    &= L_{\kappa} (1+M_{\kappa})M_{\kappa} e^{M_{\kappa}}|x-y|,
\end{align}
where $\text{Poi}(M_{\kappa})$ is some Poisson random variable with parameter $M_{\kappa}$.
\end{proof}

\begin{proof}[Proof of Lemma~\ref{prop:fp_exist}]
 Because $\kappa$ is Lipschitz continuous on $\cS\times \cS$, it follows by Lemma~\ref{prop:psi_lipschitz} and the fixed point property that all fixed points, if any, are in $\mathcal{H}$.

It remains to show that $\mathcal{H}$ is not empty and the existence of a least fixed point. Notice that the set $\cF_1^{Lip}$ is a partially ordered set with the partial order $\le$ as it satisfies: reflexibility, $f \le f$ for $f \in \cF_1^{Lip}$; transitivity, $f_1 \le f_2$ and $f_2 \le f_3$ implies $f_1 \le f_3$ for $f_1, f_2, f_3 \in \cF_1^{Lip}$; and anti-symmetricity, $f_1 \le f_2$ and $f_2 \le f_1$ implies that $f_1 = f_2$ for $f_1, f_2 \in \cF_1^{Lip}$. Let $L_{\Psi, \kappa}$ be the Lipschitz constant of functions in $\cF_1^{Lip}$ as derived in Lemma \ref{prop:psi_lipschitz}. It holds for all functions $f\in \mathcal{H}$ that 
\begin{equation}\label{lip:H}
f(x) \le f(y) + L_{\Psi,\kappa} \abs{x - y}.
\end{equation}
We now define the pointwise supremum $\bar{f}$ by $\bar{f}(x)=\sup_{f\in \cF_1^{Lip}} f(x)$ for all $x\in \cS$. It then holds for all $x\in \cS$ that 
\begin{equation}\label{sup:right}
f(x) \leq \bar{f}(y)+L_{\Psi,\kappa} \abs{x-y}
\end{equation}
by taking the supremum over the right-hand side of \eqref{lip:H}. Then, by taking also the supremum over the left-hand side of \eqref{sup:right}, it follows that 
\begin{equation}\label{lip:ine:sided}
\bar{f}(x)\leq \bar{f}(y)+ L_{\Psi,\kappa} \abs{x-y}.
\end{equation}
Reversing the role of $x$ and $y$ in the arguments leads to 
$\bar{f}(y)\leq \bar{f}(x)+ L_{\Psi,\kappa} \abs{x-y}$, which together with \eqref{lip:ine:sided} implies Lipschitz continuity of the function $\bar{f}$. 

A similar argument shows that the infimum $\underline{f}$ defined by $\underline{f}(x) = \inf_{f\in \cF_1^{Lip}} f(x)$ is Lipschitz continuous as well. It follows that $\bar{f}$ and $\underline{f}$ are in $\mathcal{H}$. We have shown that $(\cF_1^{Lip}, \le)$ forms a \textit{complete lattice}. Since the operator $\Psi_{\kappa}: \cF_1^{Lip} \rightarrow \cF_1^{Lip}$ is order-preserving by Lemma \ref{prop:psi_mono}, it follows from the Knaster-Tarski theorem that $\Psi_{\kappa}$ has at least one fixed point and a least fixed point (with respect to the partial ordering) on $\cF_1^{Lip}$, which we denote by $\hat{f}$. The lower bound holds because $\mathbf{0} \le \eta_0 = \Psi_{\kappa} \mathbf{0}  \le \Psi_{\kappa} \hat{f} = \hat{f}$.
\end{proof}

\begin{proof}[Proof of Lemma \ref{prop:gd}]
We show that 
\begin{equation}\label{toshow}
D P_{\kappa}^k f [\boldsymbol{\cdot}] = - \Lambda_{\kappa} [\boldsymbol{\cdot}] \(P_{\kappa}^{k}f - P_{\kappa}^{k-1}f \),
\end{equation}
from which \eqref{eq:derivative} follows immediately by the definition of $\Psi_{\kappa}$. 

We need to find for every $f\in \cF_b$ an operator $A \in L^+ (\cF_b)$ such that 
$$
\lim_{\norm{h}_\infty\rightarrow 0} \frac{\norm{P_{\kappa}^k [f+h] - P_{\kappa}^k f -Ah}_\infty }{\norm{h}_\infty}=0.
$$
Without loss of generality, assume $f \in [0,1]$. First note that by linearity of the integral, the operator $\Lambda_{\kappa}$ is linear. Moreover, since $\norm{ \kappa }$ is bounded, it follows that 
$$
\norm{\int_{s \in \cS} \kappa(s, \cdot) f(s) \dd \mu(s)}_\infty \leq \norm{ \kappa }_\infty \norm{f}_\infty
$$
and the operator is bounded. It follows that $D \Lambda_{\kappa} f =\Lambda_{\kappa} $. Next define the function $M_k: \cF_b \rightarrow \cF_b$ by 
$M_k f =\frac{f^k}{k !} e^{-f}$. In the definition, all operations are to be understood pointwise, i.e. $M_k f (x) =m(f(x))$ with the function $m: \mathbb{R} \rightarrow \mathbb{R}$ defined by $m(y)= \frac{y^k}{k !} e^{-y}$. Formally differentiating the function $M_k$ leads to the function $N_k: \cF_b \rightarrow \cF_b$
defined by $N_k f = \frac{f^{k-1}}{(k-1)!} e^{-f} - \frac{f^k}{k !} e^{-f} $. This is again meant in a pointwise manner, meaning that $(N_k f)(x)=n(f(x))$ with $n: \mathbb{R} \rightarrow \mathbb{R}$ defined by $n(y)= \frac{y^{k-1}}{(k-1) !} e^{-y} - \frac{y^k}{k !} e^{-y}$. The multiplication operator defined by $h \mapsto (N_k f)\cdot h $ is an element of $L(\cF_b)$ and we will show that it is in fact the Fr\'echet derivative of $M_k$ at the point $f$. For this we calculate 
\begin{align}
\norm{M_k [f+h] - M_k f - (N_k f)\cdot h }
    &=\sup_{x\in [0,1]} \abs{  M_k [f+h](x) - M_kf(x) - N_k f(x) \cdot h(x) } \\
    &=\sup_{x\in [0,1]} \abs{  m(f(x)+h(x)) - m(f(x)) - n( f(x)) h(x) } \\
    &=\sup_{x\in [0,1]} \abs{  n'( \zeta_x) (h(x))^2/2 } ,
\end{align}
where the $0 \leq \zeta_x \leq h(x)$ and $n'( \zeta_x) (h(x))^2/2$ is the remainder term of the Taylor approximation. Since $h(x) \leq 1$ and $n'$ bounded on $[0,1]$, it follows that the term in the last line is bounded by $C \norm{h}^2$ for some $C$. Dividing by $\norm{h}$ shows that the operator defined by $h \mapsto (N_k f)\cdot h $ is in fact the Fr\'echet derivative of $M_k$. 

Since $P_{\kappa}^k =M_k \circ \Lambda_{\kappa}  $, it follows by the chain rule for the Fr\'echet derivative that
$$ 
D P_{\kappa}^k f =(D M_k \circ \Lambda_{\kappa} ) [f] \circ D \Lambda_{\kappa} [f],
$$
from which \eqref{toshow} follows immediately noting that $D \Lambda_{\kappa}f$ is constant in $f$.
\end{proof}

\begin{proof}[Proof of Lemma~\ref{prop:partition}]
We construct a sequence of partitions indexed by $m$ with the properties 1. and 2. and such that $\mathrm{diam}(\cS_{L(m)}^l)\leq 1/m$ for $l \in [L(m)]$. We start with $m=1$. For every point $s\in \cS$, let $B_1(s)$ be the open ball around $s$ with diameter $1$. Because $\cS$ is compact, there exists a finite number of sets $B_1(s^1_1),\dots ,B_1(s^{L(1)}_1)$ such that $\cS = \cup_i^{L(1)} B_1(s^{i}_1)$, and we set $\cS^1_{L(1)}=B_1(s^1_1)$ and $\cS^l_{L(1)}=B_1(s^l_1)\setminus \cup_{j<l} B_1(s^j_1)$ for $l=2, \dots, L(1)$. Clearly, then the sets $\{ \cS^l_1 \}_{l\in [L(1)]}$ form a partition of $\cS$ and their diameter is bounded by $1$.

Let us assume that we have already partitions $\{\cS_{L(k)}^l\}_{l \in [L(k)]}$ constructed for $k=1,\dots, m$. To generate the partition $\{\cS_{L(m+1)}^l\}_{l \in [L(m+1)]}$, we start again with open balls $B_{1/(m+1)}(s)$ with diameter $1/(m+1)$ and use compactness of $\cS$ to choose finitely many points $s^1_{m+1},\dots, s^{L'(m+1)}_{m+1}$ such that 
$$\cS=\bigcup_{i=1}^{L'(m+1)} B_{1/(m+1)}(s^i_{m+1}).$$ 
Now define again disjoint sets 
$$\widetilde{B}_{1/(m+1)}(s^i_{m+1})=B_{1/(m+1)}(s^i_{m+1})\setminus \bigcup_{j<i}B_{1/(m+1)}(s^j_{m+1})$$
for $i=1,\dots, L'(m+1)$. These sets form again a partition of $\cS$ and it holds that $\mathrm{diam}(\widetilde{B}_{1/(m+1)}(s^{m+1}_{i}))< 1/(m+1)$ for all $i=1,\dots, L'(m+1)$. 

Now let $\{\cS_{L(m+1)}^l\}_{l \in [L(m+1)]}$ consist of all the intersections of the sets $\widetilde{B}_{1/(m+1)}(s^i_{m+1})$ for $i=1,\dots, L'(m+1)$ with the sets $\cS_{L(m)}^k$ for $k \in [L(m)]$ where we define $L(m+1)$ to be the number of all resulting sets. In the entire process, we remove any empty set that appears. It is straightforward to see that all constructed sets are measurable as we start the construction with the open balls.  We have thus derived a sequence of partitions with the required properties. 
\end{proof}

\begin{proof}[Proof of Lemma \ref{lem:converge}]
We prove it for $\kappa_L^+$ and it is analogous for $\kappa_L^-$.
We first show the uniform convergence of $\kappa_L^+$ to $\kappa$. Recall by the construction of the partition $\{\cS_L^l\}_{l\in[L]}$ and $\kappa_L^+$, we have
\begin{align}
\norm{\kappa^+_L - \kappa}_\infty
    &= \sup_{s, s' \in \cS} \abs{\kappa^+_L(s, s') - \kappa(s,s')}\\
    &= \sup_{s, s' \in \cS}  \abs{\sum_{l,l' \in [L]} I_{\cS_L^l}(s) I_{\cS_L^{l'}}(s') \sup_{x \in \cS^l, y \in \cS^{l'}} \kappa(x,y) - \kappa(s, s')}.\\
\end{align}
Let $(s_1, s_2)$ be a point where the supremum is reached, and $l_1, l_2 \in [L]$ such that $s_1 \in \cS_L^{l_1}$ and $s_2 \in \cS_L^{l_2}$. It follows
\begin{align}
\norm{\kappa^+_L - \kappa}_\infty
    = \abs{ \sup_{x \in \cS_L^{l_1}, y \in \cS_L^{l_2}} \kappa(x,y) - \kappa(s_1, s_2)} 
    \le \sup_{x, x' \in \cS_L^{l_1}, y, y' \in \cS_L^{l_2}} \abs{ \kappa(x,y) - \kappa(x', y')}.
\end{align}
By Assumption~\ref{ass:kernel} on continuity of $\kappa$, for all $\epsilon >0$, there exists $\delta >0$ such that $\abs{\kappa(x,y) - \kappa(x', y')} < \epsilon$ for $d(x,x') < \delta$ and $d(y, y') < \delta$. According to Lemma~\ref{prop:partition}, it holds that $\mathrm{diam}(\cS_L^{l_1}),\mathrm{diam}(\cS_L^{l_2}) \leq \delta$ for $L$ large and the result follows. Therefore, as $L \rightarrow \infty$, we have the following inequalities and convergences:
\begin{align}
\norm{\Lambda_{\kappa^+_L} f_L - \Lambda_{\kappa}f}_\infty
    &= \norm{\Lambda_{\kappa^+_L} f_L - \Lambda_{\kappa} f_L + \Lambda_{\kappa} f_L - \Lambda_{\kappa}f}_\infty \\
    &\le \norm{\Lambda_{\kappa^+_L} f_L - \Lambda_{\kappa} f_L}_\infty + \norm{\Lambda_{\kappa} f_L - \Lambda_{\kappa}f}_\infty \\
    &=\norm{\int_{s \in \cS} \(\kappa^+_L(s, \cdot) - \kappa(s, \cdot)\)f_L(s) \dd \mu(s)}_\infty \\
    &\quad + \norm{\int_{s \in \cS} \kappa(s, \cdot) (f_L(s) - f(s)) \dd \mu(s)}_\infty \\
    &\le \int_{s \in \cS} \norm{\kappa^+_L - \kappa}_\infty f_L(s) \dd \mu(s) +  M_{\kappa} \abs{\int_{s \in \cS} (f_L(s) - f(s)) \dd \mu(s)}\\
    &\rightarrow 0, \\
\norm{P_{\kappa^+_L}^k f_L - P_{\kappa}^k f}_\infty
    &= \norm{\frac{(\Lambda_{\kappa^+_L} f_L)^k}{k !} e^{-\Lambda_{\kappa^+_L} f_L} - \frac{(\Lambda_{\kappa} f)^k}{k !} e^{-\Lambda_{\kappa} f}}_\infty \\
    &\le  e^{-\Lambda_{\kappa^+_L} f_L} \norm{\frac{(\Lambda_{\kappa^+_L} f_L)^k}{k !} - \frac{(\Lambda_{\kappa} f)^k}{k !}}_\infty + \frac{(\Lambda_{\kappa} f)^k}{k !} \norm{e^{-\Lambda_{\kappa^+_L} f_L} - e^{-\Lambda_{\kappa} f}}_\infty\\
    &\le \frac{(1+M_{\kappa})M_{\kappa}^{k-1}}{(k-1)!}\norm{\Lambda_{\kappa^+_L} f_L - \Lambda_{\kappa}f}_\infty\\
    &\rightarrow 0,
\end{align}
where the inequalities for $P_{\kappa}^k$ are constructed in the same fashion as in the proof of Lemma \ref{prop:psi_lipschitz}. Thus, the convergence of $\Psi_{\kappa}$ and its Fr\'echet derivative can be concluded as follows:
\begin{align}
\norm{\Psi_{\kappa_L^+} f_L - \Psi_{\kappa} f}_\infty
    &= \norm{\sum_{k=0}^\infty  \eta_k(\cdot) \sum_{k'=0}^{k-1}(P_{\kappa}^{k'}f - P_{\kappa^+_L}^{k'}f_L)}_\infty \\
    &\le \sum_{k=0}^\infty \sum_{k'=0}^{k-1} \norm{P_{\kappa}^{k'}f - P_{\kappa^+_L}^{k'}f_L}_\infty \\
    &\le (1+M_{\kappa})M_{\kappa} e^{M_{\kappa}}\norm{\Lambda_{\kappa^+_L} f_L - \Lambda_{\kappa}f}_\infty\\
    &\rightarrow 0, \\
\norm{D \Psi_{\kappa_L^+} f_L [h] - D\Psi_{\kappa} f [h]}_\infty 
    &= \norm{\Lambda_{\kappa^+_L} [h] V[f_L]  - \Lambda_{\kappa} [h] V[f]}_\infty\\
    &\le  V[f_L]\norm{\Lambda_{\kappa^+_L}[h] - \Lambda_{\kappa}[h]}_\infty + \Lambda_{\kappa} [h] \norm{V[f_L] - V[f]}_\infty\\
    &=  V[f_L]\norm{\Lambda_{\kappa^+_L}[h] - \Lambda_{\kappa}[h]}_\infty + \Lambda_{\kappa} [h] \norm{\sum_{k=0}^\infty P_{\kappa}^k f_L - P_{\kappa}^k f}_\infty\\
    &\le V[f_L]\norm{\Lambda_{\kappa^+_L} f_L - \Lambda_{\kappa}f}_\infty + \Lambda_{\kappa} [h] \norm{\Psi_{\kappa_L^+} f_L - \Psi_{\kappa} f}_\infty\\
    &\rightarrow 0.
\end{align}
\end{proof}

\begin{proof}[Proof of Lemma~\ref{lem:fp_conv}]
We carry out the proof for the fixed points $\{\hat{f}^-_L\}_{L \ge 1}$ but the same arguments hold true for $\{\hat{f}^+_L\}_{L \ge 1}$ as well. We first show the existence of least fixed points $\hat{f}_L^-$ of $\Psi_{\kappa^-_L}$. Let 
$$\mathcal{H}^L :=\left \{f: \cS \ni s \mapsto \sum_{l \in [L]} c^l I_{\cS_L^l}(s), c^l \in [0,1] \right \} ,$$ 
be the set of step functions based on the partition $\{\cS_L^l\}_{l \in [L]}$. It is easy to verify that $\Psi_{\kappa^-_L}$ actually maps functions from $\cF_b$ to $\mathcal{H}^L$. Thus, the fixed points, if any, are elements in $\mathcal{H}^L$. We now use a similar argument as in the proof of Lemma \ref{prop:fp_exist}. It is clear that $\mathcal{H}^L$ is partially ordered with the pointwise relation ''$\leq$''. For each subset of $\mathcal{H}^L$, by taking the infimum and the supremum of the constant steps $c^l$, one observes that each subset of $\mathcal{H}^L$ has an infimum and supremum defined, which is an element of $\mathcal{H}^L$. Thus, $(\mathcal{H}^L, \le)$ forms a complete lattice. Since the $\Psi_{\kappa^\pm_L}$ are order-preserving, by the Knaster-Tarski Theorem their least fixed point $\hat{f}_L^\pm$ exists in $\mathcal{H}^L$.

 For the convergence of the fixed points, we will first show that $\{\hat{f}^-_L\}_{L \ge 1}$ is an increasing sequence. Note that for $\kappa_1 \le \kappa_2$ it holds that $\Psi_{\kappa_1} \le \Psi_{\kappa_2}$ point wise, so it follows that $\{\Psi_{\kappa^-_L}\}_{L \ge 1}$ is increasing for the increasing kernel sequence $\{\kappa^-_L\}_{L \ge 1}$. 

We now show that for a function $g_{L+1} \in \mathcal{H}^{L+1}$ which is a fixed point of $\Psi_{\kappa^-_{L+1}}$ it holds that $g_{L+1} \geq \hat{f}_L^-$. Recall that $\{\cS_{L+1}^l\}_{l \in [L+1]}$ is a refinement of $\{\cS_L^l\}_{l \in [L]}$, so for each subset $\cS_L^l$ with $l \in [L]$, we can find an index set $J(l)$ such that  $\cup_{j \in J(l)} \cS_{L+1}^{j} = \cS_L^l$. Recall from above that the fixed point of $\Psi_{\kappa^-_L}$ is a step function and can therefore be expressed as $\hat{f}^-_L = \sum_{l \in [L]} c_f^l I_{\cS^l}(s)$ for constants $\{c_f^l\}_{l \in [L]}$. 
Similarly, 
$$ g_{L+1}(s) = \sum_{l \in [L+1]} c_g^l I_{\cS_{L+1}^l}(s),$$ with constants $\{c_g^l\}_{l \in [L+1]}$. Using the disjoint index sets $\{J(l) \}_{l\in L}$ this sum can be rewritten as
$$g_{L+1}(s) = \sum_{l \in [L]} \sum_{j \in J(l)} c_g^{j} I_{\cS_{L+1}^{j}}(s).$$
Define now the function $\underline{g} \in \mathcal{H}^{L}$ by $$\underline{g} := \sum_{l \in [L]} \min\{c_f^l, \min_{j \in J(l)} \{c_g^{ j}\}\} I_{\cS_L^{l}}(s).$$ 
It follows that $\underline{g} \le g_{L+1}$ and $\underline{g} \le \hat{f}^-_L$. 

Now let us assume that there exists $s^*\in \cS$ such that $g_{L+1}(s^*) < \hat{f}_L^- (s^*)$. It follows then that $\underline{g} < \hat{f}_L^-$ and because $\hat{f}_L^-$ is the first fixed point of $\Psi_{\kappa^-_L}$, a similar reasoning as used in \cite[Lem 3.5]{BichuchDetering2021}) allows to conclude that there must exist an $j^*\in [L]$ such that 
\begin{equation}
\label{pre:fixed:point:pos}
 \Psi_{\kappa^-_L} \underline{g}(s) - \underline{g}(s) >0   
\end{equation}
for $s\in \cS_L^{j^*}$. Note that not necessarily $s^*\in \cS_L^{j^*}$. In fact \eqref{pre:fixed:point:pos} implies that $\underline{g}(s) < \hat{f}_L^-(s)$ for $s\in \cS_L^{j^*}$. To see this, let $I$ be the (possibly empty) set of indices $i\in [L+1]$ such that $g_{L+1}(s) \geq \hat{f}_L^-(s)$ for $s \in \cS_L^i$. Let $\cS_L^I = \cup_{i \in I}\cS_L^i$. For all $s \in \cS_L^I$, it follows that $\underline{g}(s) = \hat{f}^-_L(s) \leq g_{L+1}(s)$ by construction of $\underline{g}$. It further holds that $\Psi_{\kappa^-_{L+1}}\underline{g}(s) - \underline{g}(s) \le 0$ as otherwise increasing the value of $\underline{g}(s)$ to $\hat{f}_L^-(s)$ on $\cS \backslash \cS_L^I$ leads to $\Psi_{\kappa^-_{L+1}}\hat{f}_L^-(s) -\hat{f}_L^-(s) > 0$ on $\cS_L^I$ by the monotonicity properties of $\Psi_{\kappa^-_{L+1}}$ (see Lemma \ref{prop:psi_mono}), which contradicts to that $\hat{f}_L^-$ is the fixed. Therefore $\underline{g}(s) \le g_{L+1}(s) < \hat{f}_L^-(s)$ for $s\in \cS_L^{j^*}$.

Now we can find an index $i^* \in J(j^*)$ such that $\underline{g}(s) = g_{L+1}(s)$ for $s\in \cS_{L+1}^{i^*}$. Combining this with the fact that $\{\Psi_{\kappa^-_L}\}_{L \ge 1}$ is increasing, we obtain
\begin{align}\label{eq:psi_lplus1}
\Psi_{\kappa^-_{L+1}}g_{L+1}(s) - g_{L+1}(s) \ge \Psi_{\kappa^-_L} \underline{g}(s) - \underline{g}(s) >0 \text{ for } s\in \cS_L^{j^*}.
\end{align}
Hence, $g_{L+1}$ can only be a fixed point of $\Psi_{\kappa^-_{L+1}}$ if $\hat{f}^-_L\leq  g_{L+1}$. 

We conclude that $\{\hat{f}^-_L\}_{L \ge 1}$ is a bounded and increasing sequence. By the monotone convergence theorem, for all $s\in\cS$ there exists the pointwise limit $\hat{f}^-(s) := \lim_{L\rightarrow \infty}\hat{f}^-_L(s)$. Then by Lemma \ref{lem:converge}, we have that $\Psi_{\kappa_L^\pm} \hat{f}_L^- \rightarrow \Psi_{\kappa} \hat{f}^-$  uniformly. 

Collecting what we have shown so far we get that 
$$\hat{f}^- = \lim_{L \rightarrow \infty} \hat{f}_L^- =  \lim_{L \rightarrow \infty}\Psi_{\kappa_L^\pm} \hat{f}_L^- = \Psi_{\kappa} \hat{f}^-$$ and hence $\hat{f}^-$ is a fixed point of $\Psi_{\kappa}$. Moreover, $\hat{f}^-$ is indeed the least fixed point, $\hat{f}^- \equiv \hat{f}$. It follows from the fact that the sequence $\{\hat{f}_L^-\}$ is bounded by $\hat{f}$. To see this, assume there exists an $L^*$ such that $\hat{f}_{L^*}^- > \hat{f}$. Then it follows in particular that $\hat{f}_{L}^- \geq \hat{f}_{L^*}^-  > \hat{f}$ for $L\geq L^*$. Let $g = \hat{f} + \varepsilon h < \hat{f}_{L^*}^-$ for a small enough $\varepsilon$. From the condition on the Fr\'echet derivative of $\Psi_{\kappa}$ it follows that $\Psi_{\kappa}g - g < 0$. As $\kappa^-_{L} \le \kappa$, we obtain that $\Psi_{\kappa^-_{L}}g - g < 0$. Approximating $g$ with step functions $g_{L}$ on $\{\cS^l_L \}_{l\in L}$, it follows for $L$ large that also $\Psi_{\kappa^-_{L}}g_L - g_L < 0$. This implies that $\hat{f}_{L}^-$ cannot be the least fixed point of $\Psi_{\kappa^-_{L}}$ (compare again \cite[Lem 3.5]{BichuchDetering2021}).

Recalling convergence of the operator, we see that $\lim_{L \rightarrow \infty} \hat{f}_L^- =  \lim_{L \rightarrow \infty}\Psi_{\kappa_L^-} \hat{f}_L^- = \Psi_{\kappa} \hat{f} = \hat{f}$, and we obtain the uniform convergence of the least fixed points. Finally, It follows directly that $\int_{\cS}\hat{f}^-_L \dd \mu \rightarrow \int_{\cS} \hat{f} \dd \mu$ and again by Lemma \ref{lem:converge} that $D \Psi_{\kappa^-_L} \hat{f}_L^- \rightarrow  D \Psi_{\kappa}\hat{f}$.
\end{proof}
\begin{proof}[Proof of Lemma \ref{prop:no_fix}]
We use contradiction. Assume there exists a fixed point $g$, $\Psi [g] = g$, such that $g<a_0 h$ holds in at least a nonempty subset of $\cS$. Consider $\hat{a} = \sup_{a\in[0, a_0]}\{a h(s) \le g(s) \ \forall s \in \cS \}$ and let $\hat{s} \in \cS$ be one point where $\hat{a} h(\hat{s}) = g(\hat{s})$. We obtain that $g(s) \ge \hat{a}h(s)$ for $s \in \cS\backslash \{\hat{s}\}$. Then by the order preserving property of $\Psi$ (Lemma \ref{prop:psi_mono}), $\Psi[g](\hat{s}) \ge \Psi[\hat{a}h](\hat{s})$. Given the condition $\Psi [a h] > a h$ for $a \in [0, a_0]$, we have $\Psi[g](\hat{s}) - g(\hat{s}) \ge \Psi[\hat{a}h](\hat{s}) - \hat{a}h(\hat{s}) > 0$. In other words, $\Psi [g] \neq g$ at the point $\hat{s}$, which is in contradiction with $g$ being a fixed point.
\end{proof}

\setcitestyle{numbers}
\bibliographystyle{chicago}
\bibliography{bib}

\end{document}